\documentclass[11pt]{amsart}
\usepackage{amsmath,amssymb,color,cite,colonequals,pdflscape,verbatim,mathrsfs}

\usepackage{tikz}
\usetikzlibrary{positioning}

\usepackage{tikz-cd}

\numberwithin{equation}{section}
\newtheorem{warning}[equation]{Warning}

\newtheorem{thm}[equation]{Theorem}
\newtheorem{prop}[equation]{Proposition}
\newtheorem{lemma}[equation]{Lemma}
\newtheorem{cor}[equation]{Corollary}

\theoremstyle{definition}
\newtheorem{rmk}[equation]{Remark}

\newtheorem{definition}[equation]{Definition}
\newtheorem{example}[equation]{Example}

\newtheorem{prop-def}{Proposition-Definition}

\newcommand{\bA}{\mathbb{A}}
\newcommand{\bP}{\mathbb{P}}

\newcommand{\pt}{\text{pt}}
\DeclareMathOperator{\CH}{CH}
\DeclareMathOperator{\opCH}{opCH}
\DeclareMathOperator{\Pic}{Pic}
\DeclareMathOperator{\Tor}{Tor}

\newcommand{\Z}{\mathbb{Z}}
\newcommand{\F}{\mathbb{F}}

\newcommand{\ind}{\text{ind}}
\newcommand{\irr}{\text{irr}}
\newcommand{\sep}{\text{sep}}
\newcommand{\non}{\text{non}}
\newcommand{\id}{\text{id}}
\newcommand{\M}{\mathcal{M}}
\newcommand{\OO}{\mathcal{O}}
\newcommand{\im}{\text{im}}

\newcommand{\divisor}{\text{div}}

\newcommand{\Gm}{\mathbb{G}_m}
\DeclareMathOperator{\charp}{char}

\DeclareMathOperator{\ord}{ord}

\DeclareMathOperator{\codim}{codim}

\DeclareMathOperator{\GL}{GL}
\DeclareMathOperator{\Aut}{Aut}

\DeclareMathOperator{\Spec}{Spec}

\makeatletter
\def\udots{\mathinner{\mkern1mu\raise\p@
\vbox{\kern7\p@\hbox{.}}\mkern2mu
\raise4\p@\hbox{.}\mkern2mu\raise7\p@\hbox{.}\mkern1mu}}
\makeatother

\renewcommand{\bar}[1]{#1\llap{$\overline{\phantom{\rm#1}}$}}

\usepackage[colorlinks,pagebackref,pdftex, bookmarks=false]{hyperref}
\hypersetup{citecolor=blue}

\newif\ifpdf
\ifx\pdfoutput\undefined
\pdffalse
\else
\pdfoutput=1
\pdftrue
\fi

\ifpdf
\usepackage[colorlinks,pagebackref,pdftex, bookmarks=false]{hyperref}
\else
\usepackage[colorlinks,pagebackref]{hyperref}
\fi

\begin{document}

\title{The first higher Chow groups of $\mathcal{M}_{1,n}$ for $n\leq 4$}
\author{William C. Newman}
\begin{abstract}
    For $n\leq 4$, we compute the indecomposable higher Chow groups $\overline{\operatorname{CH}}(\mathcal{M}_{1,n},1)$ with integer coefficients. As an application, we give new proofs of presentations of the integral Chow rings $\operatorname{CH}(\overline{\mathcal{M}}_{1,n})$ for $n\leq 4$ and determine formulas for the classes of boundary strata in these rings.
\end{abstract}
\maketitle
\tableofcontents
\section{Introduction} 
Much is known about the rational coefficient Chow groups of the moduli space of smooth/stable curves with genus $g$ and $n$ marked points for $g$ and $n$ small \cite{CL1}. Significant progress has also been made towards understanding these groups with integral coefficients \cite{Bish1, DLPV, BDL2}. Much less is known about the higher Chow groups of these spaces. 

The higher Chow groups of a variety $X$, $\CH(X,j)$, were introduced by Bloch in \cite{Bloch}. They are an extension of the usual Chow groups, with the same set of functorial properties. One cannot expect to compute the higher Chow groups of a given variety for all $j$, because it is not even known how to describe the higher Chow groups of $\Spec(k)$ for all $j$. Turning to the case $j=1$, the groups $\CH(\M_{0,n},1)$ were computed in \cite{BS2} for all $n$. The primary goal of this paper is to study the groups $\CH(\M_{1,n},1)$, with integer coefficients. The groups $\CH(\bar\M_{g,n},1)$ are less interesting (see Remark~\ref{HigherBarMgnBad}). With rational coefficients, the groups $\CH(\M_{g,n},1)_{\mathbb Q}$ for small $g,n$ can be computed by a computer (see Remark~\ref{rational}).

Higher Chow groups are usually very large. For example, for any smooth variety $X$ over the field $k$, $\CH^1(X,1)$ contains $k^\times$. For this reason, we instead study a variant $\overline\CH(X,1)$, the indecomposable higher first Chow group. These are finitely generated in the cases we are interested in. In \cite{ELarson, Bish1, Bish2}, the notions of $\ell$-adic higher Chow groups were used for similar reasons. We show how indecomposable Chow groups relate to $\ell$-adic Chow groups in Proposition~\ref{elladic}.

Our computations of $\overline\CH(\M_{1,n},1)$ from Theorem~\ref{M11}, Theorem~\ref{M12}, Corollary~\ref{higherM13}, and Corollary~\ref{higherM14} are summarized by the following theorem:
\begin{thm}
    For $n\leq 4$, one has $\overline\CH^1(\M_{1,n},1)=0$, and for $i\geq 2$, we have
    \begin{align*}
        & \overline\CH^i(\M_{1,1},1)=0\\
        & \overline\CH^i(\M_{1,2},1)=\frac{\Z}{2\Z}\\
        & \overline\CH^i(\M_{1,3},1)=\left(\frac{\Z}{2\Z}\right)^2 \\
    & \overline\CH^i(\M_{1,n},1)=\begin{cases}
        \Z\oplus \left(\frac{\Z}{2\Z}\right)^5 & i=2\\
        \left(\frac{\Z}{2\Z}\right)^2 & i\geq 3. 
    \end{cases}
    \end{align*}
\end{thm}
This just describes the $\overline\CH(\M_{1,n},1)$ as graded Abelian groups. We are also able to understand the $\CH(\M_{1,n})$-module structure on these groups, as well as the pullback maps $\pi^*: \overline\CH(\M_{1,n},1)\to \overline\CH(\M_{1,n+1})$.

If $k$ is algebraically closed, Proposition~\ref{AlgebraicallyClosed} explains how to compute $\CH(\M_{1,n},1)$ from $\overline\CH(\M_{1,n},1)$. Remarks~\ref{totalhigherchowM11} and \ref{totalhigherchowM12} give some information on $\CH(\M_{1,n},1)$ over an arbitrary field.

Knowledge of $\overline\CH(\M_{g,n},1)$ can be used to compute $\CH(\bar\M_{g,n})$ by using the following exact sequence
\[\overline\CH(\M_{g,n},1)\xrightarrow{\partial_1} \CH(\partial\bar\M_{g,n})\to \CH(\bar\M_{g,n})\to \CH(\M_{g,n})\to 0.\]
Using this, we compute $\CH(\bar\M_{1,n})$ for $n\leq 4$. For $n=3,4$, we actually need at least part of the computation of $\CH(\bar\M_{1,n})$ to finish computing $\overline\CH(\M_{1,n},1)$. This is because our first results about $\overline\CH(\M_{1,n},1)$ Theorem~\ref{M13} and Theorem~\ref{M14} do not give the full set of relations between the generators. This use of first higher Chow groups to determine Chow rings of moduli spaces was inspired by \cite{BS2} and \cite{ELarson}. The Chow rings $\CH(\bar\M_{1,n})$ for $n\leq 4$ were previously computed in \cite{EG, DLPV, Bish2, BDL1} respectively. $\CH(\bar\M_{1,n})$ was further computed for $n=5,6$ in \cite{BDL2}. 

The following theorem combines the computations of the Chow rings of $\bar\M_{1,n}$ from Theorem~\ref{barM11}, Theorem~\ref{barM12}, Theorem~\ref{barM13}, and Theorem~\ref{barM14}.
\begin{thm}
    The Chow rings of $\bar\M_{1,n}$ are given by
    \begin{align*}
        &\CH(\bar\M_{1,1})=\frac{\Z[\lambda]}{(24\lambda^2)}\\
        &\CH(\bar\M_{1,2})=\frac{\Z[\lambda,\delta_\emptyset]}{(24\lambda^2,\lambda^2+\delta_\emptyset\lambda)}\\
        &\CH(\bar\M_{1,3})=\frac{\Z[\lambda,\delta_1,\delta_2,\delta_3,\delta_\emptyset]}{I_3}\\
        &\CH(\bar\M_{1,4})=\frac{\Z[\lambda, \{\delta_{ij}\}_{i\neq j\in [4]},\{\delta_i\}_{i\in [4]}, \delta_\emptyset]}{I_4}
    \end{align*}
    where $I_3$ and $I_4$ are as described in Theorem~\ref{barM13} and Theorem~\ref{barM14}, respectively.
\end{thm}
Over the course of computing these Chow rings, we also get formulas for the classes of boundary strata, which are given in the statements of Theorem~\ref{barM12}, Theorem~\ref{barM13}, and Theorem~\ref{barM14}. The class of boundary strata with a separating node can be computed as a pushforward using previously known boundary strata classes, so it suffices to only compute the classes of boundary strata with no separating nodes.

The techniques from this paper are used by the author in \cite{Me2} to compute the Chow rings of moduli spaces of genus $0$ curves where marked points are allowed to collide.

\subsection*{Acknowledgments} The author would like to thank Ishan Banerjee, Martin Bishop, Roy Joshua, Jake Huryn, Johannes Schmitt, Ramesh Sreekantan, Burt Totaro, and Hsian-Hua Tseng for helpful conversations. The author would also like to thank Rahul Pandharipande for the invitation to give a talk about this project in the ETH Algebraic Geometry and Moduli Seminar. Finally, the author thanks the organizers and the speakers of the 2023 AGNES summer school on Intersection Theory on Moduli Spaces. 

This material is based upon work supported by the National Science Foundation under Grant No. DMS-2231565.

\subsection*{Notation} 
\begin{itemize}
    \item $k$ is a field whose characteristic is not $2$ or $3$.
    \item Over such a field, the stacks $B\mu_2$ and $B(\Z/2\Z)$ are canonically isomorphic, so we identify them.
    \item All stacks $X$ are separated finite type quotient stacks over $k$.
    \item The set $\{1,2,\dots,n\}$ is denoted $[n]$.
    \item  For equidimensional $X$, we set $\CH^i(X,j):=\CH_{\dim X-i}(X,j)$ (i.e. this does not mean operational Chow).
    \item The symbol $\lambda$ denotes the first Chern class of the Hodge bundle on a stack with a map to $\bar\M_{1,1}$.
    \item For a ring $R$, we use the notation $R\langle a_1,\dots,a_n\rangle$ to denote the free $R$ module on $a_1,\dots,a_n$. Given an $R$-module with elements $f_1,\dots,f_r$, $\langle f_1,\dots,f_r\rangle$ denotes the $R$-submodule of $M$ generated by $f_1,\dots,f_r$. 
    \item For a $1$-pointed genus $1$ curve, $(C,p_1)$ we let $\oplus$ denote the corresponding group law on the points of $C$ with $p_1$ the additive identity, and let $[n]$ denote the multiplication by $n$ morphism for $n\in \Z$.
\end{itemize}


\section{Indecomposable Higher Chow Groups}
\subsection{Higher Chow Groups}
We recall some basic properties of higher Chow groups, most of which can be found in \cite{Bloch}. 
\begin{enumerate}
    \item The higher Chow groups $\CH(X,j)$ are the cohomology of a complex \[z(X,\cdot):=\dots\to z(X,i)\to z(X,i-1)\to\dots\to z(X,1)\to z(X,0)\to 0,\] 
    where $z(X,i)$ is a subgroup of the cycles on $X\times \Delta^i$. This complex decomposes as a direct sum of complexes $z_i(X,\cdot)$ using the dimension of the cycle, giving a direct sum decomposition 
    \[\CH(X,j)=\bigoplus_i \CH_i(X,j).\]
    \item Higher Chow groups have proper pushforwards, flat pullbacks, and arbitrary pullbacks for morphisms between smooth schemes. These functorialities interact with each other and affect the grading in the same way as usual Chow groups.  
    \item $\CH_i(X,0)=\CH_i(X)$, and the functorialities agree.
    \item If $X$ contains an open subscheme $j:U\hookrightarrow X$ with complement $\iota:Z\hookrightarrow X$, there is an exact sequence
    \[\dots\to \CH_i(Z,j)\xrightarrow{\iota_*} \CH_i(X,j)\xrightarrow{j^*}\CH_i(U,j)\xrightarrow{\partial_j}\CH_i(Z,j-1)\to\dots\]
    and the connecting homomorphism is natural. This was first proven for quasiprojective schemes in \cite{Bloch2} and for general schemes in \cite{Levine}.
    \item There is a cross product
    \[\times: \CH_i(X,j)\times\CH_k(Y,\ell)\to \CH_{i+k}(X\times Y,j+\ell)\]
    For smooth $X$, setting $X=Y$ and pulling back along the diagonal gives the structure of a bigraded ring on $\CH(X,*)=\bigoplus_{i,j}\CH^i(X,j)$. In particular, $\CH(X,j)$ has the structure of a $\CH(X)$-module. 
\end{enumerate}
More information on higher Chow groups can be found in the appendix.

In \cite{EG}, Edidin and Graham extended (higher) Chow groups to quotient stacks $[X/G]$ using equivariant (higher) Chow groups $\CH^G$. This is done by setting 
\[\CH_i([X/G],j):=\CH^G_{i+\dim(V)-\dim(G)}(X,j)=\CH_{i+\dim(V)}((X\times U)/G,j),\]
where $U$ is an open subset of a $G$-representation, $E$, on which $G$ acts freely, such that $\codim(V\setminus U)$ is sufficiently large. All of the above properties of higher Chow groups extend to stacks, replacing morphisms by representable morphisms. Because the Chow groups of a stack are defined as Chow groups of a scheme, proving a statement about the Chow groups of stacks usually reduces to proving the statement for schemes.

We can use higher Chow groups to prove things about ordinary Chow groups, like the following lemma.
\begin{lemma}\label{SplitChow}
    Suppose we have a proper map of quotient stacks $g:X\to Y$, with a closed substack $\iota:Z\hookrightarrow Y$ and open complement $U\subseteq Y$, such that $g^{-1}(U)\xrightarrow{g} U$ is an isomorphism. Then, 

    \begin{center}
        \begin{tikzcd}
        \CH(g^{-1}Z)\arrow[r,"\iota'_*"]\arrow[d,"g_*"] & \CH(X)\arrow[d,"g_*"]\\
        \CH(Z)\arrow[r,"\iota_*"] & \CH(Y)
    \end{tikzcd}
    \end{center}
    is a pushout diagram, where $\iota': g^{-1}Z\hookrightarrow X$ is the inclusion.
\end{lemma}
\begin{proof}
    We get the diagram
 \begin{center}
        \begin{tikzcd}
            \CH(g^{-1}U,1) \arrow[r,"\partial_1"]\arrow[d,"\sim"] & \CH(g^{-1}Z)\arrow[r,"\iota'_*"]\arrow[d,"g_*"] & \CH(X)\arrow[r]\arrow[d,"g_*"] &\CH(g^{-1}U)\arrow[r]\arrow[d,"\sim"] & 0\\
            \CH(U,1) \arrow[r,"\partial_1"] & \CH(Z)\arrow[r,"\iota_*"] & \CH(Y)\arrow[r]& \CH(U)\arrow[r] & 0
        \end{tikzcd}
    \end{center}
    with exact rows. Now the lemma follows from a diagram chase.
\end{proof}

The only higher Chow groups we will use in this paper are $\CH(X,1)$. The following proposition explains how we work with classes in this group.
\begin{prop}\label{formalsum}
Given a scheme $X$, a closed irreducible subvariety $V\subseteq X$ of dimension $i$, and a rational function $f\in K(V)^\times$, one can associate an element $[V,f]\in z_{i-1}(X,1)$, satisfying the following properties:
\begin{enumerate}
\item If $\iota:Z\hookrightarrow X$ is a closed embedding, and $V\subseteq Z$ with rational function $f$ on $V$, we have
\[\iota_*([V,f])=[V,f].\]
\item Given $\pi:X\to Y$ flat of constant relative dimension, if $\pi^{-1}(V)$ is integral, we have 
\[\pi^*([V,f])=[\pi^{-1}(V),\pi^*(f)].\]
\item If $X$ is irreducible, there is a homomorphism
\[\OO_X(X)^\times\to \CH^1(X,1)\]
\[f\mapsto [X,f]\]
which is an isomorphism if $X$ is smooth.
\item If $X$ is irreducible, $U\subseteq X$ is open with complement $Z$, and $f\in \OO_U(U)^\times$, then the map
\[\partial_1: \CH^1(U,1)\to \CH(Z)\]
sends $[U,f]$ to $\divisor_X(f)$.
\end{enumerate}
\end{prop}
The author does not know of an elementary, complete reference for the contents of this proposition; the definition of $[V,f]$ and proofs of these properties can be found in the appendix. 

Now, suppose we have a smooth variety $X$ and closed $Z\subseteq X$ of pure codimension $1$ with complement $U$, and we wish to compute the boundary map
\[\partial_1: \OO_U(U)^\times \to \CH^0(Z).\]
The above proposition says that the map sends $f$ to $\divisor_X(f)$. To compute this, one must know $\ord_W(f)$ for components $W\subseteq Z$. To do this, we will use the following proposition.

\begin{prop}\label{transverse}
    Suppose we have a smooth variety $X$ and closed irreducible subvariety $Z\subseteq X$ of codimension $1$ with complement $U$. Additionally, suppose we have a curve $C\subseteq X$ with $C$ intersecting $Z$ transversely at a point $p\in C$. Then, for $f\in \OO_U(U)^\times$, we have 
    \[\ord_Z(f)=\ord_p(f|_C).\]
\end{prop}
\begin{proof}
    Because $C$ intersects $Z$ transversely and $f\in \OO_U(U)^\times$, $f|_C$ makes sense.

Because $X$ is smooth, we can shrink $X$ and assume that $Z\subseteq X$ is cut out by $g$. We can write $f=ug^n$ for $u\in \OO_X(X)^\times$. Because $C$ intersects $Z$ transversely at $p$, we have that $g|_C$ makes sense and is a uniformizer for $p\in C$. Now  we have
\[\ord_Z(f)=\ord_Z(ug^n)=n=\ord_p(g|_C^n)=\ord_p(u|_C\cdot g|_C^n)=\ord_p(f|_C),\]
because $u|_C$ is a unit on $C$.
\end{proof}

Now, for a quotient stack $[X/G]$, given a closed irreducible substack $[V/G]\subseteq [X/G]$ of dimension $i$ and $f\in K([V/G])^\times$, we can define $[[V/G],f]\in \CH_{i-1}([X/G],1)=\CH_{i-1}(X\times U/G,1)$ as $[V\times U/G,f]$, where $E$ is a representation with free locus $U$ and $\codim(V\setminus U)$ is sufficiently large. This class is independent of the representation. The analogues of the first two parts of Proposition~\ref{formalsum} follow from the definition; we prove the analogues of the third and fourth parts next.

\begin{lemma}\label{units}
\mbox{}
\begin{enumerate}
\item[($3'$)]  If $[X/G]$ is irreducible, there is a homomorphism
\[(\OO_X(X)^\times)^G\to \CH^1(X,1)\]
\[f\mapsto [[X/G],f]\]
which is an isomorphism if $X$ is smooth.
\item[($4'$)]  If $[X/G]$ is irreducible, $U\subseteq X$ is open with complement $Z$ of pure codimension $1$, and $f\in \OO_U(U)^\times$, then the map
\[\partial_1: \CH^1([U/G],1)\to \CH^0([Z])\]
sends $[[U/G],f]$ to $\divisor_X(f)\in \CH^0([Z/G])=\CH^0(Z)^G$.
\end{enumerate}
\end{lemma}
\begin{proof}
Take a representation of $G$, $E$, with $G$ acting freely on the open subset $U\subseteq E$ and $\codim(E\setminus U)$ small. Note $\OO_{X\times U}(X\times U)^G=\OO_{X\times U/G}(X\times U/G)$. Then we have
\[\OO_X(X)^G\to\OO_{X\times U/G}(X\times U/G)\xrightarrow{\sim} \CH^1(X\times U/G,1)=\CH^1([X/G],1),\]
where the second arrow is an isomorphism because of Proposition~\ref{formalsum}(3). Because $\codim(E\setminus U)$ is small and $X$ is normal, we have that global functions of $X\times U$ are the same as that of $X\times E$. And because units on $X$ are the same as units on $X\times \bA^1$ we can conclude that $\OO_X(X)^G\to\OO_{X\times U/G}(X\times U/G)$ is an isomorphism. This gives $(3')$.

To prove $(4')$, consider the commutative diagram
\begin{center}
    \begin{tikzcd}
        \CH^1(U,1)\arrow[r,"\partial_1"] & \CH^0(Z)\\
        \CH^1([U/G],1)\arrow[u,"\pi^*"]\arrow[r,"\partial_1"] & \CH^0([Z/G])\arrow[u,"\pi^*"],
    \end{tikzcd}
\end{center}
where $\pi$ is the quotient map.
In codimension $0$, equivariant Chow groups are the same as $G$-invariant Chow groups, hence $\CH^0([Z/G])=\CH^0(Z)^G$. This means the right $\pi^*$ is injective. Then, using that
\[\partial_1(\pi^*([[U/G],f]))=\partial_1([U,f])=\divisor_X(f)\]
by Proposition~\ref{formalsum}(3), we get
\[\partial_1([[U/G],f])=\divisor_X(f).\qedhere\]
\end{proof}

\begin{lemma}\label{Module}
Suppose $S$ is a smooth, irreducible stack. Then for any representable $f: X\to S$, we get a $\CH(S)$-module structure on $\CH(X,j)$, compatible with pushforwards, pullbacks, and connecting homomorphisms in the localization exact sequence. 
\end{lemma}
\begin{proof}
Note it suffices to prove this for $S,X$ schemes. Define 
\[\opCH(T):=\lim_{\substack{T\to V\\ V \text{ smooth}}} \CH(V).\] 
These are the ``Chow cohomology groups'' developed in \cite{Fulton2}, as opposed to the groups in \cite[Chapter 17]{Fulton}. One has a homomorphism
\[\cap : \opCH(T)\otimes \CH(T,j)\to \CH(T,j)\]
by consider the Gysin pullback associated to $T\to T\times S$ \cite[Corollary 5.6]{Bloch}. Thus $\CH(T,j)$ is a module over $\opCH(T)$.

Because $S$ is smooth and irreducible, the map 
\[\cap [S]: \opCH(S)\to \CH(S)\] 
is an isomorphism. We get a map $f^*: \opCH(S)\to \opCH(X)$, and hence a $\CH(S)$-module structure on $\CH(X,j)$. 

The fact that pushforward is $\CH(S)$-linear follows from the projection formula, and the fact that pullback is $\CH(S)$-linear is functoriality of pullback on $\opCH^*$. As the connecting homomorphism is constructed out of pushforwards and pullbacks, it is also $\CH(S)$-linear.
\end{proof}

\subsection{Indecomposable Higher Chow}
Now, we define the indecomposable first higher Chow groups, $\overline\CH(X,1)$. These will have essentially all of the same properties as the first higher Chow groups, with the bonus of being finitely generated and independent of base field in the cases that we care about.

By Proposition~\ref{formalsum}(3), we have 
\[\CH_{-1}(k,1)=\CH^1(k,1)=k^\times.\]
\begin{definition}
    The decomposable first higher Chow groups of a quotient stack $X$, $\CH^{\text{dec}}_i(X,1)$ is defined to be the image of the cross product
    \[\CH_{i+1}(X)\otimes k^\times=\CH_{i+1}(X)\otimes \CH_{-1}(k,1)\to \CH_i(X\times_k k,1)\xrightarrow{\sim} \CH_i(X,1).\] We then define the indecomposable first higher Chow group of a quotient stack $X$ to be 
    \[\overline\CH_i(X,1):=\CH_i(X,1)/\CH^{\text{dec}}_i(X,1).\]
\end{definition}

If $X$ is smooth with structure morphism $f:X\to \Spec(k)$, note that for $\alpha\in \CH_{i+1}(X)$ and $\beta\in \CH_{-1}(k,1)$, under the identification of $X\times_k k$ with $X$,we have $\alpha\times \beta=\pi_1^*\alpha\cup \pi_2^*\beta=\alpha\cup f^*\beta$. 

\begin{warning}
There are different, nonequivalent notions of indecomposable higher Chow groups. Our definition is the same as those appearing in \cite{CF} and \cite{BS2}. Notions of indecomposable cycles were originally defined in order to give some notion of ``nontriviality'' of a higher cycle, in the sense that it is not decomposing into other cycles in some specified way.
\end{warning}

The basic operations for these groups remain unchanged:
\begin{itemize}
    \item For a proper map $f:X\to Y$, compatibility between the product $\times$ and proper pushforward $f_*$ implies that $f_*(\CH_i(X,1)_{\text{dec}})\subseteq \CH_i(Y,1)_{\text{dec}}$. Hence, $f_*$ induces a map $f_*:\overline\CH_i(X,1)\to\overline\CH_i(Y,1)$.
    \item For a flat map of relative dimension $d$, compatibility between the product and $f^*$ implies that $f^*(\CH_i(Y,1))_{\text{dec}}\subseteq \CH_{i+d}(X,1)_{\text{dec}}$. Hence, $f^*$ induces a map $f^*:\overline\CH_i(Y,1)\to\overline\CH_{i+d}(X,1)$. 
    \item The product map 
    \[\CH(X)\otimes \CH(X,1)\to \CH(X,1)\]
    sends decomposable cycles to decomposable cycles by skew commutativity of the product. Hence, $\overline \CH(X,1)$ is naturally a $\CH(X)$-module. 
\end{itemize}

Moreover, we still have the last part of the localization exact sequence: 
\begin{lemma}
    If $U\subseteq X$ is open with complement $Z$, then $\partial:\CH(U,1)\to \CH(Z)$ is zero on $\CH(U,1)_{\text{dec}}$, and the induced sequence
        \[\overline\CH(Z,1)\to \overline\CH(X,1)\to\overline\CH(U,1)\to \CH(Z)\to\CH(X)\to\CH(U)\to 0\]
    is exact. 
\end{lemma}
\begin{proof}
    We need only check exactness at the first four terms in the sequence, as the rest of the sequence is unchanged from the standard localization exact sequence. 
    Consider the diagram
    \begin{center}
    \begin{tikzcd}[sep=1em, font=\small]
        k^\times \otimes\CH(Z)\arrow[r]\arrow[d] & k^\times \otimes\CH(X)\arrow[r]\arrow[d] & k^\times \otimes\CH(U)\arrow[r]\arrow[d] & 0\arrow[r]\arrow[d] & 0\arrow[d]\\
       \CH(Z,1)\arrow[r] & \CH(X,1)\arrow[r,"j^*"] & \CH(U,1)\arrow[r,"\partial"] & \CH(Z)\arrow[r] & \CH_i(X)
    \end{tikzcd}
    \end{center}
    Note that both rows are exact sequences, because the localization exact sequence is exact and because tensor is right-exact. As noted above, compatibility between $\times$ and pullbacks/pushforwards gives that the first two squares commute. To see that the third square commutes, we need that $\partial(a\times \alpha)=0\in \CH_i(Z)$ for $a\in \CH_{-1}(k,1)$ and $\alpha\in \CH_i(U)$. This is true because we can lift $\alpha$ to $\widetilde{\alpha}\in \CH_i(X)$, and then
    \[\partial(a\times \alpha)=\partial(a\times j^*(\widehat{\alpha}))=\partial(j^*(a\times \widetilde{\alpha}))=0.\]
    Thus, this diagram of exact sequences commutes. The sequence we want to be exact is precisely the cokernel complex
    \[\overline\CH(Z,1)\to \overline\CH(X,1)\to\overline\CH(U,1)\xrightarrow{\partial} \CH(Z)\to\CH(X).\]
    This is a routine diagram chase.
\end{proof}
With all of these properties, we obtain direct analogues of the results from the previous subsection for indecomposable higher Chow groups. 

Under certain circumstances, we have that the indecomposable first higher Chow groups aren't too far off from the complete first higher Chow groups.
\begin{prop}\label{AlgebraicallyClosed}
    Suppose $k$ is algebraically closed, and $X$ is a smooth irreducible quotient stack over $k$ with $\CH^i(X)$ torsion for $i>0$. Then 
    \[\CH(X,1)\cong k^\times \oplus \overline \CH(X,1).\]
\end{prop}
\begin{proof}
    By definition, we have an exact sequence
    \[k^\times \otimes \CH(X)\to \CH(X,1)\to \overline\CH(X,1)\to 0\]
    of graded abelian groups. Because $k$ is algebraically closed, $k^\times$ is abelian. Thus, in degree $i>1$, $k^\times\otimes \CH^{i-1}(X)=0$ because $\CH^{i-1}(X)$ is torsion. Thus, $\CH^i(X,1)=\overline\CH^i(X,1)$. In degree $1$, $\CH^0(X)=\Z$, so we have the exact sequence
    \[k^\times \to \CH^1(X,1)\to \overline\CH^1(X,1)\to 0.\]
    Picking a point $\Spec(k)\to X$, we can consider the pullback
    \[k^\times\CH^1(k,1)\to \CH^1(X,1)\to \CH^1(k,1).\]
    Because the composition is the identity, the first map must be injective, so our above exact sequence in degree $1$ is exact on the left. Because $k^\times$ is divisible, this exact sequence splits.
\end{proof}
We have $\CH^i(\M_{1,n})$ is torsion for $n\leq 4$ by Theorems~\ref{M11},\ref{M12},\ref{M13},\ref{M14} or \cite{Bish1}, and so our computations of $\overline\CH(\M_{1,n},1)$ give computations of $\CH(\M_{1,n},1)$ over an algebraically closed field.

\subsection{Universal Separable Homeomorphisms}
We now introduce the notion of a universal separable homeomorphism. 
\begin{definition}
    A map of stacks $f:Z'\to Z$ is a universal separable homeomorphism if it is a representable finite surjective map such that, for all fields $K$ and maps $\Spec(K)\to Z$, the fiber product $Z'_K$ consists of a single point whose residue field is equal to $K$.
\end{definition}
It is clear from the definition that universal separable homeomorphisms are preserved by base changes. We also have
\begin{prop}\label{GLnQuotients}
    Suppose $Z'\to Z$ is a $\GL_n$-equivariant universal separable homeomorphism. Then the induced map $ [Z'/\GL_n]\to [Z/\GL_n]$ is also a universal separable homeomorphism. 
\end{prop}
\begin{proof}
    By Lemmas 35.23.7, 35.23.23, and 35.23.8 in \cite[\href{https://stacks.math.columbia.edu/tag/02YJ}{Tag 02YJ}]{stacks-project}, we have that  the induced map $ [Z'/\GL_n]\to [Z/\GL_n]$  is surjective, finite, and universally injective, respectively. Then, taking a map $\Spec(K)\to [Z/\GL_n]$, we can write
    \[\Spec(A):=[Z'/\GL_n]\times_{[Z/\GL_n]} \Spec(K)\]
    for some local Artinian ring that is finite dimensional over $K$. We want to show that the residue field of $A$ is $K$.
    
    The map $\Spec(K)\to [Z/\GL_n]$ is induced by a $\GL_n$ equivariant map from a $\GL_n$-torsor over $\Spec(K)$ to $Z$. Such torsors are trivial by Hilbert's Theorem $90$, so the torsor has a section. Then $\Spec(K)\to [Z/\GL_n]$ factors through $Z$. Thus, $\Spec(A)$ is also $Z'\times_Z \Spec(K)$, which has residue field $K$ because $f$ is a universal separable homeomorphism.
\end{proof}

Here is an example we will use in the following section.
\begin{example}\label{cusp}
Consider the normalization of the node:
\[f:\Spec(k[t])\to \Spec(k[x,y]/(y^2-x^3))\]
\[t\mapsto (t^2,t^3).\]
This is a finite  surjection, being a normalization. Away from the cusp, $(0,0)$, this map is an isomorphism, hence for all $\Spec(K)$ mapping to points other than the cusp, the base change to $\Spec(K)$ is an isomorphism. Next, the point $(0,0)$ has a unique preimage, $0$, and the map on residue fields is an isomorphism, so $f$ is a universal separable homeomorphism.
\end{example}

\begin{rmk}
    As the name suggests, universal separable homeomorphisms are precisely the universal homeomorphisms which induce separable field extensions on residue fields. One direction is clear: universal separable homeomorphisms are homeomorphisms, being closed continuous bijections, and are universal homeomorphisms because they are closed under base change. The other direction is true because universal homeomorphisms are exactly the finite surjective morphisms so that every point in the target has a unique preimage and the induced map on residue fields are purely inseparable (see \cite[\href{https://stacks.math.columbia.edu/tag/04DF}{Tag 04DF}]{stacks-project} and \cite[\href{https://stacks.math.columbia.edu/tag/01S2}{Tag 01S2}]{stacks-project}).
\end{rmk}

We compare the notion of universal separable homeomorphisms to the notion of Chow envelopes in the literature. A map of stacks $f:Z'\to Z$ is a Chow envelope if it is a representable, proper, and if for all integral substacks $V\subseteq Z'$ there exists some integral substack $W\subseteq Z$ mapping to $V$ birationally. This implies that the pushforward $f_*$ is surjective on integral Chow groups. This has been used in the following way (e.g. \cite{EF}): Suppose you want to compute the Chow groups of $U\subseteq X$ and you know the Chow groups of $X$. Setting $Z:=X\setminus U$, the localization exact sequence gives
\[\CH(Z)\xrightarrow{\iota_*} \CH(X)\to \CH(U)\to 0,\]
and so it suffices to compute the image of $\iota_*$. The space $Z$ may be hard to work with (e.g. $Z$ may be singular) so it suffices to find some Chow envelope $f:Z'\to Z$ and compute the image of $(\iota\circ f)_*$, since this is just the image of $\iota_*$. 

In this paper, the above is true, but we also would like to compute the kernel of $\iota_*$, because we are also trying to compute the indecomposable higher Chow groups $\overline\CH(U,1)$, which fit into the sequence in the following way
\[\overline\CH(X,1)\to \overline\CH(U,1)\xrightarrow{\partial_1}\CH(Z)\xrightarrow{\iota_*}\CH(X).\]
Hence, we would like a map $f:Z'\to Z$ which induces an \emph{isomorphism} $f_*:\CH(Z')\to \CH(Z)$. By the proposition below, universal separable homeomorphisms $f$ do this. 

There do not always exist a universal separable homeomorphism $Z'\to Z$. For example, if $Z$ is a nodal curve, there is no universal separable homeomorphism from a smooth stack to $Z$. We will have to deal with such cases in the next section.

\begin{prop}\label{ushChow}
    Let $f:Z'\to Z$ be a universal separable homeomorphism. Then the induced maps $f_*:\CH(Z',j)\to \CH(Z,j)$ are isomorphisms. 
\end{prop}
\begin{proof}
By definition, $f$ is finite, so we can indeed consider the pushforward $f_*$. To show $f_*$ is an isomorphism, it suffices to assume that $Z'$ and $Z$ are schemes. The pushforward $f_*:\CH(Z',j)\to \CH(Z,j)$ is induced by the map $f_*:z(Z',*)\to z(Z,*)$ on complexes, given by the pushforward of cycles. Because $f:Z'\to Z$ is a universal separable homeomorphism, so is $f:Z'\times \Delta^n\to Z\times \Delta^n$. Therefore, for any closed, irreducible $V\subseteq Z'\times\Delta^n$, we have $f_*(V)=\deg(V/f_*(V))\cdot f(V)=f(V)$, as the map on residue fields is an isomorphism. And now, we see that $f_*:z_i(Z',j)\to z_i(Z,j)$ has an inverse, sending $W$ to $f^{-1}(W)$. Thus, the map of complexes is an isomorphism, so the induced maps $f_*:\CH(Z,j)\to \CH(Z',j)$ are isomorphisms.
\end{proof}

\subsection{The Motivic K\"unneth Property}

In \cite{Joshua}, Joshua constructs a complex   
\[z(X,\cdot) \otimes_{z(k,\cdot)}^L z(Y,\cdot)\]
for schemes $X,Y/k$, with a natural morphism to $z(X\times Y,\cdot)$. The grading by dimension on $z(X,\cdot)$ and $z(Y,\cdot)$ induces a grading on this tensor product. 

\begin{definition}
    We say that a scheme $X$ has the motivic K\"unneth property (MKP) if for all schemes $Y$, the map $z(X,\cdot) \otimes_{z(k,\cdot)}^L z(Y,\cdot)\to z(X\times Y,\cdot)$ is an isomorphism. 
\end{definition}
As discussed in \cite{Joshua}, this map being an isomorphism is equivalent to saying that a certain spectral sequence involving the higher Chow groups of $X$ and $Y$ converges to the higher Chow groups of $X\times Y$.

We also need the definition of the MKP for stacks.
\begin{definition}
    We say that a stack $[X/G]$ has the MKP if for all schemes $Y$, $(z(X\times U/G,\cdot) \otimes_{z(k,\cdot)}^L z(Y,\cdot))_i\to z_i(X\times Y,\cdot)$ is an isomorphism for all representations $E$ of $G$ with free locus $U$ such that the $\codim(E\setminus U)$ is sufficiently large.
\end{definition}

\begin{prop}\label{MKPprops}
    \mbox{}
    \begin{enumerate}
        \item Suppose $X$ is a stack with closed substack $Z$ and set $U:=X\setminus Z$. If two out of three of $X,U$, and $Z$ have the MKP, then so does the third.
        \item If $X$ and $X'$ have the MKP, so does $X\times X'$.
        \item If $X$ has the MKP and $X'\to X$ is a vector bundle, then $X'$ has the MKP.
	\item If $X\to Y$ is a universal separable homeomorphism, then $X$ has the MKP if and only if $Y$ does.
        \item $\bP^n$ and $B\Gm$ have the MKP. 
        \item If $X$ has the MKP, then $X$ has the Chow K\"unneth property, i.e. for all stacks $Y$, the map
        \[\CH(X)\otimes\CH(Y)\to \CH(X\times Y)\]
        is an isomorphism.
    \end{enumerate}
\end{prop}
\begin{proof}
    \mbox{}
    \begin{enumerate}
        \item By the localization theorem, we have an exact triangle
        \[z(Z,\cdot)\to z(X,\cdot)\to z(U,\cdot)\to z(Z,\cdot)[1]\]
        in the derived category of abelian groups. The functor $\otimes^L_{z(k,\cdot)} z(Y,\cdot))$ is exact, so the claim the follows from the commutative diagram
        \begin{center}
            \begin{tikzcd}[sep=1
            em, font=\small]
               z(Z,\cdot)\otimes_{z(k,\cdot)}^L z(Y,\cdot)\arrow[r]\arrow[d] & z(X,\cdot)\otimes_{z(k,\cdot)}^L z(Y,\cdot)\arrow[r]\arrow[d] &z(U,\cdot)\otimes_{z(k,\cdot)}^L z(Y,\cdot)\arrow[d]\\
               z(Z\times Y,\cdot)\arrow[r] & z(X\times Y,\cdot)\arrow[r] &z(U\times Y,\cdot)
            \end{tikzcd}
        \end{center}
        and properties of exact triangles in derived categories.
        \item Note
        \begin{align*}
            z(X\times X',\cdot)\otimes^L_{z(k,\cdot)} z(X',\cdot)&\cong z(X,\cdot)\otimes^L_{z(k,\cdot)} z(X',\cdot)\otimes^L_{z(k,\cdot)} z(Y,\cdot)\\
            &\cong z(X,\cdot)\otimes^L_{z(k,\cdot)} z(X'\times Y,\cdot) \\
            & \cong z(X\times X'\times Y,\cdot).
        \end{align*}
        \item This is true because the pullback $z(X,\cdot)\to z(X',\cdot)$ is a quasi-isomorphism.
        \item This is true because the pushforward $z(X,\cdot)\to z(Y,\cdot)$ is a quasi-isomorphism by Proposition~\ref{ushChow}.
        \item We have $z(k,\cdot) \otimes^L_{z(k,\cdot)} z(Y,\cdot)\to z(Y,\cdot)$ is an isomorphism, so $\Spec(k)$ has the MKP. By (3), so does $\bA^n$. Then using (1) and induction, we have $\bP^n$ has the MKP. Finally, $\bP^n$ are the finite dimensional approximations of $B\Gm$, so $B\Gm$ has the MKP.
        \item See \cite{Joshua}. \qedhere
    \end{enumerate}
\end{proof}
\begin{rmk}
The name ``motivic K\"unneth property'' for this property was given by Totaro in \cite{Totaro}. There, it is shown that $X$ has the MKP if and only if the mixed motive $M^c(X)$ is mixed Tate. This description also gives proofs of the above statements. However, to use this equivalence in characteristic $p>0$, one must be using coefficients with $p$ invertible. Since we are using integral coefficients in this paper, we cannot rely on this description.     
\end{rmk}

The MKP also gives us the following vanishing result for indecomposable higher Chow groups, which simplifies many computations.
\begin{prop}\label{IndZero}
    Suppose $X$ is a smooth proper Deligne-Mumford stack with the MKP. Then $\overline\CH(X,1)\otimes \Z[\frac{1}{N}]=0$, where $N$ is the least common multiple of the orders of the stabilizers of points of $X$.
\end{prop}
\begin{rmk}
    As the proof shows, one need only assume that the class of the diagonal is in the image of $\CH(X)\otimes\CH(X)\to \CH(X\times X)$, though by \cite[Theorem 4.1]{Totaro}, this is equivalent for a smooth proper variety to have the MKP.
\end{rmk}
\begin{proof}[Proof of Proposition~\ref{IndZero}]
(C.f. \cite[Theorem 5.2]{Joshua})
    Let $n=\dim(X)$. Because $X$ has the MKP, we know
    \[\CH(X)\otimes_{\mathbb Z} \CH(X)\to \CH(X\times X)\]
    is surjective. Then, we can write
    \[[\Delta]=\sum_i \alpha_i\times \beta_{n-i},\]
    for $\alpha_i,\beta_i\in \CH^i(X)$, where $\Delta\stackrel{\delta}\hookrightarrow X\times X$ is the diagonal.
    Because $X$ is smooth, $\alpha_i\times \beta_{n-i}=\pi_1^*(\alpha_i)\cup \pi_2^*(\beta_{n-i})$. 

    Because $X$ is proper, we can consider pushforwards along $\pi_i$ after we invert $N$. Note that for $\gamma\in \CH^i(X,1)$, we have
    \begin{align*}
        \pi_{1*}([\Delta]\cup \pi_2^*(\gamma))&=\pi_{1*}(\delta_*(1)\cup \pi_2^*(\gamma))\\
        &=\pi_{1*}(\delta_*(1\cup \delta^*\pi_2^*(\gamma)))\\
        &=(\pi_1\circ \delta)_*(\pi_2\circ \delta)^*(\gamma)\\
        &=\gamma,
    \end{align*}
    using the projection formula. But we also have
    \begin{align*}
        \pi_{1*}([\Delta]\cup \pi_2^*(\gamma))&=\sum_i \pi_{1*}(\pi_1^*(\alpha_i)\cup \pi_2^*(\beta_{n-i})\cup \pi_2^*(\gamma))\\
        =\sum_i \alpha_i\cup \pi_{1*}\pi_2^*(\beta_{n-i}\cup \gamma).
    \end{align*}
    Now, considering
    \begin{center}
        \begin{tikzcd}
            X\times X\arrow[r,"\pi_1"]\arrow[d,"\pi_2"] & X \arrow[d,"f"]\\
            X\arrow[r,"f"] & \Spec(k),
        \end{tikzcd}
    \end{center}
    we see that the classes $\pi_{1*}\pi_2^*(\beta_{n-i}\cup \gamma)=f^*f_*(\beta_{n-i}\cup \gamma)$ are pulled back from $\Spec(k)$. Thus, 
    \[\gamma=\sum_i \alpha_i\cup \pi_{1*}\pi_2^*(\beta_{n-i}\cup \gamma)=\sum_i \alpha_i\cup f^*f_*(\beta_{n-1}\cup \gamma)\in \CH^i(X,1)_\text{dec}.\qedhere\]
\end{proof}
If $X$ is an irreducible variety with a compactification $\bar X$ that has the MKP, the proposition says we have an exact sequence
\[\CH(\bar X,1)=0\to  \overline\CH(X,1)\to \CH(\bar X\setminus X)\to  \CH(\bar X),\]
and so the group $\overline\CH(X,1)$ explains which relations need to hold in $\bar X$ between on cycles in $\bar X\setminus X$. Thus, we can roughly view $\overline\CH(X,1)$ as something that measures the failure of compactness, in such situations.
\begin{rmk}\label{HigherBarMgnBad}
     In \cite{CL2}, Canning and Larson show that $\bar\M_{g,n}$ has the CKgP (see Remark~\ref{CKgP} for the definition) with $\mathbb Q$-coefficients for many small $g$ and $n$, which implies the diagonal is in the image of the K\"unneth map. By the above proposition and remark, this gives $\overline\CH(\bar\M_{g,n},1)_{\mathbb Q}=0$ for such $g$ and $n$. Thus, these groups must be torsion, hence these groups are less interesting than $\overline\CH(\M_{g,n},1)$. They also frequently vanish: Proposition~\ref{MKPM0n} and the above proposition imply $\overline\CH(\bar\M_{0,n},1)=0$, and the author has verified that $\overline\CH(\bar\M_{1,n},1)$ vanishes for $n\leq 4$. However, $\overline\CH(\bar\M_2,1)$ must contain nontrivial $2$-torsion elements due to computations with $\partial_1$ made in \cite{ELarson}.
\end{rmk}

\begin{rmk}\label{rational}
    By the previous Remark, we have an exact sequence
    \[0=\overline\CH(\bar\M_{g,n},1)_{\mathbb Q}\to \CH(\M_{g,n},1)_{\mathbb Q}\xrightarrow{\partial_1} \CH(\partial\bar\M_{g,n})_{\mathbb Q}\xrightarrow{\iota_*}\CH(\bar\M_{g,n})_{\mathbb Q}\]
    for $g,n$ with $2g+n\leq 12$. Thus, in these cases, we have $\CH(\M_{g,n},1)_{\mathbb Q}=\ker(\iota_*)$. Also by \cite{CL2} it is true that the Chow ring $\CH(\bar\M_{g,n})$ is equal to the tautological ring $R(\bar\M_{g,n})$ in this range. Computers can preform calculations in the tautological rings via admcycles \cite{DSvZ}, so the rational $\CH(\M_{g,n},1)_{\mathbb Q}$ can be computed in this range. Moreover, by \cite{Petersen}, the kernel of $\iota_*$ should be generated by pullbacks of Getzler's relations for $g=1$ in this range.
\end{rmk}

\begin{cor}\label{Gmind}
    $\overline\CH(\bP^n,1)=0$ and $\overline\CH(B\Gm,1)=0$.
\end{cor}
\begin{proof}
    As $\bP^n$ is a scheme, this follows immediately from the proposition. Additionally, we have
    \[\overline\CH^i(B\Gm,1)=\overline\CH^i(\bP^n,1)=0\]
    for $n$ sufficiently large compared to $i$. 
\end{proof}

From what we know about $B\Gm$, we can prove the following result on $B\mu_n$.
\begin{lemma}\label{Bmu}
    Suppose $n$ is invertible in $k$. The Chow ring of $B\mu_n$ is given by
\[\CH(B\mu_n)=\frac{\Z[\lambda]}{(n\lambda)}\] 
and the higher indecomposable Chow groups are given by 
\[\overline\CH(B\mu_n,1)=0.\]
Moreover, $B\mu_n$ has the MKP.
\end{lemma}
\begin{proof}
(C.f. \cite[Proposition 3.3 (e)]{Bish1})
    Let $\Gm$ act on $\bA^1$ by $t\cdot x=t^nx$. In the quotient $[\bA^1/\Gm]$, the origin gives a closed substack isomorphic to $B\Gm$. The complement is $\Gm/(\Gm)^n\cong B\mu_n$. Thus, the localization exact sequence gives
    \[\overline\CH([\bA^1/\Gm],1) \to \overline\CH(B\mu_n,1) \to \CH(B\Gm)\to \CH([\bA^1/\Gm]).\]
    The class $[B\Gm]\in \CH([\bA^1/\Gm])=\Z[\lambda]$ is equal to $n\lambda$, so by the next lemma, we have that the pushforward is given by
    \[\Z[\lambda]=\CH(B\Gm)\to \CH([\bA^1/\Gm])=\Z[\lambda]\]
    \[\lambda\mapsto n\lambda.\]
    Thus, the cokernel is $\CH(B\mu_n)=\frac{\Z[\lambda]}{(n\lambda)}$ and the kernel is $0$. Moreover, by Corollary~\ref{Gmind}, we have $\overline\CH([\bA^1/\Gm],1)=0$, so exactness of the sequence gives $\overline\CH(B\mu_n,1)=0$.

    Finally, we know that $[\bA^1/\Gm]$ and $B\Gm$ have the MKP by Proposition~\ref{MKPprops}(3,5). Thus, by Proposition~\ref{MKPprops}(1), we have $B\mu_n$ has the MKP.
\end{proof}

Aside from the proof of the previous, the next lemma will be used frequently in the next section.
\begin{lemma}\label{PullbackAn}
    Given actions of $G$ on $\bA^n,\bA^m$, and a $G$-equivariant map $f:\bA^n\to\bA^m$, the pullback $f^*:\CH_{G}(\bA^m,j)\to \CH_{G}(\bA^n,j)$ is an isomorphism. Moreover, if $f$ is proper, we have 
    \[f_*(\beta)=\beta\cdot f_*(1)\]
    under the identification of $\CH_{G}(\bA^m,j)$ and $\CH_{G}(\bA^n,j)$ by $f^*$.
\end{lemma}
\begin{proof}
Consider the commutative diagram

\begin{center}
\begin{tikzcd}
\CH_G(\bA^m,j) && \CH_G(\bA^n,j)\arrow[ll,"f^*"]\\
& \CH_G(\Spec(k).j)\arrow[ul]\arrow[ur] &
\end{tikzcd}
\end{center}

By homotopy invariance, the maps out of $\CH_G(\Spec(k),j)$ are isomorphisms, hence $f^*$ must also be an isomorphism.

If $f$ is proper, and we use $f^*$ to identify the elements of $\CH_G(\bA^m,j)$ with $\CH_G(\bA^n,j)$, the projection formula gives
\[f_*(\beta)=f_*(f^*(\beta))=\beta\cdot f_*(1).\qedhere\] 
\end{proof}

\begin{rmk}[Replacement by CKgP]\label{CKgP}
Instead of using the MKP, we could instead use the integral coefficient Chow K\"unneth generation property (CKgP), which has the advantage of being easier to state. This property has been used in \cite{BS2} and \cite{CL2}. A stack $X$ has the CKgP if for all stacks $Y$, the map $\CH(X)\otimes \CH(Y)\to \CH(X\times Y)$ is surjective. Most of Proposition~\ref{MKPprops} is true for the CKgP:
\begin{itemize}
    \item For property (1), it is not true that if $X$ and $U$ have the CKgP, then so does $Z$, but the other two two out of three properties hold.
    \item Properties (2)-(5) listed in remain unchanged.
    \item Property (6) is true if one assumes moreover that $X$ is a smooth projective variety \cite[Theorem 4.1]{Totaro}.
\end{itemize}
We do not use more than these properties, so one could replace the MKP by the CKgP in this paper. The author instead choose to use the MKP because it is stronger, and the property would be useful in computing the Chow ring of $\bar\M_{g,n}$ for larger $g,n$. For example, the space $\bar\M_{1,n}\times\bar\M_{1,m}$ appears in the boundary of $\bar\M_{2,n+m-2}$. Because $\bar\M_{1,n}$ is not a variety, the CKgP for $\bar\M_{1,n}$ does not imply 
\[\CH(\bar\M_{1,n})\otimes\CH(\bar\M_{1,m})\to \CH(\bar\M_{1,n}\times\bar\M_{1,m})\]
is an isomorphism, but the MKP for $\bar\M_{1,n}$ does imply this is an isomorphism. 
\end{rmk}

\subsection{Comparison With $\ell$-adic Higher Chow Groups} In \cite{ELarson}, Larson introduced the notion of $\ell$-adic higher Chow groups. If the groups $\CH_i(X,1,\Z/\ell^m\Z):=H_1(z_i(X_{\overline k},*)\otimes \Z/\ell^,\Z)$ are finitely generated, then their definition is equivalent to 
\[\CH_i(X,1,\Z_\ell)=\lim_{\leftarrow} \CH_i(X_{\overline k},1,\Z/\ell^m\Z).\]
These groups share many of the same properties as indecomposable higher Chow groups \cite[Section 3]{Bish1}. For example, $\CH(\Spec(k),1,\Z_\ell)=0$, meaning that these groups are able to ``get rid of'' the $k^\times$ in $\CH^1(\Spec(k),1)$, just as indecomposable higher Chow groups can. 

We have the following result relating $\ell$-adic higher Chow groups to indecomposable higher Chow groups. This result explains why there must be similarities between these groups.
\begin{prop}\label{elladic}
Let $X$ be a stack with the  $\CH_i(X),\overline\CH_i(X,1)$ finitely generated as abelian groups for all $i$. Then
\[\CH_i(X,1,\Z_\ell)=\overline\CH_i(X_{\overline k},1) \otimes \Z_\ell.\]
\end{prop}
\begin{rmk}
Though the assumptions on $X$ are strong, they hold in all cases where $\ell$-adic higher Chow groups have been used to compute Chow groups. 
\end{rmk}
\begin{proof}
It suffices to assume $k$ is algebraically closed, so $k=\bar k$ and $X=X_{\overline k}$. We have an exact sequence
\[\CH_{i+1}(X)\otimes  k^\times \to \CH_i(X,1)\to \overline\CH_i(X,1)\to 0\]
by the definition of indecomposable higher Chow groups. 
Tensoring this with $\Z/\ell^m\Z$, we have an exact sequence
\[0\to\CH_i(X,1)\otimes \Z/\ell^m\Z\to \overline\CH_i(X,1)\otimes\Z/\ell^m\Z\to 0,\]
because $k^\times$ is divisible and tensor is right exact. Thus, we have $\CH_i(X,1)\otimes \Z/\ell^m\Z\cong  \overline\CH_i(X,1)\otimes\Z/\ell^m\Z$.

Next, by the universal coefficient theorem, we have
\[0\to \CH_i(X,1)\otimes \Z/\ell^m\Z \to  \CH_i(X,1,\Z/\ell^m\Z) \to \Tor(\CH_i(X),\Z/\ell^m\Z)\to 0.\]
We have that $\CH_i(X,1)\otimes \Z/\ell^m\Z$ and $\Tor(\CH_i(X),\Z/\ell^m\Z)$ are finitely generated because $\overline\CH_i(X,1)$ and $\CH_i(X)$ are finitely generated.
Thus, $\CH_i(X,1,\Z/\ell^m\Z)$ is finitely generated for all $m$, so we have
\[\CH_i(X,1,\Z_\ell)=\lim_{\leftarrow}\CH_i(X,1,\Z/\ell^m\Z).\]
Taking the inverse limit is left exact and the inverse limit of the groups $\Tor(\CH_i(X),\Z/\ell^m\Z)$ is $0$, because the maps between them are multiplication by $\ell$ and $\CH_i(X)$ is finitely generated, so we have
\begin{align*}
\CH_i(X,1,\Z_\ell)&=\lim_{\leftarrow}\CH_i(X,1,\Z/\ell^m\Z)\\
&=\lim_{\leftarrow}\CH_i(X,1)\otimes \Z/\ell^m\Z\\
&=\lim_{\leftarrow}\overline\CH_i(X,1)\otimes \Z/\ell^m\Z\\
&=\overline\CH_i(X,1)\otimes \Z_\ell,
\end{align*}
where the last equality holds because $\overline\CH_i(X,1)$ is finitely generated.
\end{proof}


\section{Equivariant Calculations}
In this section, we perform calculations with equivariant Chow groups to compute the indecomposable first higher Chow groups of the quotient stacks $\M_{1,1}, \M_{1,2},\M_{1,2}^0$ and $\M_{1,3}^0$. Presentations of the Chow rings of these stacks fall out from these computations.

Throughout the next three sections, we use Lemma~\ref{Module} with base $S=B\Gm$, meaning that for every stack $X$ we consider,  $\CH(X)$ and $\overline\CH(X,1)$ will be modules over $\Lambda:=\CH(B\Gm)=\Z[\lambda]$. By \cite[Corollary 2.5]{Bish1}, the class of $\lambda$ in $\CH(\M_{1,1})$ really is the first Chern class of the Hodge bundle on $\M_{1,1}$, and therefore on any stack mapping to $\M_{1,1}$, justifying the notation of $\lambda$. 

As we will see, the maps $\Lambda\to \CH(X)$ are surjective for the $X$ we care about in this section, so in determining the $\Lambda$-module structure on $\overline\CH(X,1)$, we determine the $\CH(X)$-module structure on $\overline\CH(X,1).$

Recall  that for an action of $\Gm$ on $\bA^n=\Spec(k[x_1,\dots,x_n])$ and $f\in k[x_1,\dots,x_n]$ such that $t\cdot f=z^nf$ for $z\in \Gm$, then $[V(f)]=n\lambda\in \CH_{\Gm}(\bA^n)=\Lambda$.

\subsection{$\M_{1,1}$}

There is a well-known stack presentation of $\M_{1,1}$ as  $[U_1/\Gm]$, where 
\[U_1=\{(a,b)\in \bA^2| x^3+ax+b\text{ has distinct roots}\},\] 
and $\Gm$ acts with weights $(-4,-6)$ \cite[Section 5]{EG}. The discriminant of $x^3+ax+b$ is $-4a^3-27b^2$, so the complement of $U_1$ in $\bA^2$ is given by $\iota_1:  Z_1:=V(-4a^3-27b^2)\hookrightarrow \bA^2$. 

We study the map $\iota_{1*}:\CH_{\Gm}( Z_1)\to \CH_{\Gm}(\bA^2)$. Let 
\[f_1: \bA^1\to  Z_1\]
\[c\mapsto (-3c^2,2c^3)\]
be the normalization. This map is equivariant if we act by $\Gm$ on $\bA^1$ with weight $-2$.

\begin{lemma}\label{Z1}
    $\iota_{1*}:\CH_{\Gm}( Z_1)\to \CH_{\Gm}(\bA^2)=\Lambda$ is injective and has image generated by $12\lambda$.
\end{lemma}
\begin{proof}
    We have
    \begin{center}
        \begin{tikzcd}
            \bA^1 \arrow[rr,"\iota_1\circ f_1"]\arrow[rd,"f_1"] & & \bA^2\\
            &  Z_1\arrow[ru,"\iota_1"]
        \end{tikzcd}.
    \end{center}
    By Lemma~\ref{PullbackAn}, we have $(\iota_1\circ f_1)^*$ is an isomorphism and $(\iota_1\circ f_1)_*$ is multiplication by 
    \[(\iota_1\circ f_1)_{*}([\bA^1])=\iota_{1*}f_{1*}([\bA^1])=\iota_{1*}([ Z_1])=[ Z_1].\]
    Because $-4a^3-27b^2$ has degree $-12$, we have $[ Z_1]=-12t$. Thus, $(\iota_1\circ f_1)_*$ has image generated by $12\lambda$ and is injective.
    
    Now, $f_1$ is a universal separable homeomorphism by Example~\ref{cusp}, so by Proposition~\ref{GLnQuotients} and Proposition~\ref{ushChow}, $f_{1*}$ is an isomorphism. This says that $\iota_{1*}$ has the same image, $(12\lambda)$, and is injective. 
\end{proof}

\begin{thm}\label{M11}
   The Chow ring of $\M_{1,1}$ is given by
    \[\CH(\M_{1,1})=\frac{\Lambda}{(12\lambda)}=\frac{\Z[\lambda]}{(12\lambda)},\]
    and the higher indecomposable Chow groups are given by
\[\overline\CH(\M_{1,1},1)=0.\]
Moreover, $\M_{1,1}$ has the MKP.
\end{thm}
\begin{proof}
Note 
\[\overline\CH_{\Gm}(\bA^2,1)\cong \overline\CH_{\Gm}(\pt,1)=\overline\CH(B\Gm,1)=0\]
by homotopy invariance and Corollary~\ref{Gmind}.
The localization exact sequence then gives
\[0\to \CH(\M_{1,1},1)_{\ind}\to\CH_{\Gm}( Z_1)\xrightarrow{\iota_{1*}} \CH_{\Gm}(\bA^2)\to \CH(\M_{1,1})\to 0.\]
By Lemma~\ref{Z1}, $\iota_*$ is injective and has image $(12\lambda)\subseteq \CH_{\Gm}(\bA^2)=\Lambda$. Thus, exactness gives $\overline\CH(\M_{1,1},1)=0$ and $\CH(\M_{1,1})=\Lambda/(12\lambda)=\Z[\lambda]/(12\lambda)$. 

Finally, we want to say that $\M_{1,1}$ has the MKP. By Proposition~\ref{MKPprops}(1), it suffices to show that $[\bA^2/\Gm]$ and $[ Z_1/\Gm]$ have the MKP. We know $[\bA^2/\Gm]$ and $[\bA^1/\Gm]$ have the MKP by Proposition~\ref{MKPprops}(3,5), and because $f_1$ is a universal separable homeomorphism, $[ Z_1/\Gm]$ also has the MKP by Proposition~\ref{MKPprops}(4).
\end{proof}

\begin{rmk}\label{totalhigherchowM11}
    One can adapt the above to calculate the full higher Chow groups of $\M_{1,1}$ over an arbitrary field (Proposition~\ref{AlgebraicallyClosed} handles this for an algebraically closed field). For $i\geq 1$ we have
    \[\CH^i_{\Gm}(\bA^2,1)=\CH^i_{\Gm}(\Spec(k),1)=\CH^i(\bP^n,1)=k^\times\]
    and 
    \[\CH^i_{\Gm}( Z_1)=\CH^i_{\Gm}(\bA^1,1)=k^\times.\]
    Then the localization exact sequence gives
    \[k^\times \to k^\times \to \overline\CH^i(\M_{1,1},1)\to 0.\]
    An argument analogous to Lemma~\ref{PullbackAn} gives that the map $k^\times \to k^\times$ is $z\mapsto z^{12}$. Thus, one has 
    \[\CH(\M_{1,1},1)=\bigoplus_{i=1}^\infty \frac{k^\times}{(k^\times)^{12}}=\CH(\M_{1,1})\otimes \CH_{-1}(k,1).\]
\end{rmk}

\subsection{$\M_{1,2}$}
Using the presentation for $\M_{1,1}$, one can get a presentation for $\M_{1,2}$ by including the data of a second marked point, $(x_2,y_2)$. Then, the value of $b$ is determined, so we have $\M_{1,2}\cong U_2/\Gm]$, where 
 \[U_2:=\{(a,x_2,y_2)\in \bA^3|-4a^3-27B(a,x_2,y_2)^2\neq 0\},\]
 \[B(a,x_2,y_2):=y_2^2-x_2^3-ax_2,\]
and $\Gm$ acts with weights $(-4,-2,-3)$. The morphism
\[\pi: \M_{1,2}\to \M_{1,1}\]
which forgetting the second marked point corresponds to 
\[U_2\to U_1\]
\[(a,x_2,y_2)\mapsto (a,B(a,x_2,y_2)).\]
The complement of $U_2$ in $\bA^3$ is given by 
\[\iota_2:  Z_2:=V(-4a^3-27B(a,x_2,y_2)^2)\hookrightarrow \bA^2.\] 

Now, we study the map $\iota_{2*}:\CH_{\Gm}( Z_2)\to \CH_{\Gm}(\bA^3)=\Lambda$. We have the $\Gm$-equivariant map
    \[\bA^3\to \bA^2\]
    \[(a,x_2,y_2)\mapsto (a,B(x_2,y_2,a)).\]
    This restricts to a map $ Z_2\to  Z_1$, as $ Z_2$ is, essentially by definition, the inverse image of $ Z_1$ under this map. Define $\widehat{ Z}_2:=\bA_1\times_{ Z_1}  Z_2$ so that
    \begin{center}
        \begin{tikzcd}
           \widehat{ Z}_2 \arrow[r,"f_2"]\arrow[d] &  Z_2\arrow[d]\\
            \bA^1\arrow[r,"f_1"] &  Z_1
        \end{tikzcd}
    \end{center}
    is Cartesian. Everything involved is an affine scheme, so a tensor product computation gives 
    \[ \widehat{ Z}_2=\Spec k[x_2,y_2,c]/(y_2^2-x_2^3+3c^2x_2-2c^3).\]
Next, define
\[g_2: \bA^2\to\widehat{ Z}_2\]
\[(c,d)\mapsto (d^2-2c,d^3-3cd,c).\]
To make this map equivariant, we act with weight $-2$ on $c$ and $-1$ on $d$. This map is also finite, hence proper, as both $c$ and $d$ are integral over $k[x_2,y_2,c]/(y_2^2-x_2^3+3c^2x_2+2c^3)$. 

Over the locus $W_2:=D(x_2-c)$, we have that $g_2$ is an isomorphism, as one can check that
\[(x_2,y_2,c)\mapsto \left(c,\frac{y_2}{x_2-c}\right)\]
is an inverse on this open subset. The reduced complement, $C_2\hookrightarrow \widehat Z_2$, of $W_2$ is given by $C_2=V(x_2-c,y_2)$. Then $C_2=\Spec k[x_2,y_2,c]/(x_2-c,y_2) \cong \Spec k[c]$. Moreover, $g_2^{-1}(C_2)=\Spec k[c,d]/(d^2-3c) \cong \Spec k[d]$.

\begin{lemma}\label{Z2}
    The image of $\iota_{2*}:\CH_{\Gm}( Z_2)\to \CH_{\Gm}(\bA^3)=\Lambda$ is generated by $12\lambda$, and the kernel is given by  
    \[\ker(\iota_{2*})=\frac{\Lambda\langle \gamma\rangle}{\langle 2\gamma\rangle },\]
    where $\gamma=[V(y_2,a+3x_2^2)]-[V(y_2,4a+3x_2^2)]$.
\end{lemma}
\begin{proof}
    Lemma~\ref{SplitChow} gives us that \[\CH_{\Gm}(\widehat{ Z}_2)=\CH_{\Gm}(\bA^2)\coprod_{\CH_{\Gm}(g_2^{-1}C_2)} \CH_{\Gm}(C_2).\]
    Because $f_2:\widehat{ Z}_2\to  Z_2$ is a universal separable homeomorphism, we have $f_{2*}:\CH_{\Gm}(\widehat{ Z}_2)\to \CH_{\Gm}( Z_2)$ is an isomorphism, and so we have \[\CH_{\Gm}({ Z_2})=\CH_{\Gm}(\bA^2)\coprod_{\CH_{\Gm}(g_2^{-1}C_2)} \CH_{\Gm}(C_2).\]
    Thus, to compute the image of $\iota_{2*}$, it suffices to compute the images of the pushforwards of $\bA^2\to\bA^3$ and $C_2\to \bA^3$. 

    The map $\bA^2\to \bA^3$ is the composition \[\bA^2\xrightarrow{g_2}\widehat{ Z}_2\xrightarrow{f_2}  Z_2\xrightarrow{\iota_2} \bA^3.\] The first two maps in this composition are birational, and the second is a closed embedding, so $1\in\CH_{\Gm}(\bA^2)$ pushes forward to $[ Z_2]$. Because $-4a^3-27B(a,x_2,y_2)$ has degree $-12$, we have $[ Z_2]=-12\lambda$. By Lemma~\ref{PullbackAn}, the map $\CH_{\Gm}(\bA^2)\to \CH_{\Gm}(\bA^3)$ is multiplication by $-12\lambda$, under the identification of the pullback. 

    The map $C_2\cong \Spec k[c] \to \bA^3$ is the composition 
    \[C_2\hookrightarrow\widehat{ Z}_2\xrightarrow{f_2}  Z_2\xrightarrow{\iota_2} \bA^3.\] Using the definitions of these maps, we can compute it explicitly as $c\mapsto (c,0,-3c^2)$. The image of this is the closed subscheme $V(y_2,a+3x_2^2)$. We know $[V(y_2)]=-3\lambda$ and $[V(a+3x_2^2)]=-4\lambda$, so $[V(y_2,a+3x_2^2)]=12\lambda^2$, because the intersection is transverse. By Lemma~\ref{PullbackAn}, the map $\CH_{\Gm}(C_2)\to \CH_{\Gm}(\bA^3)$ is multiplication by $12\lambda^2$, under the identification of the pullback. Combining this with the conclusion of the previous paragraph, we see that the image of $\iota_{2*}$ is generated by $12\lambda$. 

    We next compute the kernel of $\iota_{2*}$. From the above descriptions, we see that the kernel of 
    \[\Lambda^{\oplus 2}=\CH_{\Gm}(\bA^2)\oplus \CH_{\Gm}(C_2)\to \CH_{\Gm}(\bA^3)=\Lambda\] 
    is $\{(\lambda p(\lambda),p(\lambda))|p(\lambda)\in\Lambda\}$. To get the kernel out of the map out of $\CH_{\Gm}(\bA^2)\amalg_{\CH_{\Gm}(g_2^{-1}C_2)} \CH_{\Gm}(C_2)$, we need to quotient this out by the image of 
    \[\Lambda=\CH_{\Gm}(g_2^{-1}(C_2))\to \CH_{\Gm}(\bA^2)\oplus \CH_{\Gm}(C_2)=\Lambda^{\oplus 2},\] 
    where we negate one of these two pushforwards.

    The map $g_2^{-1}(C_2)\cong \Spec k[d] \to \bA^2$ is a closed embedding cut out by $d^2-3c$. So $1$ pushes forward to $[V(d^2-3c)]=-2\lambda$. By Lemma~\ref{PullbackAn}, we have $\CH_{\Gm}(g_2^{-1}(C_2))\to \CH_{\Gm}(\bA^2)$ is given by $p(\lambda)\mapsto -2\lambda p(\lambda)$. 

    The map $g_2^{-1}(C_2)\to C_2$ is given by $d\mapsto \frac{d^2}{3}$ under the isomorphisms $C_2\cong \Spec k[c]$ and $g_2^{-1}(C_2)\cong \Spec k[d]$. This is a degree $2$ map, so $1\in \CH_{\Gm}(g_2^{-1}(C_2))$ maps to $2$. By Lemma~\ref{PullbackAn}, we have $\CH_{\Gm}(g_2^{-1}(C_2))\to \CH_{\Gm}(C_2)$ is given by $p(\lambda)\mapsto 2p(\lambda)$. Thus, the map $\CH_{\Gm}(g_2^{-1}(C_2))\to \CH_{\Gm}(\bA^2)\oplus \CH_{\Gm}(C_2)$ has image $\{(2\lambda p(\lambda),2p(\lambda))|p(\lambda)\in\Lambda\}$. 

    Now
    \begin{align*}
        \ker(\iota_2)&\cong \ker(\CH_{\Gm}(\bA^2))\coprod_{\CH_{\Gm}(g_2^{-1}C_2)} \CH_{\Gm}(C_2)\to \CH_{\Gm}(\bA^3))\\
        &\cong \frac{\ker(\CH_{\Gm}(\bA^2)\oplus \CH_{\Gm}(C_2)\to \CH_{\Gm}(\bA^3))}{\im(\CH_{\Gm}(g_2^{-1}(C_2)\to \CH_{\Gm}(\bA^2)\oplus \CH_{\Gm}(C_2))}\\
        &=\frac{\{(\lambda p(\lambda),p(\lambda))|p(\lambda)\in\Lambda\}}{\{(2\lambda p(\lambda),2p(\lambda))|p(\lambda)\in\Lambda\}}\\
        &= \frac{\Lambda\langle (\lambda,1) \rangle}{\langle 2(\lambda,1)\rangle}.
    \end{align*}
    Under this isomorphism, the element $(\lambda,1)$ corresponds to $\lambda[\widehat Z_2]+[C_2]$ in $\CH(\widehat{ Z}_2)$, hence $\lambda[ Z_2]+f_{2*}([C_2])$ in $\CH( Z_2)$. We can compute
    \[f_{2*}([C_2])=[f_{2}(C_2)]=[V(y_2,a+3x_2^2)]\]
    and
    \[\lambda[ Z_2]=(f_{2}\circ g_{2})_*(\lambda)=(f_{2}\circ g_{2})_*(-[V(d)])=-[V(y_2,4a+3x_2^2)],\]
    giving the class in the statement of the lemma. 
\end{proof}

\begin{definition}
    Define the closed substack $Y_1(2)\hookrightarrow \M_{1,2}$ to parameterize curves $(C,p_1,p_2)$ so that $\OO(2(p_2-p_1))$ is trivial, and let $\M_{1,2}^0$ be its complement in $\M_{1,2}$. 
\end{definition}
Note $\OO(p_2-p_1)$ is trivial if and only if $p_2$ is a $2$-torsion point with respect to the elliptic curve group law. Thus, $Y_1(2)$ is a modular curve, which is where this notation for this stack comes from. Inside of our quotient stack, $Y_1(2)$ is the quotient of the vanishing of $y_2$ inside $U_2$ by $\Gm$. By factoring $-4a^3-27B(a,x_2,0)^2$, we have $Y_1(2)$ is given by the quotient of 
\[[(a,x_2,0)\in \bA^3| -(a+3x_2^2)^2(4a+3x_2^2)\neq 0\}\]
by $\Gm$.

\begin{thm}\label{M12}
The Chow ring of $\M_{1,2}$ is given by
\[\CH(\M_{1,2})=\frac{\Lambda}{(12\lambda)}=\frac{\Z[\lambda]}{(12\lambda)},\]
 and the higher indecomposable Chow group $\overline\CH(\M_{1,2},1)$ is given by 
 \[\overline\CH(\M_{1,2},1)=\frac{\Lambda\langle \mathfrak p\rangle}{\langle 2\mathfrak p\rangle },\]
    where 
\[\mathfrak p:=\left[Y_1(2),\frac{a+3x_2^2}{4a+3x_2^2}\right].\]
Moreover, $\M_{1,2}$ has the MKP.
\end{thm}
\begin{proof}
We have $\overline\CH_{\Gm}(\bA^3,1)=0$ by Proposition~\ref{IndZero} and Corollary~\ref{Gmind}, so the localization exact sequence gives
\[0\to \overline\CH(\M_{1,2},1)\xrightarrow{\partial_1}\CH_{\Gm}( Z_1)\xrightarrow{\iota_{2*}} \CH_{\Gm}(\bA^2)\to \CH(\M_{1,2})\to 0.\]
By Lemma~\ref{Z2}, the image of $\iota_{2*}$ is $(12\lambda)$, so 
\[\CH(\M_{1,2})=\frac{\Lambda}{(12\lambda)}=\frac{\Z[\lambda]}{(12\lambda)}.\]

Additionally, the exact sequence and Lemma~\ref{Z2} give
\[\overline\CH(\M_{1,2},1)\stackrel{\partial_1}\cong\ker(\iota_{2*})=\frac{\Lambda\langle \gamma\rangle}{\langle 2\gamma\rangle},\]
where $\gamma$ is defined in the statement of this Lemma.

The expression $(a+3x_2^2)/(4a+3x_2^2)$ is $\Gm$ invariant, and hence defines a unit on $Y_1(2)$ by the above description of $Y_1(2)$. By Proposition~\ref{formalsum}, the $\mathfrak p$ from the statement is a well defined element of $\overline\CH^2(\M_{1,2},1)$, and we have
\[\partial_1(\mathfrak p)=\divisor\left(\frac{a+3x_2^2}{4a+3x_2^2}\right)=[V(y_2,a+3x_2^2)]-[V(y_2,4a+3x_2^2)]=\gamma.\]
So the isomorphism $\overline\CH(\M_{1,2},1)\stackrel{\partial_1}\cong \Lambda\langle \gamma\rangle/\langle 2\gamma\rangle$ gives 
 \[\overline\CH(\M_{1,2},1)=\frac{\Lambda\langle \mathfrak p\rangle}{\langle 2\mathfrak p\rangle }.\]

Finally, we want to show that $\M_{1,2}$ has the MKP. By Proposition~\ref{MKPprops}(1), it suffices to show that $[\bA^3/\Gm]$ and $[ Z_2/\Gm]$ have the MKP.  By Proposition~\ref{MKPprops}(3,5), $[\bA^3/\Gm]$ has the MKP. By Proposition~\ref{MKPprops}(4), it suffices to show that $[\widehat{ Z}_2/\Gm]$ has the MKP. The closed substack $[C/\Gm]\subseteq [\widehat{ Z}_2/\Gm]$ has the MKP because it is isomorphic to $[\bA^1/\Gm]$ where $\Gm$ acts with weight $-2$, so by Proposition~\ref{MKPprops}(1) it suffices to show that the complement $[W/\Gm]$ has the MKP. For this, we use that $[W/\Gm]\cong [g^{-1}(W)/\Gm]$, and $[g^{-1}(W)/\Gm]$ is open in $[\bA^2/\Gm]$ with complement $[g^{-1}(C_2)/\Gm]\cong [\bA^1/\Gm]$. These latter two spaces have the MKP, so by Proposition~\ref{MKPprops}(1), so does $[g^{-1}(W)/\Gm]$.
\end{proof}
\begin{rmk}\label{totalhigherchowM12}
     One could try and use the above to calculate the full higher Chow groups $\CH(\M_{1,2},1)$. With similar reasoning to Remark~\ref{totalhigherchowM11}, one gets an exact sequence
     \[0\to \frac{k^\times}{(k^\times)^{12}} \to \CH^i(\M_{1,2},1)\to \frac{\Z}{2\Z}\to 0.\]
     Thus, one needs to solve an extension problem to figure out the group $\CH^i(\M_{1,2},1)$. The author was able to show the extension must be trivial if the field $k$ contains a square-root of $-1$, but does not know what happens otherwise. These group extension problems persist (and increase in number) when trying to compute $\CH^i(\M_{1,n},1)$ for $n=3,4.$
\end{rmk}

\begin{definition}\label{pij}
    For $j\neq k\in [n]$, let $\pi_{jk}:\M_{1,n}\to \M_{1,2}$ be the morphism that forgets all markings but $j$ and $k$. Define $\mathfrak p_{jk}:=\pi_{jk}^*(\mathfrak p)\in \overline\CH^2(\M_{1,n},1)$.
\end{definition}

\subsection{$\M_{1,2}^0$}
Recall $\M_{1,2}^0:=\M_{1,2}\setminus Y_1(2)$. Our description of $\M_{1,3}$ involves $\M_{1,2}^0$, so we compute $\CH(\M_{1,2}^0),\overline \CH(\M_{1,2}^0,1)$ now.

\begin{thm}\label{M120}
    The Chow ring of $\M_{1,2}^0$ is given by
    \[\CH(\M_{1,2}^0)=\frac{\Lambda}{(3\lambda)},\]
     and the higher indecomposable Chow group $\overline\CH(\M_{1,2}^0,1)$ is given by 
    \[\overline\CH(\M_{1,2}^0,1)=\frac{\Lambda\langle \mathfrak f\rangle}{\langle \lambda\mathfrak f\rangle},\]
    where 
    \[\mathfrak f:=\left[\M_{1,2}^0, \frac{y_2^4}{ \Delta}\right]\]
    and
    \[ \Delta:=-4a^3-27B(a,x_2,y_2)^2.\]
    Finally, $\M_{1,2}^0$ has the MKP. 
\end{thm}
\begin{proof}
By Theorem~\ref{M12}, the localization exact sequence for $\M_{1,2}^0\subseteq \M_{1,2}$ is 
\[\frac{\Lambda\langle \rho\rangle}{\langle 2\rho\rangle}\to \overline\CH(\M_{1,2}^0,1)\xrightarrow{\partial_1} \CH(Y_1(2))\xrightarrow{j_{*}} \frac{\Lambda}{(12\lambda)}\to \CH(\M_{1,2}^0)\to 0.\]
The class $\rho$ must map to $0$, since it is pushed forward from $\overline\CH^1(Y_1(2),1)$, giving $\overline\CH(\M_{1,2}^0,1)\stackrel{\partial_1}\cong \ker(j_{*})$.

Viewing $Y_1(2)$ as an open inside $\bA^2$ with coordinates $a,x_2$, we have the exact sequence
\[\CH_{\Gm}(V(a+3x_2^2)\cup V(4a+3x_2^2))\to \CH_{\Gm}(\bA^2)\to \CH(Y_1(2))\to 0.\]
We have 
\[[V(a+3x_2^2)]=[V(4a+3x_2^2)]=-4\lambda\in \Lambda=\CH_{\Gm}(\bA^2)\]
and
\[[V(a+3x_2^2)/\Gm]\cong [V(4a+3x_2^2)/\Gm]\cong [\bA^1/\Gm],\]
so by Lemma~\ref{PullbackAn} and exactness of the sequence, we have
\[\CH(Y_1(2))=\frac{\Z[\lambda]}{(4\lambda)}.\]

Plugging this into our previous exact sequence, we have 
\[0\to \overline\CH(\M_{1,2}^0,1)\xrightarrow{\partial_1} \frac{\Lambda}{(4\lambda)}\xrightarrow{j_{*}} \frac{\Lambda}{(12\lambda)}\to \CH(\M_{1,2}^0)\to 0.\]
By commutativity of 
\begin{center}
    \begin{tikzcd}
        Y_1(2)\arrow[r,""]\arrow[d] & \M_{1,2}\arrow[d]\\
        \bA^2\arrow[r] & \bA^3
    \end{tikzcd},
\end{center}
we have $[Y_1(2)]=-3\lambda\in \CH(\M_{1,2})$ and $j^*(\lambda)=\lambda$. Thus, $j_{*}$ is given by multiplication by $-3\lambda$, by a similar argument to Lemma~\ref{PullbackAn}.

By exactness, we get
\[\CH(\M_{1,2}^0)=\frac{\Lambda}{(3\lambda)}=\frac{\Z[\lambda]}{(3\lambda)}\]
and 
\[\overline\CH(\M_{1,2}^0,1)\stackrel{\partial_1}\cong \ker(j_{*})=\frac{\Lambda\langle 4[Y_1(2)]\rangle}{\langle \lambda \cdot 4[Y_1(2)]\rangle}.\]
We note that $\frac{y_2^4}{\Delta}$ is $\Gm$-invariant, is a unit on $\M_{1,2}^0$, meaning $\mathfrak f$ is a well-defined element of $\overline\CH^1(Y_1(2),1)$.  Moreover, we have
\[\divisor\left(\frac{y_2^4}{\Delta}\right)=4[V(y_2)]]\]
on $U_2$. Thus, by Lemma~\ref{units}$(4')$, we have 
\[\partial_1(\mathfrak f)=4[Y_1(2)],\] and the above isomorphism gives
\[\overline\CH(\M_{1,2}^0)=\frac{\Lambda\langle \mathfrak f\rangle}{\langle \lambda\mathfrak f \rangle}.\]

Finally, we argue that $\M_{1,2}^0$ has the MKP. By Proposition~\ref{MKPprops}(1), it suffices to show that $\M_{1,2}$ and $Y_1(2)$ have the MKP. The former is true by Theorem~\ref{M12}. As noted above, $Y_1(2)$ is open inside $[\bA^2/\Gm]$, with complement given by a union of two stacks isomorphic to $[\bA^1/\Gm]$ intersecting at a $B\Gm$. These all have the MKP by Proposition~\ref{MKPprops}(1,3,5), so by Proposition~\ref{MKPprops}(1), so does $Y_1(2)$.  
\end{proof}

\subsection{$\M_{1,3}^0$}
\begin{definition}
    Define the open substack $\M_{1,3}^0\subseteq \M_{1,3}$ to parameterize curves $(C,p_1,p_2,p_3)$ so that $\OO(p_2+p_3-2p_1)$ is nontrivial.  
\end{definition}

\begin{lemma}\label{M130Complement}
    The closed substack $\M_{1,3}\setminus \M_{1,3}^0\subseteq \M_{1,3}$ is isomorphic to $\M_{1,2}^0$ via restriction of the map $\pi:\M_{1,3}\to \M_{1,2}$ forgetting the third marked point. The inverse is given by
    \[\M_{1,2}^0\to \M_{1,3}\]
    \[(C,p_1,p_2)\mapsto (C,p_1,p_2,[-1]p_2).\]
\end{lemma}
\begin{proof}
    By definition, the complement of $\M_{1,3}^0$ parameterizes $(C,p_1,p_2,p_3)$ such that $\OO(p_2+p_3-2p_1) \cong \OO$. We can rewrite this as $\OO(p_3)\cong \OO(2p_1-p_2)$. Because line bundles and sections are equivalent data on a smooth genus $1$ curve, we see that we can recover $p_3$ from $p_1$ and $p_2$ on this locus. This gives a rational section to the map $\pi$ whose image is $\M_{1,3}\setminus \M_{1,3}^0$. This rational section is well defined so long as $\OO(2p_1-p_2)$ is not isomorphic to $\OO(p_1)$ nor $\OO(p_2)$. The first option would imply $p_1=p_2$, which cannot happen, and the second option is equivalent to saying that $\OO(2p_1-2p_2)$ is nontrivial. Thus, we have $\M_{1,3}\setminus \M_{1,3}^0\cong \M_{1,2}^0$ via $\pi$. Converting the line bundles into points, we see that the claimed map is this rational section we have constructed.
\end{proof}

Using the presentation for $\M_{1,2}$ as a quotient stack, one can get a similar presentation for $\M_{1,3}$ by including the data of a third marked point $(x_3,y_3)$, and stipulating that $(x_2,y_2)\neq (x_3,y_3)$. Then, the value of $a$ is determined so long as $x_2\neq x_3$. The only way $x_2=x_3$ is if $p_2=[-1]p_3$, i.e. $(C,p_1,p_2,p_3)\in \M_{1,3}\setminus \M_{1,3}^0$. Thus, we have a stack presentation of $\M_{1,3}^0$ as $[U_3/\Gm]$ where
\[U_3:=\{(x_2,y_2,x_3,y_3)\in D_{\bA^4}(x_3-x_2)|-4A^3-27B^2\neq 0\},\]
where 
\[A=A(x_2,y_2,x_3,y_3):=\frac{y_3^2-y_2^2-x_3^3+x_2^3}{x_3-x_2}\]
\[B=B(x_2,y_2,x_3,y_3):=y_2^2-x_2^3-Ax_2\]
and $\Gm$ acts with weights $(-2,-3,-2,-3)$. This presentation was previously used in \cite{Bish1}. The map
\[\pi: \M_{1,3}^0\hookrightarrow \M_{1,3}\to \M_{1,2}\]
forgetting the third marked point corresponds to 
\[U_3\to U_2\]
\[(x_2,y_2,x_3,y_3)\mapsto (A(x_2,y_2,x_3,y_3),x_2,y_2).\].

We first remove $V_{\bA^4}(x_3-x_2)$ from $\bA^4$ Because $V_{\bA^4}(x_3-x_2)\cong \bA^3$ and $x_3-x_2$ has degree $2$, the localization exact sequence and Lemma~\ref{PullbackAn} give 
\[\CH(D_{\bA^4}(x_3-x_2))=\frac{\Lambda}{(2\lambda)}\]
and
\[\overline\CH(D_{\bA^4}(x_3-x_2),1)=0.\]

Let $\iota_3: Z_3^0\hookrightarrow D_{\bA^4}(x_3-x_2)$ be the complement of $U_3$ in $D_{\bA^4}(x_3-x_2)$, so $ Z_3^0=V(-4A^3-27B^2)$.

We study the map $\iota_{3*}:\CH_{\Gm}( Z_3^0)\to \CH_{\Gm}(D_{\bA^4}(x_3-x_2))$.     We have the equivariant map
    \[D_{\bA^4}(x_3-x_2)\to \bA^2\]
    \[(x_2,y_2,x_3,y_3)\mapsto (A(x_2,y_2,x_3,y_3),B(x_2,y_2,x_3,y_3)).\]
    This restricts to a map $ Z_3^0\to  Z_1$, as $ Z_3^0$ is, essentially by definition, the inverse image of $ Z_1$ under this map. Define $\widehat{ Z}_3^0:=\bA_1\times_{ Z_1}  Z_3^0,$ so 
    \begin{center}
        \begin{tikzcd}
           \widehat{ Z}_3^0 \arrow[r,"f_3"]\arrow[d] &  Z_3^0\arrow[d]\\
            \bA^1\arrow[r,"f_1"] &  Z_1
        \end{tikzcd}
    \end{center}
    is Cartesian. Everything involved is an affine scheme, so a tensor product computation gives 
    \[ \widehat{ Z}_3^0=\Spec k[x_2,y_2,x_3,y_3,c]_{x_3-x_2}/(A+3c^2,B-2c^3).\]

Now, we parameterize $\widehat{ Z}_3^0$ by
\[(c,d_2,d_3)\mapsto (d_2^2-2c,d_2^3-3cd_2,d_3^2-2c,d_3^3-3cd_3,c).\]
Note the similarities with the parameterization $g_2:\bA^2\to \widehat Z_2$. This this tuple satisfies the equations $A+3c^2=B-2c^3=0$, so they determine a morphism to $\widehat{ Z}_3^0$ so long as $x_2\neq x_3$, i.e. $d_2^2\neq d_3^2$. Therefore, the above gives a morphism $g_3:D_{\bA^3}(d_3^2-d_2^2) \to\widehat{ Z}_3^0$. To make this map equivariant, we act with weights $(-2,-1,-1)$. Note $g_3$ is finite, hence proper, as $c,d_2,d_3$ are integral over $k[x_2,y_2,x_3,y_3,c]_{x_3-x_2}/(A+3c^2,B-2c^3).$ 

Over the locus $W_3:=D_{\widehat Z_3^0}((x_2-c)(x_3-c))$, we have that $g_3$ is an isomorphism, as 
\[(x_2,y_2,x_3,y_3,c)\mapsto \left(c,\frac{y_2}{x_2-c},\frac{y_3}{x_3-c}\right)\]
gives an inverse on this open subset. The reduced complement of $W_3$ is given by $C_3:=C_3^2\cup C_3^3\subseteq\widehat{ Z}_3^0$, where 
\[C_3^\ell:=V(x_\ell-c,y_\ell)\hookrightarrow \widehat Z_3^0.\] Calculations on the ideals show that $C_3^2=\Spec(R)$ where 
\[R=\frac{k[x_2,y_2,x_3,y_3,c]_{x_3-c}}{(y_2,x_2-c,y_3^2-x_3^3+3c^2x_3-2c^3)}\cong \frac{k[x_3,y_3,c]_{x_3-c}}{(y_3^2-x_3^3+3c^2x_3-2c^3)}.\]

Note that this latter expression gives exactly the open $W\subseteq\widehat{ Z}_2$, up to changing variable names, and recall $g_2^{-1}(W_2)\to W_2$ was an isomorphism. Using this, we have an isomorphism
\[\Spec k[c,d]_{d^2-3c}\xrightarrow{\sim}C_3^2\]
\[(c,d)\mapsto  (c,0,d^2-2c,d^3-3cd,c)\]
from which we can compute the Chow groups of $C_3^2$. Furthermore, we can compute, 
\[g_3^{-1}(C_3^2)=\Spec k[c,d_2,d_3]_{d_3^2-d_2^2}/{(d_2^2-3c)}\cong \Spec k[d_2,d_3]_{d_3^2-d_2^2}.\]
Analogously, we get an isomorphism
\[\Spec k[c,d]_{d^2-3c}\xrightarrow{\sim} C_3^3\]
\[(c,d)\mapsto (d^2-2c,d^3-3cd,c,0,c)\]
and
\[g_3^{-1}(C_3^3)=\Spec{k[c,d_2,d_3]_{d_3^2-d_2^2}}/{(d_3^2-3c)}\cong \Spec k[d_2,d_3]_{d_3^2-d_2^2}.\]

\begin{lemma}\label{Z30}
    The pushforward $\iota_{3*}:\CH_{\Gm}( Z_3^0)\to \CH_{\Gm}(D_{\bA^4}(x_3-x_2))$ is equal to zero and 
    \[\CH( Z_3^0)=\frac{\Lambda \langle [ Z_3^0], \rho_2, \rho_3 \rangle}{\langle \lambda [ Z_3^0], 2\rho_2, 2\rho_3\rangle}\]
    where 
    \[\rho_j:=[V(y_j, A+3x_j^2)].\]
\end{lemma}
\begin{proof}
First, note that $C_3^2$ and $C_3^3$ are disjoint: they are cut out of of $\widehat{ Z}_3^0$ by $(x_2-c,y_2)$ and $(x_3-c,y_3)$ respectively, so any point in their intersection would be a zero of $x_3-x_2$, which cannot happen. Furthermore, this implies $g_3^{-1}(C_3^2)$ and $g_3^{-1}(C_3^3)$ are also disjoint. 

Using Lemma~\ref{SplitChow}, we have 
\begin{align*}
    \CH_{\Gm}(\widehat{ Z}_3^0)&=\CH_{\Gm}(W_3)\coprod_{\CH_{\Gm}(g_3^{-1}C_3)} \CH_{\Gm}(C_3)\\
    &=CH_{\Gm}(W_3)\coprod_{\CH_{\Gm}(g_3^{-1}C_3)} (\CH_{\Gm}(C_3^2)\oplus \CH_{\Gm}(C_3^3)).
\end{align*}
And we also know the pushforward $f_{3*}:\CH_{\Gm}(\widehat{ Z}_3^0)\to\CH_{\Gm}( Z_3^0)$ is an isomorphism, since $f_3$ is a universal separable homeomorphism. Thus, to show $\iota_{3*}$ is zero, it suffices to show that pushforwards of $D_{\bA^3}(d_3^2-d_2^2)\to D_{\bA^4}(x_3-x_2)$ and $C_3^j\to D_{\bA^4}(x_3-x_2)$ are zero. 

The map $g_3: D_{\bA^3}(d_3^2-d_2^2)\to D_{\bA^4}(x_3-x_2)$ is the composition 
\[D_{\bA^3}(d_3^2-d_2^2)\xrightarrow{g_3}\widehat{ Z}_3^0\xrightarrow{f_3}  Z_3^0\stackrel{\iota_3}{\hookrightarrow} D_{\bA^4}(x_3-x_2).\] 
Using the definitions of these maps, we can compute it explicitly as 
\[(c,d_2,d_3)\mapsto (d_2^2-2c,d_2^3-3cd_2,d_3^2-2c,d_3^3-3cd_3).\] This extends to a morphism $p:\bA^3\to \bA^4$. This extension is finite, hence proper, because $d_2$ is a zero of the monic polynomial $t^3-3t(d_2^2-2c)+2(d_2^3-3cd_2)=0$ and $d_3$ is a zero of the analogous polynomial. By Lemma~\ref{PullbackAn}, the pushforward on Chow is given by multiplication by $p_*(1)$, under the identification of the pullback $p^*$. Because $p|_{D_{\bA^3}(d_3^2-d_2^2)}:D_{\bA^3}(d_3^2-d_2^2)\to  Z_3^0$ is birational, $p_*(1)=[ Z_3]$, where $ Z_3$ is defined to be the closure of $ Z_3^0$ in $\bA^4$. $ Z_3^0$ is defined by $-4A^3-27B^2=0$, but this equation has denominators, so we cannot say the same equation defines $ Z_3$. Instead, we have $ Z_3=V((x_3-x_2)^3(-4A^3-27B^2))$, as we need three factors of $x_3-x_2$ to clear the denominators. The weight of $(x_3-x_2)^3(-4A^3-27B^2))$ is $-18$, so the image of $p_*$ is $(18\lambda)$. Using the commutative diagram
\begin{center}
    \begin{tikzcd}
        \CH_{\Gm}(\bA^3)\arrow[r,"p_*"]\arrow[d] & \CH_{\Gm}(\bA^4)=\Lambda\arrow[d]\\
        \CH_{\Gm}(D_{\bA^3}(d_3^2-d_2^2)) \arrow[r]& \CH_{\Gm}(D_{\bA^4}(x_3-x_2))={\Lambda}/{(2\lambda)},
    \end{tikzcd}
\end{center}
we have that the bottom map is $0$, because multiples of $18\lambda$ go to $0$ under the right map and the left map is surjective.

The map $C_3^3\to D_{\bA^4}(x_3-x_2)$ is the composition $C_3^3\hookrightarrow\widehat{ Z}_3^0\xrightarrow{f''}  Z_3^0\stackrel{\iota_3}{\hookrightarrow} D_{\bA^4}(x_3-x_2)$. Using the definitions of these maps, we can compute it explicitly as $(x_2,y_2,c)\mapsto (x_2,y_2,c,0)$, using our identification 
\[C_3^2\cong \Spec(k[x_2,y_2,c]_{x_2-c}/(y_2^2-x_2^3+3c^2x_2-2c^3)).\]
Note that this is just the space $W$. Composing the isomorphism $D_{\bA^2}(d^2-2c)=g^{-1}(W)\to W$ with the above $C_3^3\to D_{\bA^4}(x_3-x_2)$, we get
\[(c,d)\mapsto (d^2-2c,d^3-3cd,c)\mapsto (d^2-2c,d^3-3cd,c,0).\]
This extends to a finite map $q:\bA^2\to \bA^4$. The image of this map is $V(y_3,y_2^2-x_2^3+3x_3^2x_2-2x_3^3)$, which has fundamental class $-18\lambda^2$. A similar argument to the last paragraph using this extension $q$, Lemma~\ref{PullbackAn}, and a commutative diagram shows that the pushforward of $C_3^3\to D_{\bA^4}(x_3-x_2)$ is $0$, since $-18\lambda^2$ gets killed in $\CH_{\Gm}(D_{\bA^4}(x_3-x_2))=\Lambda/(2\lambda)$. Analogously, the pushforward of $C_3^2\to D_{\bA^4}(x_3-x_2)$ is $0$, so we indeed have $\iota_{3*}=0$.

Now, we compute 
\[\CH_{\Gm}( Z_3^0)\cong \CH_{\Gm}(D_{\bA^3}(d_3^2-d_2^2))\coprod_{\CH_{\Gm}(g_3^{-1}C')} \CH_{\Gm}(C').\]
Note $D_{\bA^3}(d_3^2-d_2^2)\subseteq D_{\bA^3}(d_3-d_2)$. Because $d_3-d_2$ has weight $-1$, the localization exact sequence and Lemma~\ref{PullbackAn} give $\CH(D(d_3-d_2))=\Z=\Lambda/(\lambda)$. Then another localization exact sequence gives a surjection $\CH_{\Gm}(D(e-f))\twoheadrightarrow \CH_{\Gm}(D_{\bA^3}(d_3^2-d_2^2))$, so $\CH_{\Gm}(D_{\bA^3}(d_3^2-d_2^2))=\Lambda/(\lambda)$. Furthermore, $C_3^j\cong W\cong g^{-1}(W)=D_{\bA^2}(d^2-3c)$, and $d^2-3c$ has weight $-2$, so by the localization exact sequence and Lemma~\ref{PullbackAn}, we have $\CH_{\Gm}(C_3^j)=\Lambda/(2\lambda)$.

Now, $g_3^{-1}(C_3)$ is a closed subscheme of $D_{\bA^3}(d_3^2-d_2^2)$, so the pushforward on $\CH_{\Gm}$ must land in degrees $\geq 1$. But, as we just saw, $\CH_{\Gm}(D_{\bA^3}(d_3^2-d_2^2))=\Z$, which is concentrated in degree $0$, so the pushforward is the zero map.  To compute the pushforward of the other map $g_3^{-1}(C_3)\to C_3$, recall $C_3=C_3^2\cup C_3^3$ and $g_3^{-1}(C_3)=g_3^{-1}(C_3^2)\cup g_3^{-1}(C_3^3)$. It suffices to just compute the pushforward $g_3^{-1}(C_3^\ell)\to C_3^\ell$. Note that 
\[\Spec(k[d_2,d_3]_{d_3^2-d_2^2})\cong g_3^{-1}(C_3^3)\xrightarrow{g_3} C_3^3\cong \Spec\left(\frac{k[x_2,y_2,c]_{x_2-c}}{(y_2^2-x_2^3+3c^2x_2-2c^3)}\right)\]
\[(d_2,d_3)\mapsto (d_2^2-\frac{2}{3}d_3^2,d_2^3-d_3^2d_2,\frac{d_3^2}{3})\]
factors as 
\[g_3^{-1}(C_3^3)\xrightarrow{h} D_{\bA^2}(d^2-3c)\to C_3^3\]
\[\quad \quad (d_2,d_3)\mapsto (\frac{d_3^2}{3},d_2) \quad \quad \quad \quad \quad \;\;\]
\[\quad \quad \quad \quad  \quad \quad \quad \quad \quad \quad \quad \quad(c,d) \mapsto (d^2-2c,d^3-3cd,c).\]

Moreover, this second map is precisely the map $g|_{g^{-1}(W)}$, which we know to be an isomorphism. Similar to our computation that $\CH_{\Gm}(D_{\bA^3}(d_3^2-d_2^2))=\Lambda/(\lambda)$ in the previous paragraph, we have 
\[\CH_{\Gm}(g_3^{-1}(C_3^3))=\CH_{\Gm}(D_{\bA^2}(d_3^2-d_2^2))=\Lambda/(\lambda),\]
and 
\[\CH_{\Gm}(D_{\bA^2}(d^2-3c))=\Lambda/(2\lambda).\]
We can compute $h_*(1)=2$, since $h$ is a degree $2$ map between varieties of the same dimension. Because $\CH_{\Gm}(g_3^{-1}(C_3^3))=\Z$, this determines $h_*$. 

Now, we have
\begin{align*}
       \CH_{\Gm}(\widehat{ Z}_3^0)&=\CH_{\Gm}(D_{\bA^3}(d_3^2-d_2^2))\coprod_{\CH_{\Gm}(g_3^{-1}C')} \CH_{\Gm}(C')\\
        &= \frac{\Lambda\langle[\widehat{ Z}_3^0] \rangle}{\langle\lambda [\widehat{ Z}_3^0]\rangle}\coprod_{\frac{\Lambda\langle [g_3^{-1} C_3^2]\rangle}{\langle\lambda[g_3^{-1} C_3^2]\rangle}\oplus\frac{\Lambda\langle [g_3^{-1} C_3^3]\rangle}{\langle\lambda[g_3^{-1} C_3^3]\rangle}} \frac{\Lambda\langle [C_3^2]\rangle}{\langle 2\lambda[C_3^2]\rangle}\oplus \frac{\Lambda\langle [C_3^3]\rangle}{\langle 2\lambda[C_3^3]\rangle}\\
        &= \frac{\Lambda\langle[\widehat{ Z}_3^0] \rangle}{\langle\lambda [\widehat{ Z}_3^0]\rangle}\oplus  \frac{\Lambda\langle [C_3^2]\rangle}{\langle 2[C_3^2]\rangle}\oplus \frac{\Lambda\langle [C_3^3]\rangle}{\langle 2[C_3^3]\rangle}
    \end{align*}
Now, using the isomorphism $f_{3*}: \CH_{\Gm}(\widehat{ Z}_3^0)\to \CH_{\Gm}({ Z}_3^0)$, we get the statement of the proposition, using that 
\[f_{3*}([C_3^j])=[f_{3}(C_3^j)]=[V(y_j,A+3x_j^2)].\qedhere\]
\end{proof}

\begin{thm}\label{M130}
     The Chow ring of $\M_{1,3}^0$ is given by
    \[\CH(\M_{1,3}^0)=\frac{\Lambda}{(2\lambda)}=\frac{\Z[\lambda]}{(2\lambda)},\]
    and the first higher Chow group of $\M_{1,3}^0$ is given by
    \[\overline\CH(\M_{1,3}^0,1)=\frac{\Lambda \langle \mathfrak g, \mathfrak p_{12}^0, \mathfrak p_{13}^0\rangle}{\langle \lambda \mathfrak g,2\mathfrak p_{12}^0, 2\mathfrak p_{13}^0\rangle}\]
    where 
    \[\mathfrak g:= \left[\M_{1,3}^0,\frac{(x_3-x_2)^6}{ \Delta}\right]\]
    \[\Delta:=-4A^3-27B^2\]
    \[\mathfrak p_{jk}^0:=\mathfrak p_{jk}|_{\M_{1,3}^0}.\]
    Finally, $\M_{1,3}^0$ has the MKP.
\end{thm}
\begin{proof}
    We saw 
    \[\CH_{\Gm}(D_{\bA^4}(x_3-x_2))=\frac{\Lambda
    }{(2\lambda)}\] 
    and 
    \[\overline\CH_{\Gm}(D_{\bA^4}(x_3-x_2),1)=0.\] 
    So the localization exact sequence for $\M_{1,3}^0\subseteq [D_{\bA^4}(x_3-x_2)/\Gm]$ is 
    \[0\to\overline\CH(\M_{1,3}^0,1)\to\CH( Z_3^0)\xrightarrow{\iota_{3*}}\frac{\Lambda}{(2\lambda)}\to\CH(\M_{1,3}^0)\to 0.\]
    Lemma~\ref{Z30} says that $\iota_{3*}=0$, giving
    \[\CH(\M_{1,3}^0)=\frac{\Lambda}{(2\lambda)}\]
    and
    \[\overline\CH(\M_{1,3}^0,1)\stackrel{\partial_1} \cong \CH( Z_3^0)=\frac{\Lambda \langle [ Z_3^0], \rho_2, \rho_3 \rangle}{\langle \lambda [ Z_3^0], 2\rho_2, 2\rho_3\rangle}.\]
    Thus, to get our claimed description, it suffices to show $\partial_1(\mathfrak g)=-[ Z_3^0]$ and $\partial_1(\mathfrak p_{1j})=\rho_j$.
    
    Using Proposition~\ref{formalsum}, we have $\mathfrak g$ is a well-defined element in $\overline\CH(\M_{1,3}^0,1)$, and 
    \[\partial_1(\mathfrak g)=\divisor\left(\frac{(x_3-x_2)^6}{ \Delta}\right)=-[ Z_3^0].\]
    Moreover, we have
    \[\partial_1(\mathfrak p_{12}^0)=\divisor\left(\frac{A+3x_2^2}{4A+3x_2^2}\right)=[V(y_2,A+3x_2^2)]-[V(y_2,4A+3x_2^2)].\]
    Now $[V(y_2,4A+3x_2^2)]=0$ in $\CH( Z_3^0)$. To see this, recall $\CH_{\Gm}(D_{\bA^4}(d_3^2-d_2^2)=\Z$ from the proof of Lemma~\ref{Z30}, meaning 
    \[[V(d_2)]=0\in \CH_{\Gm}(D_{\bA^4}(d_3^2-d_2^2).\]
    And then we have
    \[[V(y_2,4A+3x_2^2)]=[(g_3\circ f_3)(V(d_2))]=(g_3\circ f_3)_*([V(d_2)])=(g_3\circ f_3))_*(0)=0.\]
    Thus, we have $\partial_1(\mathfrak p_{12})=\rho_2$, and then applying the $S_2$ action, we have $\partial_1(\mathfrak p_{13})=\rho_3$.

    Finally, we argue that $\M_{1,3}^0$ has the MKP. By Proposition~\ref{MKPprops}(1), it suffices to show that $[ Z_3^0/\Gm]$ and $[D_{\bA^4}(x_3-x_2)/\Gm]$ have the MKP. The latter has the MKP by Proposition~\ref{MKPprops}(1,3,5), as it is the complement of an $[\bA^3/\Gm]$ in a $[\bA^4/\Gm]$. The stack $[ Z_3^0/\Gm]$ will have the MKP if $[\widehat{ Z}_3^0/\Gm]$ does, by Proposition~\ref{MKPprops}(4), and $[\widehat{ Z}_3^0/\Gm]$ will have the MKP if $[W_3/\Gm]$ and $[C_3/\Gm]$ do by Proposition~\ref{MKPprops}. 

    As noted, $W_3\cong g_3^{-1}(W_3)=D_{\bA^4}(d_3^2-d_2^2)$. So $[W_3/\Gm]$ is isomorphic to an open substack inside $[\bA^4/\Gm]$ with complement given by a union of two stacks isomorphic to $[\bA^3/\Gm]$ intersecting at a $[\bA^2/\Gm]$. Thus, $[W_3/\Gm]$ has the MKP by Proposition~\ref{MKPprops}(1,3,5). Moreover, $C_3=C_3^2\amalg C_3^3$, so it suffices to show $[C_3^j/\Gm]$ has the MKP. This is true because we have $C_3^j\cong W_2=D_{\bA^2}(d^2-c)$, and $[D_{\bA^2}(d^2-c)/\Gm]$ has the MKP because it is isomorphic to the complement of a $[\bA^1/\Gm]$ in a $[\bA^2/\Gm]$, using Proposition~\ref{MKPprops}.
\end{proof}

\section{$\M_{1,4}^0$}
\begin{definition}
    Define the open substack $\M_{1,4}^0\subseteq \M_{1,4}$ to parameterize curves $(C,p_1,p_2,p_3,p_4)$ so that $\OO(p_1+p_4-p_2-p_3)$ is nontrivial.  
\end{definition}
\begin{lemma}\label{M140Complement}
    The closed substack $\M_{1,4}\setminus \M_{1,4}^0\subseteq \M_{1,4}$ is isomorphic to $\M_{1,3}^0$ via restriction of the map $\pi:\M_{1,4}\to \M_{1,3}$ forgetting the fourth marked point. The inverse is given by
    \[\M_{1,3}^0\to \M_{1,4}\]
    \[(C,p_1,p_2,p_3)\mapsto (C,p_1,p_2,p_3,p_2\oplus p_3).\]
\end{lemma}
The proof is similar to that of Lemma~\ref{M130Complement}.

We now describe $\M_{1,4}^0$ as an open subsvariety of a projective space. The idea is similar to what was done in \cite{Belorousski}, but instead of $\bP^2$,  we use $\bP^1\times \bP^1$ as our ambient space, as was done in \cite{CL1}. This makes the computation of the Chow ring of $\M_{1,4}$ easier, as we have broken $\M_{1,4}$ into two pieces, $\M_{1,4}^0$ and $\M_{1,4}\setminus \M_{1,4}^0$ that we can work with, whereas working with curves in $\bP^2$ breaks $\M_{1,4}$ into three manageable pieces.

Define the points
\begin{align*}
    & P_1:=([1:0],[0:1])\\
    & P_2:=([0:1],[1:0])\\
    & P_3:=([1:1],[1:0])\\
    & P_4:=([1:0],[1:1])
\end{align*}
in $\bP^1\times \bP^1$ and $V$ be the $5$-dimensional subspace of $H^0(\bP^1\times \bP^1, \OO(2,2))$ consisting of curves passing through $P_1,P_2,P_3,P_4$. Using the coordinates $([x:w],[y:z])$ for a point in $\bP^1\times \bP^1$, We can write $V$ as the set of sections of the form
\begin{equation}\label{family}
    A_0x^2y^2-A_0x^2yz-A_0xwy^2+A_4xwyz+A_5xwz^2+A_7w^2yz+A_8w^2z^2.
\end{equation}
This induces a family of arithmetic genus $1$ curves over $\bP(V)$. Let $Z_4\subseteq \bP(V)$ be the locus of singular curves.

\begin{prop}\label{M140presentation}
   The family of curves \eqref{family} induces an isomorphism $\bP(V)\setminus Z_4\xrightarrow{\sim} \M_{1,4}^0.$
\end{prop}
\begin{proof}
This family induces a morphism $\bP(V)\setminus Z_4\to \M_{1,4}$ using the ``constant'' sections $P_1,P_2,P_3,P_4$. We claim that the image of this morphism is contained in $\M_{1,4}^0$. It suffices to show that given any field extension $L/k$ and a curve of the form \eqref{family} with coefficients in $L$, $\OO(P_1+P_4-P_2-P_3)$ is nontrivial. To see this, consider the point
\[Q:=([A_8:-A_5],[0:1]).\]
This always lives on curves of this form. Now, we have $\OO(P_1+Q)\cong \OO(P_2+P_3)$ because $P_1+Q$ is the pullback of the divisor $[0:1]\in \bP^1$ and $P_2+P_3$ is the pullback of the divisor $[1:0]\in \bP^1$. By inspection, $P_4\neq Q$, so we have $\OO(P_1+P_4-P_2-P_3)$ is nontrivial, using that $R\mapsto \OO(R)$ is injective. 

We now construct the inverse to this map $\bP(V)\setminus Z_4\to \M_{1,4}^0$. Given a curve $(C,p_1,p_2,p_3,p_4)$ over a scheme $S$, we consider the linear series $|\OO(p_1+p_4)|, |\OO(p_2+p_3)|$. Using Riemann-Roch, one can compute that these linear series have degree $2$, rank $2$, and are base-point-free. Thus, we have $|\OO(p_i+p_j)|$ induces an $S$-morphism $C\to \bP(\mathcal{V}_{ij})$ for a rank $2$ vector bundle $\mathcal V_{ij}$ on $S$. Then, the sections, $p_k$, of $C$ induce sections, $p_k^{ij}$ of $\bP(\mathcal V_{ij})$.

Thinking fiber-wise over $S$, we see that $ p_1^{14}, p_2^{14}$ and $ p_3^{14}$ are disjoint sections of $\bP(\mathcal V_{14})$, while $ p_1^{14}= p_4^{14}$, because $\OO(p_1+p_4)$ is not isomorphic to $\OO(p_2+p_3)$. Thus, $\bP( \mathcal V_{14})$ is a three pointed genus $0$ curve over $S$, so there is a unique isomorphism $\bP(\mathcal V_{14})$ with $\bP^1_S$ sending 
\[ p_1^{14}\mapsto [1:0]\]
\[ p_2^{14}\mapsto [0:1]\]
\[ p_3^{14}\mapsto [1:1].\]
Similarly, there is a unique isomorphism $\bP(\mathcal V_{23})$ with $\bP^1_S$ sending 
\[p_2^{23}=p_3^{23}\mapsto [1:0]\]
\[p_1^{23}\mapsto [0:1]\]
\[p_4^{23}\mapsto [1:1].\]

Now, we consider the map $C\to \bP(\mathcal V_{14})\times_S \bP(\mathcal V_{23}).$ Because each of the series has degree $2$, the image is a $(2,2)$ curve. Moreover, by the above, the sections $p_i$ become equal to the sections $P_i$, hence giving a map $S\to \bP(V)$. The projections $C\to \bP^1_S$ have degree $2$, and are not equal up to automorphism of $\bP^1_S$ because $\OO(p_1+p_4)$ is not isomorphic to $\OO(p_2+p_3)$, so we have that $C\to \bP^1\times \bP^1$ is birational. Then, the image has geometric genus $1$ and arithmetic genus $1$, so it must be smooth. Thus, we have produced a map $S\to \bP(V)\setminus Z_4$. This gives the inverse to our map $\bP(V)\setminus Z_4\to \M_{1,4}$.
\end{proof}

One can compute the defining equation for $Z_4$ using a computer algebra system. One takes the discriminant of a bidegree $(2,2)$ polynomial with indeterminate coefficients, plugs in our constraints on the coefficients, and takes the radical. Doing this, one obtains $Z_4=V(A_0F)$, where $F$ is a homogeneous polynomial of degree $9$ with $126$ terms.
Note that $V(A_0)\subseteq Z_4$. It is quick to compute how the (higher) Chow groups of $\bP(V)$ change when removing $V(A_0)$, as one then obtains affine space, but removing the rest of $Z_4$ is more involved. Define $Z_4^0:=Z_4\setminus V(a_0)$. Homogenizing with respect to $A_0$, we have $Z_4^0$ is a closed subvariety of $\bA^4=\Spec k[a_4,a_5,a_7,a_8]$, cut out by the homogenization of $F$,
\[f_0:=\frac{F}{A_0^9}.\]

\begin{lemma}\label{Z40}
    The Chow group of $Z_4^0$ is given by
    \[\CH(Z_4^0)=\frac{\Z\langle [Z_4^0],C_1,C_2,C_3,C_4,C_5,C_6\rangle}{\langle 2[C_1],2[C_2],2[C_3],2[C_4],2[C_5],2[C_6]\rangle},\]
    where 
    \[C_1=V(a_5,a_8)\]
\[C_2=V(a_7,a_8)\]
\[C_3=V(a_5+a_8,a_4+a_7-1))\]
\[C_4=V(a_7+a_8,a_4+a_5-1))\]
\[C_5=V(a_5-a_7,a_8+a_7^2+a_4a_7)\]
\[C_6=V(a_5a_7-a_8,a_4 + a_5 + a_7 - 1).\]
 Moreover, $Z_4^0$ has the MKP. 
\end{lemma}

\begin{proof}
    Define the morphism $u$ by
    \[(x,y,t)\mapsto (-4xy + 2x + 2y + 2t,2xy^2 - y^2 - 2yt,2x^2y - x^2 - 2xt,-x^2y^2 + 2xyt),\]
    and note that switching $x$ and $y$ preserves $a_4$ and $a_8$ and flips $a_5$ and $a_7$. Using a computer algebra system, one can check that $u$ lands in $Z_4^0$, so we can say $u:\bA^3\to Z_4^0$.
    
    We claim this map is finite. One can check that $x$ satisfies
    \[2x^3+(- 3a_4 - 6a_5)x^2 + (a_4^2 + 4a_5 - 2a_7 - 4a_8)x +
        (a_4a_7 + 2a_8)=0\]
    and is therefore integral over $Z_4^0$. Then $y$ must also be integral over $Z_4^0$ by symmetry. Additionally, because $a_4=-4xy+2x+2y+2t$, $t$ must be integral as well, and so $u$ is integral, hence finite. 
    
    Additionally, we have $x=\frac{b_1}{b_2}$, where 
        \begin{align*}
b_1&:=
a_4^4a_5a_7 + a_4^3a_5^2a_7 - a_4^3a_5a_7 + 2a_4^3a_5a_8 + a_4^3a_7a_8 +
    4a_4^2a_5^2a_7 +\\& 2a_4^2a_5^2a_8 + 6a_4^2a_5a_7^2 - 4a_4^2a_5a_7a_8 -
    2a_4^2a_5a_8 - 2a_4^2a_7a_8 + 2a_4^2a_8^2 \\&+ 4a_4a_5^3a_7 + 25a_4a_5^2a_7^2 -
    4a_4a_5^2a_7a_8 - 4a_4a_5^2a_7 + 8a_4a_5^2a_8 - 4a_4a_5a_7^2 - \\& 2a_4a_5a_7a_8
    - 8a_4a_5a_8^2 + 4a_4a_7^2a_8 - 4a_4a_7a_8^2 - 3a_4a_8^2 + 18a_5^3a_7^2 +
    8a_5^3a_8 -\\& 16a_5^2a_7^2 - 4a_5^2a_7a_8 - 8a_5^2a_8^2 - 8a_5^2a_8 +
    8a_5a_7^3 + 16a_5a_7^2a_8 + 8a_5a_7a_8 \\&+ 10a_5a_8^2 - 8a_7^2a_8 - 8a_7a_8^2
    - 8a_8^3,
    \end{align*}
    and
    \begin{align*}
        b_2&:=-a_4^5a_5 - a_4^4a_5^2 + a_4^4a_5 - a_4^4a_8 - 8a_4^3a_5^2 - 7a_4^3a_5a_7 +
    8a_4^3a_5a_8 + a_4^3a_8 -\\&
    8a_4^2a_5^3 - 32a_4^2a_5^2a_7 + 8a_4^2a_5^2a_8 +
    8a_4^2a_5^2 + 6a_4^2a_5a_7 - 38a_4^2a_5a_8 - 6a_4^2a_7a_8 +\\&
    8a_4^2a_8^2 -
    24a_4a_5^3a_7 - 16a_4a_5^3 +
    44a_4a_5^2a_7 - 40a_4a_5^2a_8 - 12a_4a_5a_7^2 -\\&
    44a_4a_5a_7a_8 - 16a_4a_5a_8^2 + 28a_4a_5a_8 + 4a_4a_7a_8 - 28a_4a_8^2 -
    16a_5^4 +\\& 16a_5^3a_7 - 16a_5^3a_8 + 16a_5^3 + 2a_5^2a_7^2 - 16a_5^2a_7a_8 -
    24a_5^2a_7 - 16a_5^2a_8^2 +\\& 8a_5^2a_8 + 8a_5a_7^2 + 20a_5a_7a_8 - 8a_5a_8^2
    - 8a_7^2a_8 - 24a_7a_8^2 - 16a_8^3 + 18a_8^2,
    \end{align*}

and by symmetry, we can write $y=\frac{c_1}{c_2}$, where $c_i:=b_i(a_4,a_7,a_5,a_8)$. Because $a_4=-4xy+2x+2y+2t$, we have a rational expression for $t$ as well. Thus, the map $u$ has a birational inverse, so $u$ is an isomorphism onto $Z_4^0$ away from the vanishing of the denominators. This gives that $Z_4^0$ is irreducible and $u$ is surjective. Let $C$ be the reduced subscheme of $V(b_2c_2)\subseteq Z_4^0$. Then Lemma~\ref{SplitChow} applied to $u:\bA^3\to Z_4^0$ says
\begin{align*}
    \CH(Z_4^0)&=\CH(\bA^3) \coprod_{\CH(p^{-1}(C))} \CH(C)\\
    &=\Z\langle [Z_4^0]\rangle \oplus \text{coker}(\CH(p^{-1}(C)\to \CH(C)),
\end{align*}
using the fact that the pushforward $p^{-1}(C)\to \bA^3$ is $0$. 

Pulling back $b_2c_2$ along $u$ and using a computer algebra system, we have 
\[\widetilde{C}:=p^{-1}(C)=\widetilde C_0\cup \widetilde C_1\cup \widetilde C_2\cup \widetilde C_3\cup \widetilde C_4\cup \widetilde C_5\cup \widetilde C_6\]
where
\begin{align*}
    \widetilde C_0&:=V(t^2-x(x-1)y(y-1))\\
    \widetilde C_1&:=V(y)\cong \bA^2\\
    \widetilde C_2&:=V(x)\cong \bA^2\\
    \widetilde C_3&:=V(x-1)\cong \bA^2\\
    \widetilde C_4&:=V(y-1)\cong \bA^2\\
    \widetilde C_5&:=V(2xy - x - y - 2t)\cong \bA^2\\
    \widetilde C_6&:=V(2xy - x - y - 2t+1)\cong \bA^2.
\end{align*}
For $\ell\geq 1$, one can routinely verify that 
\begin{itemize}
    \item $p(\widetilde C_j)=C_j,$ as defined in the statement,
    \item  $C_j\cong \bA^2$, and
    \item the degree of $\widetilde C_j\to C_j$ is $2$.
\end{itemize}

We will show that $\widetilde C_0\to C_0$ induces surjections on Chow groups. Assuming that for now, we compute $\CH(Z_4^0)$. We have the commutative diagram
\begin{center}
    \begin{tikzcd}
        \bigoplus_{\ell=0}^6 \CH_i(\widetilde C_j) \arrow[r,two heads]\arrow[d,"\bigoplus_\ell (p|_{\widetilde C_j})_*"]& \CH_i(\widetilde C)\arrow[d,"p_*"]\\
        \bigoplus_{\ell=0}^6 \CH_i( C_j) \arrow[r, two heads]& \CH_i( C).
    \end{tikzcd}
\end{center}
Because the horizontal maps are isomorphisms for $i=2$ and the degree of $\widetilde C_j\to C_j$ is $2$ for $\ell\geq 1$, we get the cokernel of $p_*$ is generated by the classes $[C_j]$, $\ell\geq 1$, with each class being $2$-torsion. In degrees $i<2$, note $\CH_i(\widetilde C_j)=\CH_i(C_j)=0$ for $\ell\geq 1$, and $\CH_i(\widetilde C_0)\to \CH_i(C_0)$ is surjective, so the cokernel of $p_*$ is $0$ in all other degrees.

Now, we need that $\widetilde C_0\to  C_0:=p(\widetilde C_0)$ induces a surjection on Chow groups. The map $p: \widetilde C_0\to C_0$ is birational: we can describe a rational inverse by noting $x=\frac{d_1}{d_2}$, where
\[d_1:=a_4^3 +
        2a_4^2a_5 + 4a_4a_5 + 4a_4a_7 - 4a_4a_8 + 8a_5^2 - 4a_5a_7 - 8a_5a_8 +
        12a_8\]
\[d_2:=2a_4^2 + 24a_4a_5 + 24a_5^2 - 16a_5 + 8a_7 + 16a_8,\]
which then gives $y=\frac{e_1}{e_2}$ where $e_i:=d_i(a_4,a_7,a_5,a_8)$ and we can get an expression for $t$ using $a_4=-4xy+2x+2y+2t$. 
Let $D$ be the reduced subscheme of $V_{p(C_0)}(d_2e_2)$. A computation with a compute algebra system shows that 
\[\widetilde D:=p^{-1}(D)=\widetilde D _1\cup \widetilde D_2\cup \widetilde D_3\cup \widetilde D_4\cup \widetilde D_5\cup \widetilde D_6,\]
where
\begin{align*}
  \widetilde D _1& :=V_{\bA^3}(x+y-1,y^2-y+t)\cong \bA^1\\  
  \widetilde D _2&:=V_{\bA^3}(x-y,y^2-y-t)\cong \bA^1\\
  \widetilde D_3&:=V_{\bA^3}(t,x)\cong \bA^1\\
  \widetilde D_4&:=V_{\bA^3}(t,x-1)\cong \bA^1\\
  \widetilde D_5&:=V_{\bA^3}(t,y)\cong \bA^1\\
  \widetilde D_6&:=V_{\bA^3}(t,y-1)\cong \bA^1.
\end{align*}
It is routine to verify that $\widetilde D_\ell\to D_\ell:=p(\widetilde D_\ell)$ has degree $1$ for each $\ell$. Because all of these maps have degree $1$, we have 
\[\Z^6=\CH_1(\widetilde D)\to \CH_1(D)=\Z^6\]
is an isomorphism, hence surjective. Next, because $\widetilde D_\ell$ is a rational curve for each $\ell$, we know $D_\ell$ is also rational. Then, because $\widetilde D_\ell$ and $D_\ell$ are affine, we have $\CH_0(\widetilde D_\ell)=\CH_0(D_\ell)=0$ for all $\ell$, so $\CH_0(\widetilde D)\to \CH_0(D)$ is surjective. 

Lemma~\ref{SplitChow} says that $\CH(\widetilde C_0)\oplus \CH(D)\to \CH(C_0)$ is surjective. But because $\CH(\widetilde{D})\to \CH(D)$ is surjective, we actually have $\CH(\widetilde C_0)\to \CH(C_0)$ is surjective.

Finally, we argue $Z_4^0$ as the MKP. By Proposition~\ref{MKPprops}(1), it suffices to show that $C$ and $Z_4^0\setminus C\cong \bA^3\setminus \widetilde C$ have the MKP. By Proposition~\ref{MKPprops}(3) $\bA^3$ has the MKP, so we just need to show that $C$ and $\widetilde C$ have the MKP. For $\ell\geq 1$, we have $\widetilde C_\ell \cong \widetilde C_\ell \cong \bA^2$, which has the MKP. Moreover, for $\ell\geq 1$ the intersections between the $\widetilde C_\ell$ and between the $C_\ell$ are still affine spaces, so 
\[\bigcup_{\ell=1}^6 \widetilde C_\ell \text{ and} \bigcup_{\ell=1}^6 C_\ell\]
have the MKP by Proposition~\ref{MKPprops}(1).

By the above, $p:\widetilde C_0\to C_0$ is a universal separable homeomorphism. Because of this, if we can show that $\widetilde C_0$ and its intersections with the $\widetilde C_\ell$ have the MKP, we know that $C_0$ and its intersections with the $C_\ell$ have the MKP by Proposition~\ref{MKPprops}(4). These intersections are affine spaces, so they have the MKP. To see that $\widetilde C_0$ has the MKP, we consider the map 
\[\Spec(k[X,Y,Z]/(Y^2-XZ))_{(X-1)(Z-1)}\to \widetilde C_0\]
\[(X,Y,Z)\mapsto \left(\frac{1}{1-X},\frac{1}{1-Z},\frac{Y}{(1-X)(1-Z)}\right).\]
This is an open embedding onto the complement of $V_{\widetilde C_0}(xy)=\{(0,0,0)\}$, as a rational inverse given by
\[(x,y,t)\mapsto \left(1-\frac{1}{x},\frac{t}{xy},1-\frac{1}{y}\right).\]
The quadratic cone $\Spec(k[X,Y,Z]/(Y^2-XZ))$ has the MKP, because removing a line gives something isomorphic to the plane minus a line, using Proposition~\ref{MKPprops}(1). Therefore, after removing the point $(X,Y,Z)=(1,1,1)$, this space still has the MKP, so the above open embedding gives $\widetilde C_0\setminus (1,1,0)$ has the MKP. Thus, $\widetilde C_0$ has the MKP. 
\end{proof}

\begin{thm}\label{M140}
The Chow ring of $\M_{1,4}^0$ is given by 
\[\CH(\M_{1,4}^0)=\Z\] 
and the indecomposable first higher Chow group of $\M_{1,4}^0$ is given by 
    \[\overline\CH(\M_{1,4}^0,1)=\frac{\Z\langle \mathfrak h, \mathfrak q_1,\mathfrak q_2,\mathfrak q_3,\mathfrak q_4,\mathfrak q_5,\mathfrak q_6\rangle}{\langle 2 \mathfrak q_1,2\mathfrak q_2,2\mathfrak q_3,2\mathfrak q_4,2\mathfrak q_5,2\mathfrak q_6\rangle},\]
    where 
    \[\mathfrak h=[\M_{1,4}^0,f]\]
    and $\mathfrak q_i$ is such that $\partial_1(\mathfrak q_i)=[C_i]$. 
Moreover, $\M_{1,4}^0$ has the MKP. 
\end{thm}
We could give a more explicit description of the $\mathfrak q_i$ using the description of $C_i$ above, such as 
\[\mathfrak q_1=[V(a_5),a_8],\]
but we do not end up needing such a description later on.
\begin{proof}
    By homotopy invariance of higher Chow groups, we have
    \[\overline\CH(\bA^4,1)=\overline\CH(\Spec(k),1)=0,\]
    so the localization exact sequence for $\M_{1,4}^0\subseteq \bA^4$ reads
    \[0\to \overline\CH(\M_{1,4}^0,1)\xrightarrow{\partial_1} \CH(Z_4^0)\xrightarrow{\iota_*} \CH(\bA^4)\to \CH(\M_{1,4}^0)\to 0.\]
    By considering the grading, we have $\iota_*=0$, implying that 
    \[\CH(\bA^4)\to \CH(\M_{1,4}^0)\]
    and
    \[\overline\CH(\M_{1,4}^0,1)\xrightarrow{\partial_1} \CH(Z_4^0)\]
     are isomorphisms. Thus, we have $\CH(\M_{1,4}^0)=\Z$ and 
     \[\CH(\M_{1,4}^0,1)\stackrel{\partial_1}\cong \CH(Z_4^0)=\frac{\Z\langle [C_1],[C_2],[C_3],[C_4],[C_5],[C_6]\rangle}{\langle 2[C_1],2[C_2],2[C_3],2[C_4],2[C_5],2[C_6]\rangle}.\]
     Now, we have 
     \[\partial_1(\mathfrak h)= \divisor(f)=[Z_4^0],\]
     so letting $\mathfrak q_j$ be the element mapping to $[C_j]$, we get the description of higher Chow from the statement of the theorem.
     
     Finally, by Lemma~\ref{Z40}, $Z_4^0$ has the MKP, and by Proposition~\ref{MKPprops}(3), $\bA^4$ has the MKP, so by Proposition~\ref{MKPprops}(1), $\M_{1,4}$ has the MKP.
\end{proof}

We do not get an action of $S_4$ on $\M_{1,4}^0$ because the definition is not symmetric,  but we do get an action of $D_4:=\langle (23),(12)(34)\rangle\subseteq S_4$ on $\M_{1,4}^0$. It is given by
\[(12)(34)\cdot (a_4,a_5,a_7,a_8)=(a_4,a_7,a_5,a_8)\]
\[(23)\cdot (2-a_4,-a_5,-1+a_4+a_7,a_5+a_8)\]
This action extends to $\bA^4\supseteq \M_{1.4}^0$, and therefore also gives an action on $Z_4^0$. The localization exact sequence used above implies 
\[\overline\CH^2(\M_{1,4}^0,1)\cong \CH^1(Z_4^0)\]
as $D_4$-modules. Thus, we can compute the action of $D_4$ on $\overline\CH^2(\M_{1,4}^0,1)$ by computing the action of $D_4$ on $\{C_1,C_2,C_3,C_4,C_5,C_6\}$.

\begin{lemma}\label{D4action}
    The action of $D_4$ on $\overline\CH^2(\M_{1,4}^0,1)$ is given by
	\[\sigma \mathfrak q_i= \mathfrak q_{\sigma(i)}, i\in \{1,2,3,4\}\]
	\[(14)\mathfrak q_5=\mathfrak q_6, (12)(34) \mathfrak q_5=\mathfrak q_5\]
	\[(14)\mathfrak q_6=\mathfrak q_6, (12)(34) \mathfrak q_6=\mathfrak q_6.\]
\end{lemma}

\begin{rmk}
    It is not clear what the morphism $\M_{1,4}^0\dashrightarrow \M_{1,3}^0$ looks like in the $a_i$ coordinates. One can work out that the morphsim sends $(a_4,a_5,a_7,a_8)$ to
    {\small\[ (3a_4^2 + 12(a_5 + a_7 - a_8),108(a_8-a_5a_7),3a_4^2 +12(a_7- 2a_5 - a_8),-108(a_4a_5 +a_5a_7 +a_8)).\]}
    Using this and Proposition~\ref{formalsum}(2), one can obtain
    \[\mathfrak p_{12}|_{\M_{1,4}^0}=\mathfrak q_1+\mathfrak q_2+\mathfrak q_6.\]
    Applying the $D_4$ action and using Corollary~\ref{pijpjk}, one can get formulas for all $\mathfrak p_{jk}|_{\M_{1,4}^0}$. 
\end{rmk}

\section{Toward $\overline\CH(\M_{1,3},1)$ and $\overline\CH(\M_{1,4},1)$}
In this section we patch together information from the previous section to partially compute the indecomposable first higher Chow groups of $\M_{1,3}$ and $\M_{1,4}.$ Presentations of the Chow rings of these stacks fall out from these computations.
\subsection{$\bar\M_{1,1}$}
Throughout the rest of the paper, we will use the following straightforward fact on how to obtain a presentation for the middle term of a short exact sequence given presentations for the outer terms.
\begin{lemma}\label{exactPresentation}
    Suppose $A$ is a ring and we have a short exact sequence
    \[0\to L\to M\to N\to 0\]
    of $A$-modules with presentations 
    \[L=\frac{A\langle S\rangle}{\langle R\rangle}\quad N=\frac{A\langle T\rangle}{\langle U\rangle}\] 
    and a set of lifts $\widetilde T$ of each $t\in T$ to $\tilde t\in M$. For $p(t)\in U$, let $q_p(s)$ be the unique expression in the $s\in S$ so that $p(\tilde{t})=q_p(s)$ in $M$. Then, we have a presentation 
    \[M=\frac{A\langle S\cup \widetilde T\rangle}{\langle R\cup \widetilde U\rangle},\] where $\widetilde U=\{p(\tilde t)-q_p(s)|p\in U\}$. 
\end{lemma}
We use consequences of the computation of $\bar\M_{1,1}$ in our computations for $\M_{1,3}$ and $\M_{1,4}$, so we present this now, instead of in section 8.

\begin{thm}\label{barM11}
    The Chow ring of $\bar\M_{1,1}$ is given by 
    \[\CH(\bar\M_{1,1})=\frac{\Lambda}{(24\lambda^2)}=\frac{\Z[\lambda]}{(24\lambda^2)},\]
    the pullback of $\lambda$ to $B\mu_2$ is the generator of $\CH(B\mu_2)$,
    and the indecomposable first higher Chow group is given by 
    \[\overline\CH(\M_{1,1},1)=0.\]
\end{thm}
\begin{proof}
    We have the localization exact sequence
    \[\overline\CH(\M_{1,1},1)\to \CH(\partial\bar\M_{1,1})\xrightarrow{\iota_*} \CH(\bar\M_1,1)\to \CH(\M_{1,1})\to 0.\]
    We have $\partial\bar\M_{1,1}^{\text {red}}\cong \CH(B\mu_2)$ because $\partial\bar\M_{1,1}$ parameterizes a single object, the nodal cubic curve with a marked point, which has automorphism group $\mu_2$. (In fact, $\partial\bar\M_{1,1}\cong B\mu_2$; this follows from Proposition~\ref{OpenStrata}.) Thus, we can conclude 
    \[\CH(\partial\bar\M_{1,1})=\CH(B\mu_2)=\frac{\Z[u]}{(2u)}\]
    by Lemma~\ref{Bmu}. Thus, our localization sequence is
    \[0\to \frac{\Z[u]}{(2u)}\xrightarrow{\iota_*} \CH(\bar\M_{1,1})\to\frac{\Z[\lambda]}{(12\lambda)}\to 0,\]
    using Theorem~\ref{M11}. We have been viewing $\CH(\M_{1,1})$ as a module over $\Lambda$, and we can do the same for $\CH(\bar\M_{1,1})$ using its first Chern class of the Hodge bundle. Moreover, by pulling back along $\iota^*$, we can also do this for $\CH(\partial\bar\M_{1,1}).$ 

    Next, we claim $\iota^*(\lambda)=u$. The Hodge bundle has fiber $H^0(C,\omega_C)$ over a point $(C,p)\in \bar\M_{1,1}$. Its pullback along $\iota$ is then the vector space  $H^0(C,\omega_C)$, where $C$ is the curve $y^2=x^3+x^2$, with the $\mu_2$ action induced by the automorphism $(x,y)\mapsto (x,-y)$. The global section of $H^0(C,\omega_C)$ is given by $dx/y$, and so $\mu_2$ acts by $-1$. Thus, we do indeed have $\iota^*(\lambda)=u$. Thus, we can write our localization exact sequence as 
    \[0\to \frac{\Lambda\langle \phi\rangle}{\langle 2\lambda\phi\rangle}\to \CH(\bar\M_{1,1})\to \frac{\Lambda\langle 1\rangle}{\langle 12\lambda\cdot 1 \rangle}\to 0,\]
    where we set $\phi=[\partial\bar\M_{1,1}]$ (which matches our later notation).
    
    We use Lemma~\ref{exactPresentation} to get a presentation for $\CH(\bar\M_{1,1})$ as a $\Lambda$-module. To do this, we need to understand for which $a\in \CH^0(\partial\bar\M_{1,1})=\Z\langle \phi \rangle$ do we have $\iota_*(a\phi)=12\lambda$. 
     Mumford's relation \cite{Mumford} says that 
     \[12\lambda=\iota_*(\phi)\]
     rationally, and hence this holds integrally up to torsion. This implies we have $\iota_*((a-1)\phi)$ is torsion. Because $\iota$ is injective, we need $a=1$, and so $12\lambda=\iota_*(\phi)$ integrally. Now Lemma~\ref{exactPresentation} gives
     \[\CH(\bar\M_{1,1})=\frac{\Lambda\langle 1,\phi \rangle}{\langle12\lambda\cdot 1-\phi, 2\lambda \phi\rangle}=\frac{\Lambda\langle 1\rangle}{\langle 24\lambda^2\cdot 1\rangle}.\]
    This says what $\CH(\bar\M_{1,1})$ is as a $\Lambda$-module. Because $1^2=1$, we understand $\CH(\bar\M_{1,1})$ as a $\Lambda$-algebra, and hence a ring, so we can write
    \[\CH(\bar\M_{1,1})=\frac{\Lambda}{(24\lambda^2)}=\frac{\Z[\lambda]}{(24\lambda^2)}.\qedhere\]
\end{proof}
\begin{rmk}
    It is known that $\bar\M_{1,1}\cong \mathbb{P}(4,6)$; using this gives an easier computation of $\CH(\bar\M_{1,1})$. We gave the above proof to be illustrative of our methods, and to go along with the philosophy of needing to know as little global geometry of $\bar\M_{g,n}$ to obtain the Chow ring. 
\end{rmk}

We get some immediate consequences for $\M_{1,n}$ and $\bar\M_{1,n}$ using the map $\pi: \bar\M_{1,n}\to \bar\M_{1,1}$ forgetting all but the first point.
\begin{prop}\label{Order 12}
     \mbox{}
     \begin{enumerate}
        \item The order of $\lambda^i\in \CH^i(\bar\M_{1,n})$ is $24$.
         \item For $\bar\M_{1,n}^{\Phi}\subseteq \bar\M_{1,n}$, the curves with a non-separating node, $[\bar\M^{\Phi}_{1,n}]=12\lambda$.
         
         \item If $\overline\CH^1(\M_{1,n},1)=0$, then $\lambda$ has order $12$ in $\CH^1(\M_{1,n})$.
         
     \end{enumerate}
\end{prop}
\begin{proof}
\mbox{}
    \begin{enumerate}
        \item Because $\lambda^i$ has order $24$ in $\CH(\bar\M_{1,1})$, the order of $\pi^*(\lambda^i)=\lambda^i\in \CH(\bar\M_{1,n})$ divides $24$. Moreover, the map $\pi$ has a section, $s$, given by attaching a fixed genus $0$ curve with $n+1$ points to each $(C,p)\in \bar\M_{1,1}$. Thus, $\lambda^i$ must have order exactly $24$. 
        \item We have a commutative diagram
    \begin{center}
        \begin{tikzcd}
            \bar\M^{\Phi}_{1,n} \arrow[r]\arrow[d,"\pi"] & \bar\M_{1,n}\arrow[d,"\pi"]\\
            \partial\bar\M_{1,1} \arrow[r,"\iota"] & \bar\M_{1,1}
        \end{tikzcd}
    \end{center}
    This is Cartesian. Set-theoretically, this is clear. To argue that the fiber product is reduced, it suffices to show that the preimage of a reduced substack is reduced under $\bar\M_{1,i+1}\to \bar\M_{1,i}.$ The fibers over the closed points are reduced, so this follows from \cite[\href{https://stacks.math.columbia.edu/tag/0C0E}{Tag 0C0E}]{stacks-project}.
    Note that the vertical maps are flat and the horizontal maps are proper. Then, we have
    \[12\lambda=\pi^*(12\lambda)=\pi^*(\iota_*(1))=\iota_*(\pi^*(1))=\iota_*(1)=[\bar\M^{\Phi}_{1,n}]\]
    by push-pull, giving (1).
    \item Note that the above implies $\lambda$ has order $12$ in $\CH^1(\bar\M_{1,n}\setminus \bar\M^{\Phi}_{1,n})$. If the order $\lambda$ is not $12$ in $\CH^1(\M_{1,n})$, then there must be some relation 
    \[a\lambda=\sum_\delta a_\delta \delta\]
    with $a\neq 0$ and some $a_\delta$ nonzero, where the sum runs over the classes $\delta=[\Delta]$ of the components of $\partial\bar\M_{1,n}$ besides $\bar\M^{\Phi}_{1,n}$. Multiplying by $12$ and taking a lift to $\CH^1(\bar\M_{1,n})$, we get a nontrivial relation between the boundary divisors of $\bar\M_{1,n}$. This cannot happen if $\overline\CH^1(\M_{1,n},1)=0$ by exactness of 
    \[\overline\CH^1(\M_{1,n},1)\to \CH^0(\partial\bar\M_{1,n})\to \CH^1(\bar\M_{1,n}).\qedhere\]
    \end{enumerate}    
\end{proof}
\begin{rmk}
    $\lambda$ always has order $12$ on every $\M_{1,n}$ \cite[Theorem B]{FF}.
\end{rmk}

\subsection{$\M_{1,3}$}

\begin{thm}\label{M13}
    \[\CH(\M_{1,3})=\frac{\Lambda}{(12\lambda,6\lambda^2)}=\frac{\Z[\lambda]}{(12\lambda,6\lambda^2)}\]
    and
    \[\overline\CH(\M_{1,3},1)=\frac{\Lambda \langle \mathfrak f,  \mathfrak p_{12}, \mathfrak p_{13},\mathfrak p_{23}\rangle}{\langle  t\mathfrak f, \lambda\mathfrak f, \{2\mathfrak p_{jk}\}_{jk}, \mathfrak p_{12}+\mathfrak p_{13}+\mathfrak p_{23}\rangle},\]
    where $t\in \Z$ is some integer and 
    \[\mathfrak f:=\left[\M_{1,2}^0, \frac{y_2^4}{\Delta}\right].\]
    Moreover, $\M_{1,3}$ has the MKP.
\end{thm}
\begin{proof}
    By Lemma~\ref{M130Complement}, we have that the complement of $\M_{1,3}^0$ in $\M_{1,3}$ is isomorphic to $\M_{1,2}^0$. We will use the localization exact sequence for $\M_{1,3}^0\subseteq \M_{1,3}$. 
    
    We first compute the map $\partial_1: \overline\CH(\M_{1,3}^0,1)\to \CH(\M_{1,2}^0)$. By Theorem~\ref{M130}, we have 
    \[\overline\CH(\M_{1,3}^0,1)=\frac{\Lambda \langle \mathfrak g, \mathfrak p_{12}^0, \mathfrak p_{13}^0\rangle}{\langle \lambda \mathfrak g,2\mathfrak p_{12}^0, 2\mathfrak p_{13}^0\rangle}.\]
    We have $\partial_1(\mathfrak p_{1j}^0)=0$ for $j=2,3$ because these elements are 2-torsion, and the group they are mapping into, $\CH^1(\M_{1,2}^0)$, is $3$-torsion by Theorem~\ref{M120}. So we just need to compute $\partial_1(\mathfrak g)$. Using Proposition~\ref{formalsum}(4), we have
    \[\partial_1(\mathfrak g)=\divisor\left(\frac{(x_3-x_2)^6}{\Delta}\right)=\ord_{\M_{1,2}^0}\left(\frac{(x_3-x_2)^6}{\Delta}\right)[\M_{1,2}^0].\]
    We compute this order of vanishing using Proposition~\ref{transverse}. Choose a curve $(C=V(y^2-x^3-ax-b),\infty,(x_2,y_2))\in U_2$ with $y_2\neq 0$, where $\infty=[0:1:0]$. We get a morphism 
    \[\varphi: C\setminus \{\infty,(x_2,y_2)\}\to \M_{1,3}\]
    \[(x,y)\mapsto (C,\infty,(x_2,y_2),(x,y)).\]
    This lands in $\M_{1,3}\setminus \M_{1,3}^0$ if and only if $(x,y)=(x_2,-y_2)$, using that $y_2\neq 0$. By Proposition~\ref{transverse}, we have 
    \begin{align*}
        \ord_{\M_{1,2}^0}\left(\frac{(x_3-x_2)^6}{\Delta}\right)&=\ord_{(x_2,-y_2)}\left(\varphi^\#\frac{(x_3-x_2)^6}{\Delta}\right)\\
        &=\ord_{(x_2,-y_2)}\left(\frac{(x-x_2)^6}{-4a^3-27b^2}\right)\\
        &=6,
    \end{align*}
    using that $(x-x_2)$ is a uniformizer for $C$ at $(x_2,-y_2)$, which is true because $y_2\neq 0$. 

    With this computation of $\partial_1$, we can separate the localization exact sequence for $\M_{1,3}^0\subseteq \M_{1,3}$ into the exact sequences
    \[0\to \frac{\CH(\M_{1,2}^0)}{\langle 6[\M_{1,2}^0]\rangle}\xrightarrow{\iota_*} \CH(\M_{1,3})\to \CH(\M_{1,3}^0)\to 0\]
    and 
    \[\overline\CH(\M_{1,2}^0,1)\xrightarrow{\iota_*} \overline\CH(\M_{1,3},1) \to\frac{\overline\CH(\M_{1,3}^0,1)}{\langle \mathfrak g \rangle}\to 0.\]
    Using Theorem~\ref{M120} and Theorem~\ref{M130}, we can write these exact sequences as
    \[0\to \frac{\Lambda\langle [\M_{1,2}^0]\rangle}{\langle 6[\M_{1,2}^0],3\lambda[\M_{1,2}^0]\rangle}\xrightarrow{\iota_*} \CH(\M_{1,3})\to \frac{\Lambda}{(2\lambda)}\to 0\]
    and 
    \[\frac{\Lambda\langle \mathfrak f\rangle}{\langle \lambda\mathfrak f\rangle}\xrightarrow{\iota_*} \overline\CH(\M_{1,3},1)\to \frac{\Lambda\langle\mathfrak p_{12}^0,\mathfrak p_{13}^0 \rangle}{\langle 2\mathfrak p_{12}^0,2\mathfrak p_{13}^0 \rangle}\to 0.\]

    To get the presentation for $\CH(\M_{1,3})$ as a $\Lambda$ module, we use Lemma~\ref{exactPresentation}. We know that we can write
    \[2\lambda=\iota_*(c[\M_{1,2}^0])\in \CH(\M_{1,3})\]
    for some $c\in \Lambda$. By considering degrees, we have $c\in \Z$, and is well defined modulo $6$. From this exact sequence describing $\overline\CH(\M_{1,3},1)$, we have that $\CH^1(\M_{1,3},1)=0$, so by Proposition~\ref{Order 12}(3), we know $\lambda\in \CH(\M_{1,3})$ has order $12$. Thus, we have $c=\pm 1$. Using Lemma~\ref{exactPresentation}, we have
    \[\CH(\M_{1,3})=\frac{\Lambda\langle 1,[\M_{1,2}^0]\rangle}{\langle 2\lambda\cdot 1- \pm[\M_{1,2}^0], 6[\M_{1,2}^0], 3\lambda[\M_{1,2}^0] \rangle}=\frac{\Lambda}{(12\lambda, 6\lambda^2)}.\]

    Now we get the presentation for $\overline \CH(\M_{1,3},1)$. We do not know about the kernel of $\iota_*: \overline\CH(\M_{1,2}^0)\to \overline\CH(\M_{1,3},1)$, but we do have that
    \[0\to \frac{\Lambda\langle \mathfrak f\rangle}{\langle t\mathfrak f,\lambda\mathfrak f\rangle}\xrightarrow{\iota_*} \overline\CH(\M_{1,3},1)\to \frac{\Lambda\langle\mathfrak p_{12}^0,\mathfrak p_{13}^0 \rangle}{\langle 2\mathfrak p_{12}^0,2\mathfrak p_{13}^0 \rangle}\to 0\]
    is exact for some $t\in \Z$. We claim that the exact sequence splits. For $i=2,3$, consider the diagram
    \begin{center}
        \begin{tikzcd}
            \M_{1,3}^0\arrow[r,"j"]\arrow[rd] & \M_{1,3}\arrow[d,"\pi_{1i}"]\\
            & \M_{1,2}
        \end{tikzcd}
    \end{center}
    where $j$ is the open embedding and $\pi_{1i}$ is as in Definition~\ref{pij}. 
    \[\frac{\Lambda\langle \mathfrak p_{12}^0, \mathfrak p_{13}^0\rangle}{\langle2\mathfrak p_{12}^0,2\mathfrak p_{13}^0 \rangle}\cong\overline\CH(\M_{1,2},1)^{\oplus 2}\xrightarrow{\pi_{12}^*\oplus\pi_{13}^*} \overline\CH(\M_{1,3},1)\xrightarrow{j^*}\overline\CH(\M_{1,3}^0,1)\]
    is the identity, giving a splitting of $j^*$. Thus, we have a presentation
    \[\overline\CH(\M_{1,3},1)=\frac{\Lambda \langle \mathfrak f,  \mathfrak p_{12}, \mathfrak p_{13}\rangle}{\langle  t\mathfrak f, \lambda\mathfrak f, 2\mathfrak p_{12}, 2\mathfrak p_{13}\rangle}. \]
    Now, we can write
    \[\mathfrak p_{23}=c_{12}\mathfrak p_{12}+c_{13}\mathfrak p_{23}\]
    for $c_{1j}\in \{0,1\}.$
    Because of the $S_3$ action, we see that we must have $\mathfrak p_{23}=\mathfrak p_{12}+\mathfrak p_{13}$, or $\mathfrak p_{12}+\mathfrak p_{13}+\mathfrak p_{23}=0$.

    Finally, by Theorem~\ref{M120} and Theorem~\ref{M130}, $\M_{1,2}^0$ and $\M_{1,3}^0$ have the MKP, so then so does $\M_{1,3}$ by Proposition~\ref{MKPprops}(1).
\end{proof}

By pulling back $\mathfrak p_{12}+\mathfrak p_{13}+\mathfrak p_{23}$ along a morphsim $\M_{1,n}\to \M_{1,3}$ forgetting all marking but $j,k,\ell$, one obtains the following corollary.
\begin{cor}\label{pijpjk}
    For $j,k,\ell\in [n]$ distinct, we have $\mathfrak p_{jk}+\mathfrak p_{j\ell}+\mathfrak p_{k\ell}=0\in\overline\CH^2(\M_{1,n},1)$.
\end{cor}

In Proposition~\ref{higherChowImM13}, we will see that $\mathfrak f=0$ in $\overline\CH^2(\M_{1,3},1)$.  To do this, we will use the following Proposition. Let $\M_{1,n}^\irr\subseteq \bar\M_{1,n}$ be the locus of irreducible curves.

\begin{prop}\label{irr}
     The group $\overline \CH^2(\M_{1,3}^\irr,1)$ is free over $\Z$ of rank $\leq 3$. 
\end{prop}
\begin{proof}
    Define
    \[\M_{1,3}^{0,\irr}:=\{(C,p_1p_2,p_3)\in \M_{1,3}^\irr|\OO(p_2+p_3-2p_1)\text{ is nontrivial}\}\]
    and 
    \[\M_{1,2}^{0,\irr}:=\{(C,p_1,p_2)\in \M_{1,2}^0| \OO(2p_2-2p_1)\text{ is nontrivial}\}.\]
    Arguing as in Lemma~\ref{M130Complement}, we have that $\M_{1,3}^\irr\setminus \M_{1,3}^{0,\irr}\cong \M_{1,2}^{0,\irr}$. We will use the localization exact sequence for $\M_{1,3}^{0,\irr}\subseteq \M_{1,3}^\irr$.

    We can modify our quotient stack presentations for $\M_{1,3}^0$ and $\M_{1,2}^0$ to get quotient stack presentations for $\M_{1,3}^{0,\irr}$ and $\M_{1,2}^{0,\irr}$ respectively. We simply allow the cubic curves to be nodal, but not cuspidal, and that the marked points are nonsingular (i.e. not equal to the unique node, if the curve is singular). Then, we can write $\M_{1,3}^{0,\irr}$ as the quotient of 
    \[\{(x_2,y_2,x_3,y_3)\in D_{\bA^4}(x_3-x_2)| (A,B)\neq (0,0), (2Ax_i,y_i)\neq (-3B,0)\}\]
    by $\Gm$. Thus, the complement of $\M_{1,3}^{0,\irr}$ in $D_{\bA^4}(x_3-x_2)$ has $3$ components, so the localization exact sequence gives
    \[0=\overline \CH_{\Gm}^2(D(x_3-x_2),1)\to \overline\CH^2(\M_{1,3}^{0,\irr},1)\to \Z^{3}.\]
    Thus, $\overline\CH^2(\M_{1,3}^{0,\irr},1)$ is free of rank $\leq 3$.
    
    Similarly, we can write $\M_{1,2}^{0,\irr}$ as the quotient of 
    \[U_2^{0,\irr}:=\{(a,x_2,y_2)\in \bA^3| (a,B)\neq (0,0), y_2\neq 0\}\]
    by $\Gm$. Then, using Lemma~\ref{units}$(3')$, we have
    \[\overline\CH^1(\M_{1,2}^{0,\irr},1)=(\OO(U_2^{0,\irr})^\times)^{\Gm}/k^\times=0.\]
    Then, the localization sequence for $\M_{1,3}^{0,\irr}\subseteq \M_{1,3}^\irr$ gives
    \[0=\overline\CH^1(\M_{1,2}^{0,\irr},1)\to \overline\CH^2(\M_{1,3}^\irr,1)\to \overline\CH^2(\M_{1,3}^{0,\irr},1).\]
    As $\overline\CH^2(\M_{1,3}^{0,\irr},1)$ is free of rank $\leq 3$, we can say the same for $\overline\CH^2(\M_{1,3}^\irr,1).$
\end{proof}

\subsection{$\M_{1,4}$}
\begin{thm}\label{M14}
     The Chow ring of $\M_{1,4}$ is given by
     \[\CH(\M_{1,4})=\frac{\Lambda}{(12\lambda,2\lambda^2)}=\frac{\Z[\lambda]}{(12\lambda,2\lambda^2)},\]
    and the indecomposable first higher Chow groups satisfy
    \begin{itemize}
        \item $\overline\CH^1(\M_{1,4},1)=0$,
        \item there is an exact sequence
        \[\Z\langle \mathfrak g\rangle \to \CH^2(\M_{1,4},1)\to \frac{\Z\langle  \mathfrak q_1,\mathfrak q_2,\mathfrak q_3,\mathfrak q_4,\mathfrak q_5,\mathfrak q_6\rangle}{\langle2 \mathfrak q_1,2\mathfrak q_2,2\mathfrak q_3,2\mathfrak q_4,2\mathfrak q_5,2\mathfrak q_6\rangle}\to 0,\]
        and
        \item $\overline\CH^{\geq 3}(\M_{1,4},1)$ is generated as a $\Lambda$-module by the $\lambda\mathfrak p_{jk}$ subject to only the relations
        \begin{align*}
            & 2\lambda \mathfrak p_{jk}\\
            & \lambda\mathfrak p_{jk}+\lambda\mathfrak p_{j\ell}+\lambda\mathfrak p_{k\ell}\\
            & \lambda \mathfrak p_{jk}+\lambda \mathfrak p_{\ell m}
        \end{align*}
        for $\{j,k,\ell,m\}=[4]$.
    \end{itemize}
     Moreover, $\M_{1,4}$ has the MKP.
\end{thm}
\begin{proof}
    By Lemma~\ref{M140Complement}, the complement of $\M_{1,4}^0$ in $\M_{1,4}$ is isomorphic to $\M_{1,3}^0$. We use the localization exact sequence for $\M_{1,4}^0\subseteq \M_{1,4}$. 

    We first compute the map $\partial_1: \overline\CH(\M_{1,4}^0,1)\to \CH(\M_{1,3}^0)$. By Theorem~\ref{M140}, we have
    \[\overline\CH(\M_{1,4}^0,1)=\frac{\Z\langle \mathfrak h, \mathfrak q_1,\mathfrak q_2,\mathfrak q_3,\mathfrak q_4,\mathfrak q_5,\mathfrak q_6\rangle}{\langle \lambda\mathfrak h, 2 \mathfrak q_1,2\mathfrak q_2,2\mathfrak q_3,2\mathfrak q_4,2\mathfrak q_5,2\mathfrak q_6\rangle}.\]

    Using Proposition~\ref{formalsum}(4), we have
    \[\partial_1(\mathfrak h)=\divisor(f)=\ord_{\M_{1,3}^0}(f)[\M_{1,3}^0].\]
    We compute this order of vanishing using Proposition~\ref{transverse}. Pick a general point $(a_4,a_5,a_7,a_8)\in \M_{1,4}^0$ corresponding to the curve $(C,P_1,P_2,P_3,P_4)$. We have a closed embedding 
    \[g:C\setminus\{P_1,P_2,P_3\} \to \M_{1,4}\]
    \[P\mapsto (C,P_1,P_2,P_3,P).\]
    This intersects $\M_{1,4}\setminus \M_{1,4}^0$ transversely because it is a fiber of $\pi: \M_{1,4}\to \M_{1,3}$ and $\M_{1,4}\setminus \M_{1,4}^0\cong \M_{1,3}^0$ is the image of a rational section of $\pi$ by Lemma~\ref{M140Complement}.

    This curve intersects $\M_{1,4}\setminus \M_{1,4}^0$ exactly at $([a_8:-a_5],[0:1]).$ We compute $g$ in the affine neighborhood $xz\neq 0$ away from $([a_8:-a_5],[0:1]).$ Set $u:=\frac{w}{x}$, $v:=\frac{y}{z}$, and consider a point $P=(u_0,v_0)\in C$. Following the proof of Proposition~\ref{M140presentation}, we embed $C$ in $\bP^1\times\bP^1$ using $\mathcal{O}(P_1+P)$ and $\mathcal{O}(P_2+P_3)$. The functions $1,\frac{v_0a_0-u_0a_5+(u_0v_0a_7-v_0^2a_0)u}{v_0u-u_0v}$ form a basis of $\mathcal{O}(P_1+P)$, and the functions $1,v $ form a basis for $\mathcal{O}(P_2+P_3)$. Hence, we are supposed to use the embedding
    \[C\to \bP^1\times\bP^1\]
    \[([x:w],[y:z])\mapsto ([v_0wz-u_0xy:(v_0a_0-u_0a_5)xz+(u_0v_0a_7-v_0^2a_0)wz],[y:z])\]
    up to the automorphism of $\bP^1\times\bP^1$ which sends $P$ to $P_4$ and $P_i$ to $P_i$ for $i\in[3]$. After applying the automorphism, our embedding sends $([x:w],[y:z])$ to
    { \[([(v_0a_0-u_0a_5)xz+u_0v_0a_7wz-a_0v_0u_0xy:a_0v_0^2wz-a_0v_0u_0xy],[y:v_0z]).\]}
    This has image given by the vanishing of
    \[x^2y^2-x^2yz-xwy^2+\frac{2a_7u_0v_0+2a_8u_0 +v_0a_4}{v_0^2}xwyz+\frac{a_5}{v_0^2}xwz^2+\]
    \[\frac{-a_8u_0+ v_0a_7-u_0v_0a_7}{v_0^2}w^2yz+\frac{v_0a_8-a_5u_0a_7v_0-a_5a_8u_0}{v_0^4}w^2z^2\]
    which one can check using a computer algebra system by noting that composing this expression with the embedding gives a multiple of the defining equation of $C$. 
    Thus, $g$ is given by sending $(u_0,v_0)$ to 
    {\footnotesize \[\left(\frac{2a_7u_0v_0+2a_8u_0 +v_0a_4}{v_0^2a_0},\frac{a_5}{v_0^2a_0},\frac{-a_8u_0+ v_0a_7-u_0v_0a_7}{v_0^2a_0},\frac{a_0v_0a_8-a_5u_0a_7v_0-a_5a_8u_0}{v_0^4a_0^2}\right).\]}
    Composing this map with $f$, one gets
    \[\frac{h(u_0,v_0)}{v_0^{20}},\]    
    for some polynomial $h(u_0,v_0)$, using a computer algebra system. Expanding $h(u_0,v_0)$ around $(u_0,v_0)=(-\frac{a_5}{a_8},0)$, one gets an expression whose lowest term is degree $4$ in $(u_0+\frac{a_5}{a_8}),v_0$. The defining equation of $C$ is equivalent to 
    \[u_0+\frac{a_5}{a_8}=\frac{1}{a_5a_8^2}(a_8^2a_7(u_0+\frac{a_5}{a_8})^2v_0 + a_8^3(u_0+\frac{a_5}{a_8})^2 +(a_4a_8^2 - 2a_5a_7a_8)(u_0+\frac{a_5}{a_8})v_0-\]
    \[(u_0+\frac{a_5}{a_8})a_8^2v_0^2 + (a_5a_8
    + a_8a_8)v_0^2 + (-a_8^2 - a_4a_5a_8 + a_5^2a_7)v_0).\]
    Substituting this expression for $u_0+\frac{a_5}{a_8}$ into $h$, one gets an expression whose lowest term is degree $8$ in $(u_0+\frac{a_5}{a_8}),v_0$. performing the same substitution again, one gets $v_0^8$ plus terms of larger degree. 
    Since $v_0$ is a uniformizer at $(-\frac{a_5}{a_8},0)$, we have 
    \begin{align*}
        \ord_{\M^0_{1,3}}(f)&=\ord_{(-\frac{a_5}{a_8},0)}(f\circ g))\\
        &=\ord_{(-\frac{a_5}{a_8},0)}\left(\frac{v_0^8+\text{higher order terms}}{v_0^{20}}\right)\\
        &=-12.
    \end{align*}
    Thus, $\partial_1(\mathfrak h)=-12[\M_{1,3}^0]$. 
    
    Next, we claim $\partial_1(\mathfrak q_i)=0$ for all $i\in [6]$. Consider the following part of the localization exact sequence 
    \[\overline\CH(\M_{1,4}^0,1)\xrightarrow{\partial_1} \CH(\M_{1,3}^0)\to \CH(\M_{1,4}).\]
    Using Theorem~\ref{M130} and Theorem~\ref{M140}, this becomes
    \[\frac{\Z\langle \mathfrak h, \mathfrak q_1,\mathfrak q_2,\mathfrak q_3,\mathfrak q_4,\mathfrak q_5,\mathfrak q_6\rangle}{\langle \lambda\mathfrak h, 2 \mathfrak q_1,2\mathfrak q_2,2\mathfrak q_3,2\mathfrak q_4,2\mathfrak q_5,2\mathfrak q_6\rangle}\xrightarrow{\partial_1}\frac{\Lambda\langle[\M_{1,3}^0] \rangle}{\langle 2\lambda[\M_{1,3}^0]\rangle}\to \CH(\M_{1,4}).\]
    Because the degree of $c_i=2$, it suffices to show that the pushforward $\iota_*: \CH(\M_{1,3}^0)\to \CH(\M_{1,4})$ is injective in degree $2$, i.e. that $\iota_*(\lambda [\M_{1,3}^0])\neq 0$. Because $\overline\CH^i(\M_{1,4}^0,1)=0$ for $i>2$, we have $\iota_*$ is injective in degrees larger than $2$. Thus, we have
    \[0\neq \iota_*(\lambda^2 [\M_{1,3}^0])=\lambda\iota_*(\lambda[\M_{1,3}^0]),\]    
    so $\iota_*(\lambda[\M_{1,3}^0])\neq 0$.

    With this computation of $\partial_1$, we can separate the localization exact sequence for $\M_{1,4}^0\subseteq \M_{1,4}$ into the exact sequences
    \[0\to \frac{\CH(\M_{1,3}^0)}{\langle 12[\M_{1,3}^0]\rangle}\xrightarrow{\iota_*} \CH(\M_{1,4})\to \CH(\M_{1,4}^0)\to 0\]
    and
    \[\overline\CH(\M_{1,3}^0,1)\xrightarrow{\iota_*} \overline\CH(\M_{1,4},1)\to \frac{\overline\CH(\M_{1,4}^0)}{\langle \mathfrak h\rangle}\to 0.\]
    Using Theorem~\ref{M130} and Theorem~\ref{M140}, we can write these exact sequences as
    \[0\to \frac{\Lambda\langle [\M_{1,3}^0]\rangle}{\langle 12[\M_{1,3}^0],2\lambda[\M_{1,3}^0]\rangle}\xrightarrow{\iota_*} \CH(\M_{1,4})\to \Z\to 0\]
    and 
    \[\frac{\Lambda \langle \mathfrak g, \mathfrak p_{12}^0,\mathfrak p_{13}^0\rangle}{\langle \lambda\mathfrak g,2\mathfrak p_{12}^0,2\mathfrak p_{13}^0\rangle}\to \overline\CH(\M_{1,4},1)\to \frac{\Z\langle\mathfrak q_1,\mathfrak q_2,\mathfrak q_3,\mathfrak q_4,\mathfrak q_5,\mathfrak q_6\rangle}{\langle 2 \mathfrak q_1,2\mathfrak q_2,2\mathfrak q_3,2\mathfrak q_4,2\mathfrak q_5,2\mathfrak q_6\rangle}\to 0.\]

    To get the presentation for $\CH(\M_{1,4})$ as a $\Lambda$-module, we use Lemma~\ref{exactPresentation}. We know that we can write
    \[\lambda=\iota_*(c[\M_{1,3}^0])\in \CH(\M_{1,4})\]
    for some $c\in \Lambda$. By considering degrees, we have $c\in \Z$, well defined modulo $12$. From this exact sequence describing $\CH(\M_{1,4},1)$, we have that $\overline\CH^1(\M_{1,4},1)=0$, so by Proposition~\ref{Order 12}(3), we know $\lambda\in \CH(\M_{1,4})$ has order $12$. Thus, we have $c$ is invertible modulo $12$. Using Lemma~\ref{exactPresentation}, we have
    \[\CH(\M_{1,4})=\frac{\Lambda\langle 1,[\M_{1,3}^0]\rangle}{\langle \lambda- c^{-1}[\M_{1,3}^0], 12[\M_{1,3}^0], 2\lambda [\M_{1,3}^0] \rangle}=\frac{\Lambda}{(12\lambda,2\lambda^2)}.\]

    Now we get the presentation for $\overline\CH(\M_{1,4},1)$. We have seen $\overline\CH^1(\M_{1,4},1)\allowbreak =0.$ Because $j_*$ is injective on $\CH^2(\M_{1,3}^0)$, the localization exact sequence in degree $2$ gives the claimed description for $\overline\CH^2(\M_{1,4},1)$. For degrees $\geq 3$, we consider the commutative diagram
    \begin{center}
        \begin{tikzcd}
            \M_{1,3}^0\arrow[rr,"j"]\arrow[rd] & & \M_{1,4}\arrow[ld,"\pi"]\\
             & \M_{1,3} & 
        \end{tikzcd}.
    \end{center}
    By Theorem~\ref{M13} and Theorem~\ref{M130}, we have
    \[(\pi\circ j)^*: \overline\CH(\M_{1,3},1)\to \overline\CH(\M_{1,3}^0,1)\]
    is an isomorphism in degrees $\geq 3$, and so we have a surjection
    \[\pi^*: \left(\frac{\Z}{2\Z}\right)^2\cong \overline\CH^i(\M_{1,3},1)\to \overline\CH^i(\M_{1,4},1)\]
    for $i\geq 3$. Moreover, the localization exact sequence gives a surjection
    \[j_*: \left(\frac{\Z}{2\Z}\right)^2\cong \overline\CH^{i-1}(\M_{1,3}^0,1)\to \overline\CH^i(\M_{1,4},1)\]
    for $i\geq 3$. Hence, we must have $\pi^*$ is an isomorphism in degrees $\geq 3$. Thus, by Theorem~\ref{M13}, we can write
    \[\overline\CH^{\geq 3}(\M_{1,4},1)=\frac{\Lambda\langle \lambda\mathfrak p_{12},\lambda\mathfrak p_{13},\lambda\mathfrak p_{23}  \rangle}{\langle 2\lambda \mathfrak p_{12},2\lambda \mathfrak p_{13}, 2\lambda\mathfrak p_{23},\lambda\mathfrak p_{12}+\lambda\mathfrak p_{13}+\lambda\mathfrak p_{23} \rangle},\]
    using that $\pi^*(\mathfrak p_{jk})=\mathfrak p_{jk}$.
    The classes $\lambda\mathfrak p_{j4}$ also live in this group, and by considering the $S_4$ symmetry, we can conclude
    \[\lambda \mathfrak p_{j4}=\lambda \mathfrak p_{k\ell}\]
    where $\{j,k,\ell\}=[3]$. This gives the claimed presentation for $\overline\CH^{\geq 3}(\M_{1,4},1)$.
    
    Finally, we show that $\M_{1,4}$ has the MKP. By Theorem~\ref{M130}, $\M_{1,3}^0$ has the MKP; by Theorem~\ref{M140}, $\M_{1,4}^0$ has the MKP. Thus, by Proposition~\ref{MKPprops}(1), $\M_{1,4}$ has the MKP.
\end{proof}


\section{Stratification of $\bar\M_{g,n}$}
\begin{definition}
    Let $\Gamma$ be a stable graph. We define $\M^\Gamma_{g,n}$ to be the locus of curves in $\bar\M_{g,n}$ with stable graph $\Gamma$. Its closure is denoted $\bar\M^\Gamma_{g,n}$. When there is no confusion as to which $\bar\M_{g,n}$ we are working in, we omit the subscripts, writing just $\M^\Gamma$ and $\bar\M^\Gamma$. 
    
    We also define
    \[\bar\M_\Gamma:=\prod_{v\in \Gamma} \bar\M_{g_v,n_v}\]
    and 
    \[\M_\Gamma:=\prod_{v\in \Gamma} \M_{g_v,n_v}.\]
\end{definition}
There is a surjective representable map 
    \[\xi_\Gamma: \bar\M_\Gamma\to \bar\M^\Gamma\]
that glues together the various curves. Moreover, $\xi_\Gamma(\M_\Gamma)=\M^\Gamma$. There is an action of $\Aut(\Gamma)$ on $\bar\M_\Gamma$, which induces 
\[[\bar\M_\Gamma/\Aut(\Gamma)]\to \bar\M^\Gamma\]
which we also denote by $\xi_\Gamma$. The following result is \cite[Proposition 2.5]{BS1}.
\begin{prop}\label{OpenStrata}
    The map $\xi_\Gamma$ restricts to an isomorphism
    \[\M_\Gamma/\Aut(\Gamma)\cong \M^\Gamma.\]
\end{prop}

Next, we will need to know how to intersect boundary strata and multiply boundary strata classes. The following results from \cite[Appendix A]{GP} describes how to do these computations.
\begin{thm}\label{tautologicalcalculus}
    Let $\Gamma,\Gamma'$ be genus $g$, $n$-pointed stable graphs. Then we have
        \[\bar\M_\Gamma \times_{\overline\M_{g,n}} \bar\M_{\Gamma'}=\coprod_{\substack{\Delta\text{ generic}\\ (\Gamma,\Gamma')-\text{graph}}} \bar\M_{\Delta}\]
        and
         \[\xi_{\Gamma'}^*([\bar\M^{\Gamma}])=\sum_{\substack{\Delta\text{ generic}\\ (\Gamma,\Gamma')-\text{graph}}} \xi_{\Delta,\Gamma'*}\left(\prod_{e=h+h'} -\pi_v^*(\psi_h)-\pi^*_{v'}(\psi_{h'})\right),\] 
    where $\xi_{\Delta,\Gamma'}$ is the morphsim $\bar\M_{\Delta}\to \bar\M_{\Gamma'}$.
\end{thm}
The fibered product calculation lets us compute the intersection of $\bar\M^\Gamma$ and $\bar\M^{\Gamma'}$ by considering the images inside $\bar\M_{g,n}$. The proof for the pullback formula given in \cite{GP} was for rational coefficients, but the proof goes through with integral coefficients because this is just a computation with normal bundles and the excess intersection formula. In the contexts we will use this theorem in, $\Gamma$ and $\Gamma'$ will always be codimension $1$. This means there will always be a unique $\Delta$. 

To perform computations with this theorem, we need to have formulas for some $\psi$ classes on $\bar\M_{0,n}$ and $\bar\M_{1,n}$. The following proposition allows for the recursive computation of these classes.
\begin{prop}\label{psi}
    We have $\psi_1=0$ on $\bar\M_{0,3}$ and $\psi_1=\lambda$ on $\bar\M_{1,1}$. Additionally, if $2g-2+n>0$, then for $\pi: \bar\M_{g,n+1}\to \bar\M_{g,n}$, we have
    \[\psi_i=\pi^*(\psi_i)+\gamma\]
    where $\Gamma$ is the stable graph
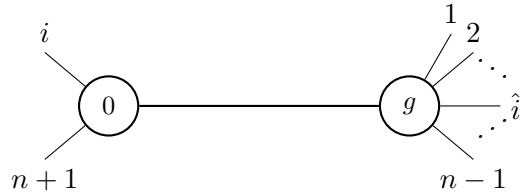
\[\begin{tikzpicture}[
  vertex/.style={
    circle, draw, thick, minimum size=2em, inner sep=0pt, font=\small
  }
]

  \node[vertex] (v0) at (0,0) {$0$};
  \node[vertex] (vg) at (4,0) {$g$};

  \draw[thick] (v0) -- (vg);

  \draw (v0.140) -- ++(140:0.7) node[above] {$i$};
  \draw (v0.-140) -- ++(-140:0.7) node[below] {$n+1$};

  \draw (vg.60)  -- ++(60:0.7)  node[above] {$1$};
  \draw (vg.40)  -- ++(40:0.7)  node[above] {$2$};
  \node[above] at (5.1,.2) {$\ddots$};

  \draw (vg.0)  -- ++(0:0.8)  node[right] {$\hat i$};
   \node[above] at (5.1,-.6) {$\udots$};
  
  \draw (vg.-40) -- ++(-40:0.7)  node[below] {$n-1$};

\end{tikzpicture}.\]
\end{prop} 
\begin{proof}
     Note we have $\psi_i\in \CH^1(\bar\M_{0,3})=0$. 
     Next, if $\omega_\pi$ is the relative dualizing sheaf on $\mathcal C_{1,1} (\cong \bar\M_{1,1})$, and $\sigma:\bar\M_{1,1}\to \mathcal C_{1,1}$ is the section, then we have  
    \[\omega\cong \pi^*\sigma^*\omega\implies \pi_*(\omega)\cong \pi_*(\pi^*\sigma^*\omega)\cong \pi_*(\OO_{\mathcal C})\otimes \sigma^*\omega\cong \sigma^*\omega.\]
    Thus,
    \[\lambda=c_1(\pi_*(\omega))=c_1(\sigma^*\omega)=\psi_1.\]
    The final claim is \cite[Equation 8]{Getzler2}.
\end{proof}

Now, Proposition~\ref{OpenStrata} describes the morphism $\xi_\Gamma$ on the open locus $\M_\Gamma$. Our next result explains what happens on other the other strata in $\bar\M_\Gamma$. Suppose $\M^{\Gamma'}\subseteq \bar\M^\Gamma$. Choose a $\Gamma$-structure on $\Gamma'$. Then we have that there is some graph $\Gamma'_v$ for $v\in \Gamma$ so that inserting the $\Gamma'_v$ into $\Gamma$ at $v$ gives the graph $\Gamma'$. We can write 
\[\M_{\Gamma'}=\prod_{w\in \Gamma'} \M_{g(w),n(w)}=\prod_{v\in \Gamma}\prod_{w\in \Gamma'} \M_{g(w),n(w)}=\prod_{v\in \Gamma} (\M_{g(v),n(v)})_{\Gamma'_v}.\]
We get a morphism 
\[\M_{\Gamma'}\to \bar\M_{\Gamma}=\prod_v \bar\M_{g(v),n(v)}\]
defined by $(\xi_{\Gamma'_v})_v$. Let $\M_{\Gamma}^{\Gamma'}$ denote the image, and let $\xi_{\Gamma}^{\Gamma'}$ be the restriction of $\xi_\Gamma$ to $\M_\Gamma^{\Gamma'}$. Define 
\[\Aut_\Gamma(\Gamma'):=\{\varphi\in \Aut(\Gamma')|\varphi(\Gamma'_v)=\Gamma'_v\}.\]
\begin{lemma}\label{automorphism}
    Then the morphism $\xi_\Gamma^{\Gamma'}:\M_{\Gamma}^{\Gamma'}\to \M^{\Gamma'}$ has degree equal to $[\Aut(\Gamma'):\Aut_\Gamma(\Gamma')]$, and is an isomorphism if this number is $1$.
\end{lemma}
\begin{proof}
    By definition, we have
    \[\M_\Gamma^{\Gamma'}=(\xi_{\Gamma'_v})_v(\M_{\Gamma'})=(\xi_{\Gamma'_v})_v\left(\prod_{v\in \Gamma} (\M_{g(v),n(v)})_{\Gamma'_v}\right)=\prod_{v\in \Gamma} \M_{g(v),n(v)}^{\Gamma'_v}.\]
    By Proposition~\ref{OpenStrata}, we have 
    \[[(\M_{g(v),n(v)})_{\Gamma'_v}/\Aut(\Gamma'_v)]\stackrel{\xi_{\Gamma'_v}}\cong \M_{g(v),n(v)}^{\Gamma'_v},\]
    and so $(\xi_{\Gamma'_v})_v$ induces an isomorphism
    \[\M^{\Gamma'}_\Gamma\cong \prod_{v\in \Gamma} [(\M_{g(v),n(v)})_{\Gamma'_v}/\Aut(\Gamma'_v)]=[\M_{\Gamma'}/\Aut_\Gamma(\Gamma')].\]

    Consider the following diagram
    \begin{center}
        \begin{tikzcd}
            \M_{\Gamma'}\arrow[d,"(\xi_{\Gamma'_v})_v",swap]\arrow[dr,"\xi_{\Gamma'}"] & \\
             \M_\Gamma^{\Gamma'} \arrow[r,"\xi_\Gamma^{\Gamma'}"] & \M^{\Gamma'}.
        \end{tikzcd}
    \end{center}
    We know $\xi_{\Gamma'}$ has degree $\# \Aut(\Gamma')$ by Proposition~\ref{OpenStrata}, and $(\xi_{\Gamma'_v})_v$ has degree $\# \Aut_\Gamma(\Gamma')$, so the degree of $\xi_\Gamma: \M_\Gamma^{\Gamma'}\to \M^{\Gamma'}$ must be $[\Aut(\Gamma'): \Aut_\Gamma(\Gamma')]$. If this number is $1$, we have $\Aut_\Gamma(\Gamma')=\Aut(\Gamma')$, and hence \[\M_\Gamma^{\Gamma'}=[\M_{\Gamma'}/\Aut_\Gamma(\Gamma')]=[\M_{\Gamma'}/\Aut(\Gamma)]=\M^{\Gamma'}\]
    via the identifications given by the morphisms in the above diagram.
\end{proof}

\begin{prop}\label{SeparatingProduct}
    Suppose $\Gamma$ is a genus $1$ stable graph with a (necessarily unique) genus $1$ vertex, $v_0$. Then $\xi_\Gamma: \bar\M_\Gamma \to \bar\M^\Gamma$ is a universal separable homeomorphism. 
\end{prop}
\begin{proof}
    Take $\Gamma'$ with $\M^{\Gamma'}\subseteq \bar\M^\Gamma$ and choose a $\Gamma$ structure on $\Gamma'$. For $v\in \Gamma$, let $\Gamma'_v$ be the graphs glued into $\Gamma$ at $v$ to make $\Gamma'$. By Lemma~\ref{automorphism}, if we show that $\Aut(\Gamma')=\Aut_\Gamma(\Gamma')$ for all $\Gamma'$ with, then we have $\xi_\Gamma^{\Gamma'}$ is an isomorphism. And then, for any map $x:\Spec(K)\to \bar\M^\Gamma$ landing in $\M^{\Gamma'}$, we will be able to find a $K$-point of $\M_\Gamma^{\Gamma'}\subseteq \bar\M_\Gamma$ lifting $x$.
    
    Take an automorphism $\varphi\in \Aut(\Gamma')$. Define the subgraph $\Gamma''\subseteq \Gamma_{v_0}'$ to be the ``genus $1$ part'' of $\Gamma_{v_0}'$, meaning either the genus $1$ vertex in $\Gamma_{v_0}'$ if it exists and otherwise just the vertices on the unique cycle making $\Gamma_{v_0}'$ have genus $1$. Certainly $\varphi(\Gamma'')\subseteq \Gamma''$. Next consider, the graph $\Gamma'/\Gamma''$ obtained by collapsing $\Gamma''$ to a point. Then $\varphi$ induces an automorphism of $\Gamma'/\Gamma''$, and this graph is a tree. The leaves of this tree must be fixed, since they are either the vertex $\{\Gamma''\}$ or otherwise have  a marked point attached, and any automorphism of a tree fixing the leaves is the identity. Thus, $\varphi$ fixes every vertex outside of $\Gamma''$ pointwise, hence $\varphi(\Gamma'_v)=\Gamma'_v$ for all $v$. 
\end{proof}

\begin{rmk}
    It is not always true that 
    \[\xi_\Gamma: [\bar\M_\Gamma/\Aut(\Gamma)]\to \bar\M^\Gamma\] 
    is a universal separable homeomorphism. The graph $\Theta$ defined in \ref{secbarM12} is one example. 
\end{rmk}

The following proposition explains how to pull back boundary strata.
\begin{prop}\label{ReducedFibers}
    For $\pi: \bar\M_{g,n+1}\to \bar\M_{g,n}$ the morphism forgetting the $n+1$-st point, and an $n$-pointed stable graph of genus $g$ $\Gamma$, we have
    \[\#\Aut(\Gamma)\cdot\pi^*([\bar\M^{\Gamma}])=\sum_{\Gamma'} [\bar\M^{\Gamma'}]\]
    where the sum is over $n+1$-pointed stable graphs of genus $g$ that stabilize to $\Gamma$ after removing the $n+1$ marking.
\end{prop}
\begin{proof}
    It is clear set-theoretically that 
     \[\pi^{-1}(\bar\M^{\Gamma})=\bigcup_{\Gamma'} \bar\M^{\Gamma'}\]
     where $\Gamma'$ are as in the statement. By the definition of flat pullback, it suffices to argue that $\pi^{-1}(\bar\M^{\Gamma})$ is reduced. This follows from \cite[\href{https://stacks.math.columbia.edu/tag/0C0E}{Tag 0C0E}]{stacks-project} because all of the fibers over closed points are reduced. 
\end{proof}

Now, we describe our strategy for computing $\CH(\partial\bar\M_{g,n})$.

\begin{definition}
    For a fixed $g,n$, let $\bar\M^{\sep}$ denote the locus of curves in the boundary with at least one separating node, and let $\bar\M^{\sep,\geq p}$ denote the union of $\bar\M^{\sep}$ with the locus of curves with at least $p$ nodes. Finally, let $\M^{\non=p}=\bar\M^{\sep,\geq p}\setminus \bar\M^{\sep,\geq p+1}$. 
\end{definition}
The following are immediate from the definition:
\begin{enumerate}
    \item $\bar\M^\sep$ is the union of $\bar\M^\Gamma$ over stable graphs $\Gamma$ with one node which is separating,
    \item $\bar\M^{\sep,\geq p}=\bar\M^\sep$ for $p> \dim(\bar\M_{g,n})$, 
    \item $\bar\M^{\sep,\geq 1}=\partial\bar\M_{g,n},$ and
    \item $\M^{\non=p}$ is the disjoint union of $\M^\Gamma$ over graphs $\Gamma$ with exactly $p$ nodes, all of which are non-separating.
\end{enumerate}
We compute $\CH(\partial\bar\M_{g,n})$ using the stratification 
\[\bar\M^{\sep}\subseteq\dots\subseteq \bar\M^{\sep,\geq p+1}\subseteq \bar\M^{\sep,\geq p}\subseteq \dots\subseteq \bar\M^{\sep,\geq 2}\subseteq \bar\M^{\sep,\geq 1}=\partial\bar\M_{g,n}\]
and the localization exact sequence. To do this, one must know the Chow ring of the bottom piece of the filtration, $\bar\M^\sep$, the Chow ring of the open parts, $\M^{\non=p}$, and then one uses the localization exact sequence to fit the Chow rings of these spaces together. We compute the Chow ring of $\M^{\non=p}$, using the description of $\M^\Gamma$ given in Proposition~\ref{OpenStrata}. To compute the Chow group of $\bar\M^{\sep}$, note we have an exact sequence
\[\bigoplus_{\Gamma,\Gamma'\in S}\CH(\bar\M^\Gamma\cap \bar\M^{\Gamma'})\to \bigoplus_{\Gamma\in S} \CH(\bar\M^\Gamma)\to \CH(\bar\M^\sep)\to 0\]
where $S$ is the set of stable graphs with one separating node. In this exact sequence, we can compute $\CH(\bar\M^\Gamma)$ and $\CH(\bar\M^\Gamma\cap \bar\M^{\Gamma'})$ for $\Gamma,\Gamma'\in S$ using Proposition~\ref{SeparatingProduct} and the MKP for $\bar\M_{0,n}$.

We compute $\CH(\bar\M^{\sep,\geq p})$ in terms of $\CH(\bar\M^{\sep,\geq p+1})$ and $\CH(\M^{\non= p})$ using the following proposition. Let $S_p$ be the set of $\Gamma$ with exactly $p$ nodes, all of which are non-separating, and let $\partial_1^\Gamma$ be the map $\overline\CH(\M^\Gamma,1)\to \CH([\bar\M_\Gamma/\Aut(\Gamma)])$. 
\begin{prop}\label{filtration}
    We have an exact sequence
    \[0\to \frac{\CH(\bar\M^{\sep,\geq p+1})}{\underset{\Gamma\in S_p}{\bigoplus} \xi_{\Gamma *}(\im(\partial_1^\Gamma))}\to \CH(\bar\M^{\sep,\geq p})\to \underset{\Gamma\in S_p}{\bigoplus} \CH(\M^\Gamma)\to 0.\]
    Moreover, splittings of $\CH([\bar\M_\Gamma/\Aut(\Gamma)])\twoheadrightarrow \CH(\M^\Gamma)$, for all $\Gamma\in S_p$ induce a splitting of this exact sequence, and hence a direct sum decomposition
    \[\CH(\bar\M^{\sep,\geq p})=\frac{\CH(\bar\M^{\sep,\geq p+1})}{\underset{\Gamma\in S_p}{\bigoplus} \xi_{\Gamma *}(\im(\partial_1^\Gamma))}\oplus \underset{\Gamma\in S_p}{\bigoplus} \CH(\M^\Gamma).\]
\end{prop}
\begin{proof}
    Using Proposition~\ref{OpenStrata}, we have the following commutative diagram
    \begin{center}
    \begin{tikzcd}[sep=1em, font=\scriptsize]
        \overline\CH(\M^\Gamma,1)\arrow[d]\arrow[r,"\partial_1^\Gamma"] & \CH([\partial\bar\M_\Gamma/\Aut(\Gamma)])\arrow[r]\arrow[d,"\xi_{\Gamma*}"] & \CH([\bar\M_\Gamma/\Aut(\Gamma)])\arrow[r]\arrow[d,"\xi_{\Gamma *}"] &  \CH(\M^\Gamma)\arrow[r]\arrow[d] & 0 \\
        \overline\CH(\M^{\non=p},1)\arrow[r,"\partial_1"] & \CH(\bar\M^{\sep,\geq p+1})\arrow[r] & \CH(\bar\M^{\sep,\geq p})\arrow[r] & \CH(\M^{\non=p}) \arrow[r] & 0.
    \end{tikzcd}
\end{center}
Because 
\[\M^{\non=p}=\coprod_{\Gamma\in S_p} \M^\Gamma,\]
we have 
\[\overline \CH(\M^{\non=p},1)= \bigoplus_{\Gamma\in S_p} \overline\CH(\M^\Gamma,1) \text{ and } \CH(\M^{\non=p})= \bigoplus_{\Gamma\in S_p} \CH(\M^\Gamma).\]
Thus, the bottom row of the diagram induces 
\[0\to \frac{\CH(\bar\M^{\sep,\geq p+1})}{\im(\partial_1)}\to \CH(\bar\M^{\sep,\geq p})\to \underset{\Gamma\in S_p}{\bigoplus} \CH(\M^\Gamma)\to 0,\]
and we have $\im(\partial_1)=\underset{\Gamma\in S_p}{\bigoplus} \xi_{\Gamma *}(\im(\partial_1^\Gamma))$. Finally, the splittings of \[\CH([\bar\M_\Gamma/\Aut(\Gamma)])\twoheadrightarrow \CH(\M^\Gamma)\] 
induce a splitting of \[\CH(\bar\M^{\sep,\geq p})\to \CH(\M^{\non=p})\]
by the diagram.
\end{proof}


\section{$\bar\M_{0,n}$ and $[\bar\M_{0,n}/\mu_2]$}
In this section, we study $\M_{0,n}, \bar\M_{0,n}, [\M_{0,n}/\mu_2],$ and $[\bar\M_{0,n}/\mu_2]$, where $\mu_2$ acts by switching the last two markings.

Here are our $\M_{0,n}$ conventions. In order to make the action by $\mu_2$ on $\bA^{n-3}$ linear, our conventions are nonstandard. Suppose $n\geq 3$. Firstly, we label the markings $1,2,\dots,n-2,a,b$. Now, given a family $\pi: \mathcal{C}\to S$ of smooth genus $0$ curves, with $n$-sections $\sigma_i: S\to \mathcal{C}$, we can find a unique $S$-isomorphism $\mathcal{C}\cong \bP^1\times S$ such that 
\begin{itemize}
    \item $\sigma_1$ is the constant $\infty$ section,
    \item $\sigma_a$ is the constant $1$ section, and
    \item $\sigma_b$ is the constant $-1$ section.
\end{itemize}
Thus, we have an isomorphism 
\begin{equation}\label{*}
    \M_{0,n}\cong \bA^{n-3}\setminus V\left(\prod_{p<q}(x_p-x_q)\prod_{p} (x_p^2-1)\right),
\end{equation}
where $\bA^{n-3}=\Spec k[x_2,\dots,x_{n-2}]$, and $x_p$ records the $p$-th marking. The induced $\mu_2$-action on the right is $(x_2,\dots,x_{n-2})\mapsto (-x_2,\dots,-x_{n-2})$, which extends to $\bA^{n-3}$. With this action, we have
\[\CH_{\mu_2}(\bA^{n-3})=\CH(B\mu_2)=\frac{\Z[u]}{(2u)},\]
where $u$ is the class of a hyperplane going through the origin, by Lemma~\ref{Bmu} and homotopy invariance.

\subsection{MKP in Genus $0$}
\begin{prop}\label{MKPM0n}
    $\M_{0,n},[\M_{0,n}/\mu_2],[\M_{0,n}\times \M_{0,m}/\mu_2],$ and  $\bar\M_{0,n}$ have the MKP.
\end{prop}
\begin{proof}
First, we have that if $Z\subseteq \bA^N$ is a union of hyperplanes then $Z$ has the MKP. To see this, one repeatedly uses Proposition~\ref{MKPprops}(1) and the fact that intersections of hyperplanes always give an affine space, which has the MKP by Proposition~\ref{MKPprops}(3). Because $\M_{0,n}$ is the complement of a union of hyperplanes inside an affine space, both of which have the MKP, we have that $\M_{0,n}$ has the MKP. 

Similarly, suppose we have $[Z/\mu_2]\subseteq [\bA^N/\mu_2]$ with $Z$ is a union of hyperplanes, where $\mu_2$ acts on $\bA^N$ by negation. For any collection of components of $Z$, the intersection is isomorphic to an affine space if one of the components does not contain the origin, and affine space has the MKP by Proposition~\ref{MKPprops}(3). Otherwise, the intersection is isomorphic to $[\bA^m/\mu_2]$, which has the MKP by Proposition~\ref{MKPprops}(3) and Lemma~\ref{Bmu}. Thus, both $[\M_{0,n}/\mu_2]$ and $[\M_{0,n}\times \M_{0,m}/\mu_2]$ have the MKP because they are complements of such $[Z/\mu_2]$ inside some $[\bA^N/\mu_2]$. 

    Now, we know $\bar\M_{0,n}$ has a stratification into 
    \[\M^\Gamma= [\M_\Gamma/\Aut(\Gamma)]=[\prod_{v\in \Gamma} \M_{g(v),n(v)}/\Aut(\Gamma)]\] 
    for genus $0$ stable $n$-pointed graphs $\Gamma$. Because the graph has genus $0$, all vertices must have genus $0$, and the graph must be a tree. Then, any leaf on the graph must have marked points attached, to ensure that there are at least 3 special points on the component. An automorphism of $\Gamma$ must then fix all leaves, and is therefore the identity. Thus, we have
    \[\M^\Gamma=\prod_{v\in \Gamma} \M_{0,n(v)}.\]
    By the above and Proposition~\ref{MKPprops}(2), these $\M^\Gamma$ have the MKP. Thus we have $\bar\M_{0,n}$ has the MKP using Proposition~\ref{MKPprops}(1) and the fact that these $\M^\Gamma$ stratify $\bar\M_{0,n}$.
\end{proof}

\begin{cor}
    If $\M_{1,m}$ has the MKP for $m\leq n$, then  $\bar\M_{1,n}$ has the MKP.
\end{cor}
\begin{proof}
We have a stratification of $\bar\M_{1,n}$ by $\M^\Gamma$ for stable graphs $\Gamma$. Considering the possibilities for $\Gamma$ and using Proposition~\ref{OpenStrata}, we have that $\M^\Gamma$ is isomorphic to one of the following spaces
\begin{itemize}
\item $\M_{1,m}\times \prod_{i=1}^N \M_{0,n_i} \text{ for some }m\leq n, n_i\in \Z$
\item $\prod_{i=1}^N \M_{0,n_i} \text{ for some }n_i\in \Z$
\item $ [\M_{0,n_1}\times \M_{0,n_2}/\mu_2] \text{ for some }n_i\in \Z$.
\end{itemize}
Assuming the hypothesis, all of these spaces have the MKP, using the previous proposition and Proposition~\ref{MKPprops}(2). By repeated use of Proposition~\ref{MKPprops}(1), $\bar\M_{1,n}$ has the MKP.
\end{proof}

By Theorem~\ref{M11}, Theorem~\ref{M12}, Theorem~\ref{M13}, and Theorem~\ref{M14}, we have $\M_{1,n}$ has the MKP for $n\leq 4$, giving the following corollary.
\begin{cor}\label{barM1nMKP}
    $\bar\M_{1,n}$ has the MKP for $n\leq 4$.
\end{cor}

\subsection{Chow in Genus $0$}
We now work toward computing $\CH([\M_{0,n}/\mu_2])$ and $\overline\CH([\M_{0,n}/\mu_2],1)$. First a lemma on $[\bA^m/\mu_2]$.
\begin{lemma}\label{ContainOrigin}
     Let $U\subseteq \bA^m$ be a $\mu_2-$invariant open subset with $0\in U$. Then $\CH_{\mu_2}(\bA^n)\to \CH_{\mu_2}(U)$ is an isomorphism. 
\end{lemma}
\begin{proof}
     Note that the pullback 
\[\CH_{\mu_2}(\bA^n)\to \CH_{\mu_2}(\{0\})\cong \CH(B\mu_2)\]
is an isomorphism because it is left inverse to the pullback \[\CH(B\mu_2)\to \CH_{\mu_2}(\bA^n),\] which is an isomorphism by homotopy invariance. Thus, for $U$ an open subset of $\bA^n$, if $0\in U$, $\CH_{\mu_2}(\bA^n)\to \CH_{\mu_2}(U)$ is surjective by the localization exact sequence and injective because the isomorphism $\CH_{\mu_2}(\bA^n)\to \CH_{\mu_2}(\{0\})$ factors through it. 
\end{proof}

\begin{thm}\label{M0n/2}
    \[\CH([\M_{0,n}/\mu_2])\cong \begin{cases}
        \CH(B\mu_2) & n\leq 4\\
        \Z & n\geq 5
    \end{cases}\]
    and
    \[\overline\CH([\M_{0,n}/\mu_2],1)=\overline\CH^1([\M_{0,n}/\mu_2],1)=(\OO_{\M_{0,n}}(\M_{0,n})^\times)^{\mu_2}.\]  
\end{thm}

\begin{proof}
     The localization exact sequence for $[\M_{0,n}/\mu_2]\subseteq [\bA^{n-3}/\mu_2]$ reads
    \[0\to \overline\CH_{\mu_2}(\M_{0,n},1)\to \CH_{\mu_2}(Z)\to \CH_{\mu_2}^*(\bA^{n-3})\to\CH_{\mu_2}(\M_{0,n})\to 0,\]
using Lemma~\ref{Bmu} and homotopy invariance to say $\overline\CH_{\mu_2}^*(\bA^{n-3},1)=0$. 
    Suppose $n\leq 4$. Then $0\in \M_{0,n}$, so Lemma~\ref{ContainOrigin} implies $\CH([\bA^{n-3}/\mu_2])\to \CH([\M_{0,n}/\mu_2])$ is an isomorphism by exactness. The localization sequence then gives $\overline\CH_{\mu_2}(\M_{0,n},1)=\CH_{\mu_2}(Z).$ If $n=3$, $Z=\emptyset$, so $\overline\CH(\M_{0,3},1)=0$. If $n=4$, $Z=V(x_1^2-1)$, so $[Z/\mu_2]=\Spec(k)$, giving $\CH^1([\M_{0,4}/\mu_2],1)=\CH^0(\Spec(k))=\Z$. 

    Now suppose $n\geq 5$. Then $V(x_p-x_q)\cong \bA^{n-4}\subseteq Z$ and $\CH_{\mu_2}(\bA^m)=\Z[s]/(2s)$. Now $[V(x_p-x_q)]=s$ because $V(x_p-x_q)$ is a hyperplane containing the origin, so pushforward $\CH_{\mu_2}(V(x_p-x_q))\to \CH_{\mu_2}(\bA^{n-3})$ is multiplication by $s$, using Lemma~\ref{PullbackAn}. This is an isomorphism in positive degrees, and so 
    \[\CH_{\mu_2}(\M_{0,n})=\mathrm{coker}(\CH_{\mu_2}(Z)\to \CH_{\mu_2}(\bA^{n-3}))=\Z.\]

    We next compute $\CH_{\mu_2}(Z)$. We have 
    \[\bigoplus_p \CH_{\mu_2}(V(x_p^2-1)) \oplus \bigoplus_{p<q} \CH_{\mu_2}(V(x_p-x_q))\twoheadrightarrow \CH_{\mu_2}(Z)\]
    with the kernel generated by pushforwards of intersections of components. In particular, this is an isomorphism in degree $0$. Note $[V(x_p^2-1)/\mu_2]\cong \bA^{n-4}$  because the action is free, so
    \[\CH_{\mu_2}(V(x_p^2-1))=\CH(\bA^{n-4})=\Z.\]
    Moreover, for $p\neq q$, we must have
    \[\CH_{\mu_2}(V(x_p^2-1)\cap V(x_q^2-1))\to \CH_{\mu_2}(V(x_p^2-1))\oplus \CH_{\mu_2}(V(x_q^2-1))\]
    equals $0$, as it maps into positive degrees. For the same reason the map
    \[\CH_{\mu_2}(V(x_p^2-1)\cap V(x_j-x_\ell))\to \CH_{\mu_2}(V(x_p^2-1))\oplus \CH_{\mu_2}(V(x_j-x_\ell))\]
    is $0$ on the first coordinate. It is also $0$ on the second coordinate, because $[V(x_p^2-1)]\in \CH_{\mu_2}(\bA^m)$ is $0$. Finally, the map
    \[\CH_{\mu_2}(V(x_p-x_q)\cap V(x_j-x_\ell))\to \CH_{\mu_2}(V(x_p-x_q))\oplus \CH_{\mu_2}(V(x_j-x_\ell))\]
    is an isomorphism onto either coordinate in positive degrees. Thus, 
    \[\CH^i_{\mu_2}(Z)=\Z/2\Z,\]
    for $i\geq 1$.
    
    As noted above, $\CH_{\mu_2}(V(x_p-x_q))\to \CH_{\mu_2}(\bA^{n-3})$ is an isomorphism in degrees $>2$, and we have just seen $\CH^{*}_{\mu_2}(V(x_p-x_q))\to \CH_{\mu_2}(Z)$ is an isomorphism in positive degrees, so $\CH_{\mu_2}(Z)\xrightarrow{\iota_*} \CH_{\mu_2}(\bA^{n-3})$ is an isomorphism in degrees $>1$. Thus, 
    \[\overline\CH^i_{\mu_2}(\M_{0,n},1)=\ker(\iota_*)=0\]
    for $i\geq 2$. Moreover, 
    \[\overline\CH^1_{\mu_2}(\M_{0,n},1)=(\OO_{\M_{0,n}}(\M_{0,n})^\times)^{\mu_2}\]
    by Lemma~\ref{units}$(3')$.
\end{proof}
The following splitting will let us use Proposition~\ref{filtration}.
\begin{lemma}\label{barM04/2}
The surjection
\[\CH([\bar\M_{0,4}/\mu_2])\twoheadrightarrow \CH([\M_{0,4}/\mu_2])\]
has a splitting given by $\CH([\M_{0,4}/\mu_2])=\CH(B\mu_2)\xrightarrow{\nu^*} \CH([\bar\M_{0,4}/\mu_2])$, where $\nu$ is the canonical morphism $[\bar\M_{0,4}/\mu_2]\to B\mu_2$.  
\end{lemma}
\begin{proof}
We have that $\CH([\M_{0,4}/\mu_2])\to \CH(B\mu_2)\to \CH([\M_{0,4}/\mu_2])$ is the identity. Because we can factor the latter map as 
\[\CH(B\mu_2)\to \CH([\bar\M_{0,4}/\mu_2])\to \CH([\M_{0,4}/\mu_2]),\]
we have that the claimed map is a splitting.
\end{proof}

Next, we calculate the image of the homomorphisms
\[\partial_1:\overline\CH(\M_{0,n},1)\to \CH(\partial\bar\M_{0,n})\]
and 
\[\partial_1^{\mu_2}: \overline\CH([\M_{0,n}/\mu_2],1)\to \CH([\partial \bar\M_{0,n}/\mu_2]).\]
For $A,B\subseteq \{1,\dots,n-2,a,b\}$ disjoint subsets with $\#A,\#B\geq 2$, we have the divisors $D(A|B)=D(B|A)\subseteq \bar\M_{0,n}$ which correspond to curves that have a node separating the markings from sets $A$ and $B$. Such a divisor is a sum of irreducible divisors $D(A'|B')$ where $B'$ is the complement of $A'$. Let $\widehat D(A|B)$ be the image of $D(A|B)$ in $[\bar\M_{0,n}/\mu_2]$.  Additionally, let $\widehat T\in [\M_{0,4}/\mu_2]$ be the unique point with nontrivial stabilizer. By abuse of notation, we also say $D(A|B),\widehat D(A|B),\widehat T$ are the corresponding classes in the respective Chow rings.

\begin{thm}\label{WDVV mod 2}
\mbox{}
   \begin{itemize}
       \item The image of 
    \[\overline\CH(\M_{0,n},1)\xrightarrow{\partial_1} \CH(\partial \bar\M_{0,n})\]
    is freely generated by 
    \[D(1b|ja)-D(1j|ab)\]
    \[D(1a|jb)-D(1j|ab)\]
    for $j\in \{2,\dots,n-2\}$ and
    \[D(jk|ab)-D(jb|ka)\]
    for $j,k\in \{2,\dots,n-2\}$ with $j\neq k$.
       \item The image of 
    \[\overline\CH([\M_{0,n}/\mu_2],1)\xrightarrow{\partial_1^{\mu_2}} \CH([\partial \bar\M_{0,n}/\mu_2])\]
    is freely generated by 
    \[\widehat D(pa|1b)-2\widehat D(1p|ab)\]
    for $p\in \{2,\dots,n-2\}$ and
    \[\alpha_{23}+\alpha_{pq}\]
    for $p, q\in \{2,\dots,n-2\}$ with $p\neq q$, where 
    \[\alpha_{pq}:=\widehat D(pqb|1a)+\widehat D(1ab|pq)-\widehat D(1pq|ab)-\widehat D(qab|1p)-\widehat D(pab|1q).\]
   \end{itemize}
\end{thm}
\begin{proof}
    Note $\OO_{\M_{0,n}}(\M_{0,n})/k^\times=\overline\CH(\M_{0,n},1)$ is freely generated by $\{x_p\pm 1\},\{x_p-x_q\},$ and so $(\OO_{\M_{0,n}}(\M_{0,n})^\times)^{\mu_2}/k^\times=\overline\CH([\M_{0,n}/\mu_2],1)$ is freely generated by 
    \[\{x_p^2-1\}_p,\{(x_2-x_3)(x_{p}-x_{q})\}_{p\neq q}.\]
    Therefore, $\partial_1^{\mu_2}$ applied to this generating set generates the image of $\partial_1^{\mu_2}$.

    On $\M_{0,4}$, because $\bar\M_{0,4}\cong \bP^1$, we have
    \[\partial_1(x-1)=D(1b|2a)-D(12|ab).\]
    We leverage commutativity of 
    \begin{center}
        \begin{tikzcd}
            \overline\CH^1(\M_{0,n},1)  \arrow[r,"\partial_1"] & \CH^0(\partial \bar\M_{0,n})\\
            \overline\CH^1(\M_{0,4},1) \arrow[r,"\partial_1"]\arrow[u,"\varphi^*"] & \CH^0(\partial \bar\M_{0,4})\arrow[u,"\varphi^*"]
        \end{tikzcd}
    \end{center}
    induced by morphisms $\varphi: \bar\M_{0,n}\to \bar\M_{0,4}$ with $\varphi(\M_{0,n})\subseteq \M_{0,4}$ to extend this computation to general $n$.

    For a general $n$, given a choice of $4$ markings $j,k,\ell,m$, we can consider the morphism
    \[\varphi_{j,k,\ell,m}:\bar\M_{0,n}\to \bar\M_{0,4}\]
    \[(C,p_1,\dots,p_n)\mapsto (C,p_j,p_k,p_\ell,p_,).\]
    Over $\M_{0,n}\subseteq \bar\M_{0,n}$, the points $p_j,p_k,p_\ell,p_m$ are distinct smooth points of $\bP^1$, so $\varphi_{j,k,\ell,m}$ maps into $\M_{0,4}$. By the commutativity of the above diagram, we thus have
    \[\partial_1(\varphi_{j,k,\ell,m}^*(x-1))=D(jm|k\ell)-D(jk|\ell m).\]
    
    We can compute $\varphi_{j,k,\ell,m}$ restricted to $\M_{0,n}$ under our identifications made in \ref{*} by applying transformations to send $p_j\mapsto \infty$, $p_\ell\mapsto 1$, $p_m\mapsto -1$. This gives the cross ratio, modified to fit with our conventions of $\M_{0,n}$. 
\begin{itemize}
    \item When $j=1$, $\ell=a$, and $m=b$, the cross ratio is $x_k$, and so the function $x-1$ on $\M_{0,4}$ pulls back to $x_k-1$. Thus, 
    \[\partial_1(x_k-1)=D(1b|ka)-D(1k|ab).\]
    \item Similarly, when $j=1,$ $\ell=b$, and $m=a$, the cross ratio is $-x_k$, so the function $x-1$ on $\M_{0,4}$ pulls back to $-x_k-1$ on $\M_{0,4}$, and so
    \[\partial_1(x_k+1)=D(1a|kb)-D(1k|ab).\]
    \item Finally, when $m=a$ and $j=b$, the cross ratio is 
    \[2\frac{(x_\ell-1)(x_k+1)}{(x_\ell+1)(x_k-1)}-1,\]
    so $x-1$ on $\M_{0,4}$ pulls back to 
    \[4\frac{x_\ell-x_k}{(x_\ell+1)(x_k-1)},\]
    meaning 
    \[\partial_1\left(\frac{x_\ell-x_k}{(x_\ell+1)(x_k-1)}\right)=D(\ell k|ab)-D(jb|\ell a).\]
\end{itemize}
    Thus, we have that the claimed expressions generate $\CH(\partial\bar\M_{0,n})$. They will freely generate if $\partial_1$ is injective. By the localization exact sequence
    \[\overline\CH(\bar\M_{0,n},1)\to \overline\CH(\M_{0,n},1) \xrightarrow{\partial_1} \CH(\partial\bar\M_{0,n}),\]
    it suffices to show $\overline\CH(\bar\M_{0,n},1)=0$. This is true by Proposition~\ref{IndZero} because $\bar\M_{0,n}$ is a smooth projective variety which has MKP by Proposition~\ref{MKPM0n}.
    
    Combining the above expressions we get
    {\small\[\partial_1(x_j-x_k)=D(jk|ab)-D(jb|ka)+D(1a|jb)-D(1j|ab)+D(1b|ka)-D(1k|ab).\]}
    A more useful way to write this is to expand out every term that involves each of the markings we are looking at, using, for example, $D(jk|ab)=D(1jk|ab)+D(jk|1ab)$. In doing this, there are many cancellations, and one gets the expression
    \[D(1ab|jk)+D(jkb|1a)+D(jka|1b)-D(1jk|ab)-D(kab|1j)-D(jab|1k)\]
    for $\partial_1(x_j-x_k)$. Note this expression is invariant under switching either $a,b$ or $j,k$, as it should be.

    Let $\pi: \bar\M_{0,n}\to [\bar\M_{0,n}/\mu_2]$ be the quotient map. Consider the diagram
    \begin{center}
        \begin{tikzcd}
        \overline\CH^1(\M_{0,n},1)\arrow[r,"\partial_1"] & \CH^0(\partial\bar\M_{0,n})\\
        \overline\CH^1([\M_{0,n}/\mu_2],1)\arrow[u,"\pi^*"]\arrow[r,"\partial_1^{\mu_2}"] & \CH^0([\partial\bar\M_{0,n}/\mu_2])\arrow[u,"\pi^*"]
        \end{tikzcd}
    \end{center}
    Using the above computations for $\partial_1$ and Lemma~\ref{units}$(4')$, we have
    \[\partial_1^{\mu_2}(x_p^2-1)=\widehat D(1b|pa)-2\widehat D(1p|ab)\]
\[\partial_1^{\mu_2}((x_2-x_3)(x_{p'}-x_{q'}))=\alpha_{2,3}+\alpha_{p,q}.\]
Thus, the elements from the statement do indeed generate.

We now wish to say that they generate the image freely. Because we took the image of a free generating set under $\partial_1^{\mu_2}$, it suffices to show that $\partial_1^{\mu_2}$ is injective. The localization exact sequence says
\[\overline\CH^1([\bar\M_{0,n}/\mu_2],1)\to \overline\CH^1([\M_{0,n}/\mu_2],1)\xrightarrow{\partial_1^{\mu_2}} \CH^0([\partial\bar\M_{0,n}/\mu_2]).\]
By Lemma~\ref{units}$(3')$, we have $\overline\CH^1([\bar\M_{0,n}/\mu_2],1)=(k^\times)^{\mu_2}/k^\times=0$, so $\partial_1^{\mu_2}$ is injective.
\end{proof}
\begin{cor}\label{uniquetorsion}
$\CH^1([\bar\M_{0.n}/\mu_2])$ has a unique nontrivial $2$-torsion element, which is pulled back along the canonical map $[\bar\M_{0,n}/\mu_2]\to B\mu_2$.  For $n=4$, it is given by $\widehat{T}-\widehat{D}(12|ab)$. For $n\geq 5$, it is given by
\[\alpha_{23}=\widehat D(23b|1a)+\widehat D(1ab|23)-\widehat D(123|ab)-\widehat D(3ab|12)-\widehat D(2ab|13).\]
\end{cor}
This should, in principle, follow directly from analyzing the relations for $\Pic([\bar\M_{0,n}/\mu_2])$ given by the theorem, but we take an alternative approach. 
\begin{proof}
    Torsion line bundles of order $N$ are equivalent to $\mu_N$-covers on proper Deligne-Mumford stacks. Because $\bar\M_{0,n}$ is a rational smooth projective variety, it has trivial \'etale fundamental group. Hence, we have 
    \[\pi_1^{\text{\'et}}([\bar\M_{0,n}/\mu_2])=\frac{\Z}{2\Z}.\]
    Thus, there is a unique nontrivial $2$-torsion line bundle which is pulled back from $B\mu_2$. 
    
    For $n=4$, we note $[\bar\M_{0,4}/\mu_2]\cong [\bP^1/\mu_2]$ and
    \[\CH([\bP^1/\mu_2])=\frac{\Z[u,v]}{(uv,2(u-v))}\]
    by the projective bundle formula, where $u$ and $v$ are the class of $0$ and $\infty$ respectively. Under the isomorphism, we have $0\mapsto \widehat{T}$ and $\infty\mapsto \widehat D(1a|2b)$, and so $\widehat{T}-\widehat D(12|ab)$ is the unique $2$-torsion class.
    
    For $n\geq 5$, Theorem~\ref{M0n/2} gives $\CH^1([\M_{0,n}/\mu_2])=0$, hence
    \[\Pic([\bar\M_{0,n}/\mu_2])=\CH^0([\partial\bar\M_{0,n}/\mu_2])/\im(\partial_1^{\mu_2}).\]
    Then, the theorem gives that $\alpha_{23}$ is $2$-torsion class in this group. Moreover, it is nontrivial as long as $\alpha_{23}$ is not in the image of $\partial_1^{\mu_2}$, which is true because $2\alpha_{23}$ is part of a free generating set of $\im(\partial_1^{\mu_2})$. 
\end{proof}
\begin{definition}
    We let $\alpha$ refer to the unique nontrivial $2$-torsion element of $\CH^1(\bar\M_{0,n})$.
\end{definition}

Now, all but one of the $\M^\Gamma$ appearing in the boundary of $\bar\M_{1,n}$ for $n\leq 4$ are isomorphic to either $[\M_{0,m}/\mu_2]$ or $\M_{0,m}$ for some $m$. There is one exception, where $\M^\Gamma$ is isomorphic to $[\M_{0,4}\times \M_{0,4}/\mu_2]$. We need the analogue of the results we have proven in this section for this space.

Consider the exact sequence
{\footnotesize\[\CH([\partial(\bar\M_{0,4}\times\bar\M_{0,4})/\mu_2])\xrightarrow{\iota_*} \CH([\bar\M_{0,4}\times\bar\M_{0,4}/\mu_2])\to \CH([\M_{0,4}\times\M_{0,4}/\mu_2])\to 0.\]}Because $[\M_{0,4}\times\M_{0,4}/\mu_2]$ is an open subset of $[\bA^2/\mu_2]$ containing the origin, Lemma~\ref{ContainOrigin} implies its Chow ring is isomorphic to $\frac{\Z[u]}{(2u)}$. 

Let $\pi_i: [\bar\M_{0,4}\times\bar\M_{0,4}/\mu_2]\to [\bar\M_{0,4}/\mu_2]$ be the projections onto the $i$th factor. As $\partial\bar\M_{0,4}=\{D(12|ab),D(1a|2b),D(1b|2a)\}$, we have that $[\partial(\bar\M_{0,4}\times\bar\M_{0,4})/\mu_2]$ has the $4$ components
\[\pi_1^{-1}\widehat{D}(12|ab),\pi_1^{-1}\widehat{D}(1a|2b),\pi_2^{-1}\widehat{D}(12|ab),\pi_2^{-1}\widehat{D}(1a|2b).\]
\begin{lemma}\label{M04M04/2}
    \mbox{}
    \begin{enumerate}
        \item The kernel of $\iota_*$ is generated by 
    \[[\pi_1^{-1}\widehat{D}(1a|2b)]-2[\pi_1^{-1}\widehat{D}(12|ab)],[\pi_2^{-1}\widehat{D}(1a|2b)]-2[\pi_2^{-1}\widehat{D}(12|ab)].\]
    
    \item We have 
    \[\CH([\M_{0,4}\times\M_{0,4}/\mu_2])\cong \CH(B\mu_2),\]
    and the surjection
    \[\CH([\bar\M_{0,4}\times\bar\M_{0,4}/\mu_2])\twoheadrightarrow \CH([\M_{0,4}\times\M_{0,4}/\mu_2])\]
    has a splitting given by 
    $\CH([\M_{0,4}\times\M_{0,4}/\mu_2])\cong \CH(B\mu_2)\xrightarrow{\nu^*} \CH([\bar\M_{0,4}\times\bar\M_{0,4}/\mu_2])$, where $\nu$ is the canonical morphism $[\bar\M_{0,4}\times \bar\M_{0,4}/\mu_2]\to B\mu_2$.
    \end{enumerate}
\end{lemma}
\begin{proof}
\mbox{}
    \begin{enumerate}
        \item The higher indecomposable Chow groups $\overline\CH([\M_{0,4}\times \M_{0,4}/\mu_2],1)$ can be computed in a similar manner to the proof of Theorem~\ref{M0n/2}. This gives
        \[\overline\CH^i([\M_{0,4}\times \M_{0,4}/\mu_2],1)=\overline\CH^1([\M_{0,4}\times \M_{0,4}/\mu_2],1)=\langle x_2^2-1,y_2^2-1\rangle\cong \Z^2.\]
        By using push-pull on the diagram
        \begin{center}
            \begin{tikzcd}
                \overline\CH^1(\left[\M_{0,4}\times \M_{0,4}/\mu_2\right],1) \arrow[r,"\partial_1'"] & \CH^0(\left[\partial\bar\M_{0,4}\times\bar\M_{0,4})/\mu_2\right])\\
                \overline\CH^1(\left[\M_{0,4}/\mu_2\right],1) \arrow[r,"\partial_1"]\arrow[u,"\pi_i^*"] & \CH^0(\left[\partial\bar\M_{0,4}/\mu_2\right])\arrow[u,"\pi_i^*"] 
            \end{tikzcd}
        \end{center}
        for $i=1,2$, we see that $\im(\partial_1')=\ker(\iota_*)$ is generated by 
         \[[\pi_1^{-1}\widehat{D}(1a|2b)]-2[\pi_1^{-1}\widehat{D}(12|ab)],[\pi_2^{-1}\widehat{D}(1a|2b)]-2[\pi_2^{-1}\widehat{D}(12|ab)]\]
        \item We have that the canonical map $\CH(B\mu_2)\to \CH([\M_{0,4}\times\M_{0,4}/\mu_2])$ is an isomorphism because $\M_{0,4}\times \M_{0,4}$ is an open inside $\bA^2$ not containing the origin. The same argument given in Lemma~\ref{barM04/2} shows the claimed map is a splitting. \qedhere
    \end{enumerate}
\end{proof}

Finally, we will need to use the Chow rings $\CH(\bar\M_{0,n})$ for our computations in the next section, so we recall Keel's presentation from \cite{Keel}. Let $D_I:=D(I|I^c)$, where $I^c$ denotes the complement of $I$ in $[n]$.

\begin{thm}\label{barM0n}
    The Chow ring of $\bar\M_{0,n}$ is generated by the classes $D_I$ for $I\subseteq [n]$ with $\#I\notin \{0,1,n-1,n\}$, subject to the relations
    \begin{align*}
        &D_{I}-D_{I^c} & &\\
        &D_ID_J  & I \not\subseteq J,  J^c \text{ and }  J \not\subseteq I, I^c &\\
        &\sum_{\substack{I\ni j,k\\ I\not\ni \ell,m}} D_I-\sum_{\substack{I\ni j,\ell\\ I\not\ni k,m}} D_I & j,k,\ell,m \in [n] \text{ distinct} &
    \end{align*}
\end{thm}

\section{The Chow Ring of $\bar\M_{1,n}$ for $n=2,3,4$}

We can use Lemma~\ref{Module} with the stack $\bar\M_{1,1}$ to give the Chow groups of a stack with a morphism to $\bar\M_{1,1}$ the structure of graded modules over 
\[A:=\CH(\bar\M_{1,1})=\frac{\Z[\lambda]}{(24\lambda^2)},\]
using Theorem~\ref{barM11}.
Note that $A$ is a quotient of the ring $\Lambda$ meaning that the $\Lambda$-module structure determines the $A$-module structure of an $A$-module. This allows us to transform our $\Lambda$-module presentations for $\CH(\M_{1,n})$ and $\overline \CH(\M_{1,n},1)$ from previous sections into $A$-module presentations.  
Note any substack of $\bar\M_{1,n}$ has a map to $\bar\M_{1,1}$ by composing with the morphism forgetting all but the first point $\bar\M_{1,n}\to \bar\M_{1,1}$, and hence has a graded $A$-module structure on its Chow group. Moreover, for every stable graph $\Gamma$ of genus $1$ with $n$ marked points, we have 
\[\bar\M_\Gamma\to [\bar\M_\Gamma/\Aut(\Gamma)] \to \bar\M^\Gamma \to \bar\M_{1,n}\to \bar\M_{1,1},\]
giving a graded $A$-module structure on $\CH(\bar\M_\Gamma), \CH([\bar\M_\Gamma/\Aut(\Gamma)])$, and $\CH(\bar\M^\Gamma).$
 
\subsection{The Graphs $\Delta_S$ and $\Phi$}
\begin{definition}
For a subset $S\subseteq [n]$, let $\Delta_S$ be the $n$ -pointed, genus $1$ stable graph with one genus $1$ vertex and one genus $0$ vertex, with the marked points in $S$ attached to the genus $1$ vertex. That is, $\Delta_S$ is the graph
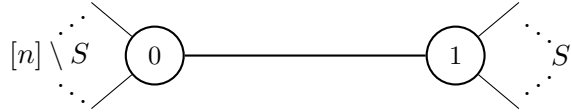
\[\begin{tikzpicture}[
  vertex/.style={
    circle, draw, thick, minimum size=2em, inner sep=0pt, font=\small
  }
]

  \node[vertex] (v0) at (0,0) {$0$};
  \node[vertex] (vg) at (4,0) {$1$};

  \draw[thick] (v0) -- (vg);

  \draw (v0.140) -- ++(140:0.7) node[above] {};
  \draw (v0.-140) -- ++(-140:0.7) node[below] {};
  \node at (-1.1,.4) {$\udots$};
  \node at (-1.1,-.4) {$\ddots$};
  \node at (-1.4,0) {$[n]\setminus S$};

  
  \draw (vg.40)  -- ++(40:0.7)  node[above] {};
  \node[] at (5.1,.4) {$\ddots$};

  \node[] at (5.4,0) {$S$};
   \node[] at (5.1,-.4) {$\udots$};
  
  \draw (vg.-40) -- ++(-40:0.7)  node[below] {};

\end{tikzpicture}.\]
For two subsets $S,T\subseteq [n]$ with $\codim(\bar\M^{\Delta_S}\cap \bar\M^{\Delta_T})=2$, we write $\Delta_{S|T}$ for the graph satisfying $\bar\M^{\Delta_S}\cap \bar\M^{\Delta_T}=\bar\M^{\Delta_{S|T}}$, which exists by Theorem~\ref{tautologicalcalculus}. The graph $\Delta_{S|T|U}$ is defined similarly. 
\end{definition}
In this section and the next, given a stable graph, always denoted by a Greek letter with a possible subscript $\Gamma_J$, we use the lowercase Greek letter with the same subscript $\gamma_J$ to refer to $[\bar\M^{\Gamma_J}]$ inside of the Chow group of a substack of $\partial\bar\M_{1,n}$. We write a $*$ inside of a subscript of a $\delta$ if the class does not depend on the choice of set $S$.

Note that the $n$ is not included in the notation for $\Delta_S$ nor $\Phi$, so ``$\Delta_1$,'' for example, on its own is ambiguous. It will be clear in context which $n$ is meant.

The following Lemma shows that the multiplication by $\lambda$ on $\CH(\bar\M_{\Delta_S})$ from the $A$-module structure is what one would think it is.

\begin{lemma}\label{SepAStruct}
Consider $\bar\M_{\Delta_S}= \bar\M_{1,\#S+1}\times \bar\M_{0,n-\#S+1}$. The class $\lambda\in \opCH^1(\bar\M_{\Delta_S})=\CH^1(\bar\M_{\Delta_S})$ is equal to the pullback of $\lambda$ along the projection $\bar\M_{\Delta_S}\to \bar\M_{1,\#S+1}$.
\end{lemma}
\begin{proof}
This follows from the commutativity of
\begin{center}
\begin{tikzcd}
\bar\M_{\Delta_S} \arrow[r,"\xi_{\Delta_S}"]\arrow[dr] & \bar\M_{1,n}\arrow[d,"\pi"]\\
& \bar\M_{1,\#S}
\end{tikzcd}
\end{center}
where $\pi$ is the map given by forgetting points outside of $S$.
\end{proof}

\begin{definition}
Let $\Phi$ be the $n$-pointed, genus $1$ stable graph with a unique vertex, which has genus $0$, and a self edge.
\end{definition}

\begin{lemma}\label{NonSepAStruct}
Let $\Gamma$ be an $n$-pointed genus $1$ graph with only nonseparating nodes. The class $\lambda\in \opCH^1([\bar\M_{\Gamma}/\Aut(\Gamma)])=\CH^1([\bar\M_{\Gamma}/\Aut(\Gamma)])$ is equal to the pullback of $s\in \CH^1(B\mu_2)$ along the representable $[\bar\M_{\Gamma}/\Aut(\Gamma)]\to \partial\bar\M_{1,1}\cong B\mu_2$.
\end{lemma}
\begin{proof}
For $\pi: \bar\M_{1,n}\to \bar\M_{1,1}$ the map forgetting all but the first point, we have the commutative diagram
\begin{center}
\begin{tikzcd}
\left[\bar\M_{\Gamma}/\Aut(\Gamma)\right]\arrow[d]\arrow[r,"\xi_\Phi"] & \bar\M_{1,n}\arrow[d,"\pi"]\\
B\mu_2\cong \partial\bar\M_{1,1}\arrow[r,"\xi_{\Phi}"] & \bar\M_{1,1}.
\end{tikzcd}
\end{center}
Then the lemma is true by Theorem~\ref{barM11}.
\end{proof}

\subsection{$\bar\M_{1,2}$}
\label{secbarM12}
Define the stable graph $\Theta$ to be
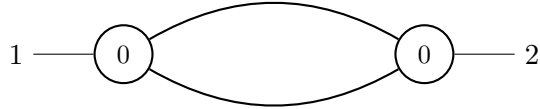
\[\begin{tikzpicture}[
  vertex/.style={
    circle, draw, thick, minimum size=2em, inner sep=0pt, font=\small
  }
]

  \node[vertex] (v0) at (0,0) {$0$};
  \node[vertex] (v1) at (4,0) {$0$};

  \draw[thick, bend left=30] (v0) to (v1);
  \draw[thick, bend right=30] (v0) to (v1);

  \draw (v0.180) -- ++(180:0.8) node[left] {$1$};

  \draw (v1.0) -- ++(0:0.8) node[right] {$2$};

\end{tikzpicture}.\]

\begin{prop}\label{boundarybarM12}
    The Chow group $\CH(\partial\bar\M_{1,2})$ is given by
\[\CH(\partial\bar\M_{1,2})=\frac{A\langle \delta_\emptyset, \phi, \theta\rangle}{\langle 2\lambda\phi, 2\theta-24\lambda\delta_\emptyset\rangle},\]
where $\delta_\emptyset, \phi$ have degree $0$ and $\theta$ has degree $1$. 
\end{prop}
\begin{proof}
    There is only one component in $\bar\M^{\sep,\geq 3}=\bar\M^{\sep}$, corresponding to the graph $\Delta_\emptyset$. Thus, $\bar\M^\sep=\bar\M^{\Delta_\emptyset}$. By Proposition~\ref{SeparatingProduct},  we have 
    \[\CH(\bar\M^\sep)=\CH(\bar\M^{\Delta_\emptyset})=\CH(\bar\M_{1,1})=A.\] 
    Remembering just the $A$-module structure, we can write $\CH(\bar\M^\sep)=A\langle \delta_\emptyset\rangle.$
    
    Next, by Proposition~\ref{OpenStrata}, we have 
    \[\bar\M^\Theta=\M^\Theta=[\M_{0,3}\times \M_{0,3}/\mu_2]=B\mu_2.\]
    This implies $\CH(\M^\Theta)=\frac{\Z[u]}{(2u)}$. By Lemma~\ref{NonSepAStruct}, we have that $\lambda$ acts by multiplication by $u$, so we can write $\CH(\M^\Theta)=\frac{A\langle \theta\rangle}{\langle 2\lambda\theta\rangle}.$ Because $\bar\M^{\sep,\geq 2}=\M^\Theta\amalg \bar\M^{\sep}$ as spaces, we have
    \[\CH(\bar\M^{\sep,\geq 2})=\CH(\bar\M^\sep)\oplus \CH(\M^\Theta)=\frac{A\langle \delta_\emptyset, \theta\rangle}{\langle 2\lambda\theta \rangle}.\]

    Next, we consider $\partial\bar\M_{1,2}=\bar\M^{\sep,\geq 1}=\bar\M^{\sep,\geq 2}\cup \M^\Phi$. Note
    \[\M^\Phi=[\M_\Phi/\Aut(\Phi)]=[\M_{0,4}/\mu_2]\]
    by Proposition~\ref{OpenStrata}. By Theorem~\ref{WDVV mod 2}, we know that the image of 
    \[\partial_1^\Phi:\overline\CH(\M^\Phi,1)\to \CH([\partial\bar\M_\Phi/\Aut(\Phi)])\] 
    is generated by $\widehat{D}(1a|2b)-2\widehat{D}(12|ab)$. We can compute the pushforward along $\xi_{\Phi}$ to be 
    \[\widehat{D}(12|ab)\mapsto 2\theta\]
    \[\widehat{D}(1a|2b)\mapsto 12\lambda \delta_\emptyset\]
    using Lemma~\ref{automorphism}. Thus, $\xi_{\Phi*}(\im(\partial_1^{\Phi}))$ is generated by $2\theta-24\lambda\delta_\emptyset$. 
    
    Additionally, by Lemma~\ref{barM04/2}, we have a splitting of $\CH([\bar\M_\Phi/\Aut(\Phi)])\twoheadrightarrow \CH(\M^\Phi)$, given by $u^i\mapsto \nu^*(u^i)$, where $\nu: [\bar\M_{0,4}/\mu_2]\to B\mu_2$. This is $A$-linear because it is given by a pullback, so Proposition~\ref{filtration} gives
    \[\CH(\partial\bar\M_{1,2})=\frac{\CH(\bar\M^{\sep,\geq 2})}{\langle 2\theta-24\lambda\delta_\emptyset\rangle}\oplus \frac{A\langle \phi\rangle}{\langle 2\lambda\phi\rangle}=\frac{A\langle \delta_\emptyset, \phi, \theta\rangle}{\langle 2\lambda\phi, 2\theta-24\lambda\delta_\emptyset\rangle}. \qedhere\]
\end{proof}    

\begin{prop}\label{higherChowImM12}
We have 
\[\partial_1(\mathfrak p)=\theta-12\lambda\delta_\emptyset-\lambda\phi\in \CH(\partial\bar\M_{1,2}).\]
\end{prop}
\begin{proof}
    By the definition of $\mathfrak p$, it is pushed forward from $Y_1(2)\subseteq \M_{1,2}$, which we also refer to by $\mathfrak p$. 

    We have the family $y^2=x^3+ax-(x_2^3+ax_2)$ over 
    \[\{(a,x_2)\in \bA^2|(a,x_2)\neq (0,0)\},\]
    which has two sections, given by the point at infinity and $(x_2,0)$. This space and the family have compatible actions of $\Gm$, where $a$ has weight $-2$, $x$ and $x_2$ have weight $-2$, and $y$ has weight $-3$. Thus, we get a family of curves over 
    \[X_1(2):=[[\{(a,x_2)\in \bA^2|(a,x_2)\neq (0,0)\}/\Gm].\]
    This family is not stable because when $a=-3x^2$, the second marked point is at a node, but we can blow up the node at this point to get a stable family. Thus, we get a morphism $f:X_1(2)\to \bar\M_{1,2}$.  We see that $f$ restricts to the inclusion $Y_1(2)\subseteq \M_{1,2}$ by the definition of $Y_1(2)$. By checking the two points in $\partial X_1(2):=X_1(2)\setminus Y_1(2)$, we see that $f$ is representable and degree $1$ on closed points.

    Now, we consider the commutative diagram
    \begin{center}
        \begin{tikzcd}
            \overline \CH(Y_1(2),1)\arrow[r,"\partial_1"]\arrow[d,"f_*"] & \CH(\partial X_1(2))\arrow[d,"f_*"]\\
            \overline \CH(\M_{1,2},1) \arrow[r,"\partial_1"] & \CH(\partial \bar\M_{1,2}).
        \end{tikzcd}
    \end{center}
    Using Lemma~\ref{units}$(4')$, we have
    \[\partial_1(\mathfrak p)=[V(a+3x_2^2)/\Gm]-[V(4a+3x_2^2)/\Gm].\]
    One can compute that the images of $[V(a+3x_2^2)/\Gm]$ and $[V(4a+3x_2^2)/\Gm]$ under $f$ are respectively $\bar\M^\Theta$ and the nodal curve with $2$-torsion second marked point, whose class we call $\theta'$. Thus, we have 
    \[\partial_1(\mathfrak p)=\theta-\theta'\in \CH(\partial\bar\M_{1,2}).\]

Now, we want to compute $\theta'$ in terms of our generators for $\partial\bar\M_{1,2}$. We have the morphism
\[\xi_{\Phi}: [\bar\M_{0,4}/\mu_2]\to \bar\M^\Phi.\]
By Lemma~\ref{NonSepAStruct}, we have $\lambda$ is acting by the pullback of $u\in \CH^1(B\mu_2)$ along $\nu: [\bar\M_{0,n}/\mu_2]$, so $\xi_{\Phi*}(\nu^*(u))=\lambda\phi$. By Corollary~\ref{uniquetorsion}, we have
\[\nu^*(u)=\widehat{T}-\widehat D(12|ab).\] 
We know $\widehat{D}(12|ab)$ pushes forward to $12\lambda\delta_\emptyset$, and we claim $\widehat{T}$ pushes forward to $\theta'$. Once we know this, the proposition is proven because then pushing forward the above gives
\[\lambda\phi=\theta'-12\lambda\delta_\emptyset,\]
and hence
\[\partial_1(\mathfrak p)=\theta-\theta'=\theta-(12\lambda\delta_\emptyset+\nu^*(u))=\theta-12\lambda\delta_\emptyset-\lambda\phi.\]

Note that the image of $\widehat T$ under $\xi_{\Phi}$ is in $\M^\Phi$, since $\widehat{T}\in [\M_{0,4}/\mu_2]$. Moreover, this map is representable; because $\widehat{T}\cong B\mu_2$, we have $\xi_{\Phi}(\widehat{T})$ has a nontrivial automorphism. The points in $\M^{\Phi}$ correspond to irreducible nodal curves with two marked points. The only automorphism of a nodal curve with one marked point is given by inversion, in terms of the group law. Thus, the second marked point must be  $2$-torsion with respect to the first, and this is indeed the point whose class is $\theta'$. 
\end{proof}
\begin{rmk}
    The map $f:X_1(2)\to \bar\M_{1,2}$ is actually a closed embedding \cite[Lemma 3.17]{Pagani}.
\end{rmk}

\begin{thm}\label{barM12}
    The Chow ring of $\bar\M_{1,2}$ is given by
    \[\CH(\bar\M_{1,2})=\frac{\Z[\lambda,{\delta_\emptyset}]}{(24\lambda^2,{\delta_\emptyset}(\lambda+{\delta_\emptyset}))}.\]
    As an $A$-module, this is $\CH(\bar\M_{1,2})=A\langle 1,\delta_\emptyset\rangle$. Finally, \[\theta=12\lambda\delta_\emptyset+12\lambda^2\] inside $\CH(\bar\M_{1,2})$. 
\end{thm}
\begin{proof}
The localization exact sequence for $\partial\bar\M_{1,2}\subseteq \bar\M_{1,2}$ is
    \[0\to  \frac{\CH(\partial\bar\M_{1,2})}{\im(\partial_1)}\to \CH(\bar\M_{1,2})\to \CH(\M_{1,2})\to 0.\]
Using Proposition~\ref{boundarybarM12} and Proposition~\ref{higherChowImM12}, we have 
\[\frac{\CH(\partial\bar\M_{1,2})}{\im(\partial_1)}=\frac{A\langle \delta_\emptyset,\phi,\theta\rangle}{\langle 2\lambda\phi,2\theta-24\lambda\delta_\emptyset,\theta-12\lambda\delta_\emptyset-\lambda\phi\rangle}=\frac{A\langle \delta_\emptyset,\phi\rangle}{\langle 2\lambda\phi\rangle},\]
remembering that $\theta=12\lambda\delta_\emptyset+\lambda\phi$ inside $\CH(\partial\bar\M_{1,2})/\im(\partial_1)$.
Additionally, our description of $\CH(\M_{1,2})$ from Theorem~\ref{M12} can be written as
\[\CH(\M_{1,2})=\frac{\Z[\lambda]}{(12\lambda)}=\frac{A\langle 1\rangle}{\langle 12\lambda\cdot 1\rangle }.\]
Thus, our exact sequence is
\[0\to \frac{A\langle \delta_\emptyset,\phi\rangle}{\langle 2\lambda\phi\rangle}\to \CH(\bar\M_{1,2})\to \frac{A\langle 1\rangle}{\langle 12\lambda\cdot 1\rangle }\to 0.\]
By Proposition~\ref{Order 12}, we know that $\phi=12\lambda\cdot 1$ inside of $\CH(\bar\M_{1,2})$. We can use this and Lemma~\ref{exactPresentation} to get
\[\CH(\bar\M_{1,2})=\frac{A\langle 1,\delta_\emptyset,\phi\rangle}{\langle 2\lambda\phi,\phi-12\lambda\cdot 1\rangle}=A\langle 1,\delta_\emptyset\rangle.\]
This describes $\CH(\bar\M_{1,2})$ as an $A$-module.

By Theorem~\ref{tautologicalcalculus} and Proposition~\ref{psi}, we can compute
\[\xi_{\Delta_\emptyset}^*(\delta_\emptyset)=-\lambda,\]
hence
\[\delta_\emptyset^2=\xi_{\Delta_\emptyset*}(\xi_{\Delta_\emptyset}^*(\delta_\emptyset))=\xi_{\Delta_\emptyset}(-\lambda)=-\lambda\delta_\emptyset.\]
Thus, we have
\[\CH(\bar\M_{1,2})=\frac{A[\delta_\emptyset]}{\delta_\emptyset^2+\lambda\delta_\emptyset)}=\frac{\Z[\lambda,\delta_\emptyset]}{(24\lambda^2,\delta_\emptyset^2+\lambda\delta_\emptyset)}.\]

Finally, inside $\CH(\bar\M_{1,2}),$ we have
\[\theta=12\lambda\delta_\emptyset+\lambda\phi=12\lambda\delta_\emptyset+12\lambda^2. \qedhere\]
\end{proof}

\subsection{$\bar\M_{1,3}$}

Define the stable graphs
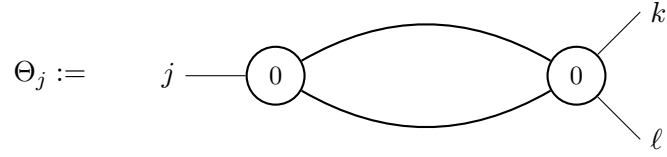
\[\begin{tikzpicture}[
  vertex/.style={
    circle, draw, thick, minimum size=2em, inner sep=0pt, font=\small
  }
]

  \node[vertex] (v0) at (0,0) {$0$};
  \node[vertex] (v1) at (4,0) {$0$};

  \draw[thick, bend left=30] (v0) to (v1);
  \draw[thick, bend right=30] (v0) to (v1);

  \draw (v0.180) -- ++(180:0.8) node[left] {$j$};

  \draw (v1.45) -- ++(45:0.8) node[right] {$k$};
  \draw (v1.-45) -- ++(-45:0.8) node[right] {$\ell$};

  \node at (-3,0) {$\Theta_j:=$};

\end{tikzpicture}\]
where $\{j,k,\ell\}=[3]$, and
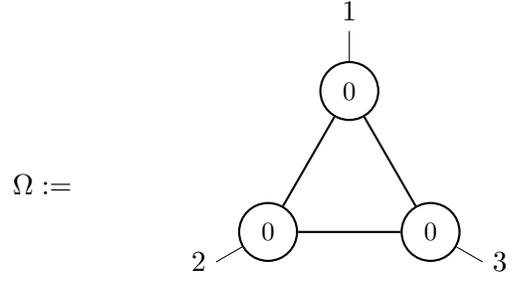
\[\begin{tikzpicture}[
  vertex/.style={
    circle, draw, thick, minimum size=2em, inner sep=0pt, font=\small
  }
]

  \node[vertex] (v1) at (-1.92,1.25) {$0$};         
  \node[vertex] (v2) at (-3,-0.625) {$0$};   
  \node[vertex] (v3) at (-.84,-0.625) {$0$};    

  \draw[thick] (v1) -- (v2);
  \draw[thick] (v2) -- (v3);
  \draw[thick] (v3) -- (v1);

  \draw (v1.90) --++(90:0.4) node[above] {1};
  \draw (v2.210) --++(210:0.4) node[left] {2};
  \draw (v3.330) --++(330:0.4) node[right] {3};

    \node at (-6,0) {$\Omega:=$};
\end{tikzpicture}.\]

\begin{prop}\label{boundarybarM13}
The Chow group $\CH(\partial \bar\M_{1,3})$ is given by
    \[A\langle \{\delta_j,\theta_j\}_j,\delta_\emptyset,\delta_{\emptyset|*}\rangle/R\]
    where $R$ is the $A$-submodule of relations generated by
    \begin{align*}
      &2\lambda\theta_j \\  
      &2\lambda\phi\\
      &\lambda\phi-2\theta_1-12\lambda(\delta_1-\delta_\emptyset-\delta_2-\delta_3)\\
      &2\theta_j+2\theta_k-24\lambda\delta_\ell-2\lambda\delta_\emptyset
    \end{align*}
    for $\{j,k,\ell\}=[3]$. Moreover, $\omega=24\lambda\delta_{\emptyset|*}$ inside $\CH(\partial\bar\M_{1,3})$. Finally, the map $\overline\CH^2(\bar\M_{1,3},1)\to \overline\CH^2(\M_{1,3}^\irr,1)$ is $0$.
\end{prop}

\begin{proof}
The components of $\bar\M^{\sep.\geq 4}=\bar\M^\sep$ are given by $\bar\M^{\Delta_j}$ for $j\in [3]$ and $\bar\M^{\Delta_\emptyset}$. 

By Proposition~\ref{SeparatingProduct}, we have that
\[\CH(\bar\M^{\Delta_\emptyset})\cong \CH(\bar\M_{\Delta_\emptyset})=\CH(\bar\M_{1,1}\times \bar\M_{0,4}).\]
By Proposition~\ref{MKPM0n}, $\bar\M_{0,4}$ has the MKP, so by Proposition~\ref{MKPprops}(6), we have
\[\CH(\bar\M_{1,1}\times\bar\M_{0,4})=\CH(\bar\M_{1,1})\otimes\CH(\bar\M_{0,4})=A\otimes \CH(\bar\M_{0,4}).\]
By Theorem~\ref{barM0n}, we can write
\[\CH(\bar\M_{0,4})=\Z\langle 1,D \rangle,\]
where $D$ is the class of any boundary divisor, so 
\[\CH(\bar\M_{1,1})\otimes\CH(\bar\M_{0,4})=A\langle 1,D\rangle.\]
Unwinding the isomorphisms, we get
\[\CH(\bar\M^{\Delta_\emptyset})=A\langle \delta_\emptyset,\delta_{\emptyset|*}\rangle,\]
where we write $*$ to indicate any one of $1,2,3$, as all $\delta_{\emptyset|j}$ are equal inside $\CH(\bar\M^{\Delta_\emptyset})$. Similarly, Proposition~\ref{SeparatingProduct} and Theorem~\ref{barM12} give 
\[\CH(\bar\M^{\Delta_j})=A\langle \delta_j,\delta_{\emptyset|j}\rangle\]
for $j\in [3].$

Using Theorem~\ref{tautologicalcalculus}, we can compute the intersections between these components to be
\begin{align*}
    & \bar\M^{\Delta_j}\cap\bar\M^{\Delta_k}=\emptyset \text{ for }j\neq k\\
    & \bar\M^{\Delta_j}\cap \bar\M^{\Delta_\emptyset}=\bar\M^{\Delta_{\emptyset|j}}.
\end{align*}
Proposition~\ref{SeparatingProduct} gives that 
\[\CH(\bar\M^{\Delta_{j,\emptyset}})\cong \CH(\bar\M_{1,1}\times \bar\M_{0,3}\times\bar\M_{0,3})\]
for $j\in [3]$. Because $\bar\M_{0,3}$ is just a point and $\CH(\bar\M_{1,1})=A$, we then have
\[\CH(\bar\M^{\Delta_{j,\emptyset}})=A\langle \delta_{\emptyset|j}\rangle.\]
Then the exact sequence
\[\bigoplus_{j\in[3]} \CH(\bar\M^{\Delta_{\emptyset|j}}) \to \bigoplus_{j\in\{1,2,3,\emptyset\}} \CH(\bar\M^{\Delta_j})\to \CH(\bar\M^\sep)\to 0\]
gives 
\[\CH(\bar\M^\sep)=A\langle\{\delta_j\}_j,\delta_\emptyset,\delta_{\emptyset|*} \rangle.\]

Next, by Proposition~\ref{OpenStrata}, we have
\[\bar\M^\Omega=\M^\Omega=[\M_\Omega/\Aut(\Omega)]=\M_{0,3}\times \M_{0,3}\times \M_{0,3}=\Spec(k).\]
Because $\bar\M^{\sep,\geq 3}=\bar\M^\Omega\amalg \bar\M^\sep$ as spaces, we have 
\[\CH(\bar\M^{\sep,\geq 3})=\CH(\bar\M^{\sep})\oplus \Z=A\langle\{\delta_j\}_j,\delta_\emptyset,\delta_{\emptyset|*},\omega \rangle/\langle \lambda \omega \rangle.\] 

Next, we consider $\bar\M^{\sep,\geq 2}=\bar\M^{\sep,\geq 3}\cup \M^{\non=2}$. We have $\M^{\non=2}=\M^{\Theta_1}\amalg\M^{\Theta_2}\amalg\M^{\Theta_3}$. Fix $j\in [3]$. Note 
\[\M^{\Theta_j}=[\M_{\Theta_j}/\Aut(\Theta_j)]=[\M_{0,4}\times\M_{0,3}/\mu_2]\cong [\M_{0,4}/\mu_2]\]
by Proposition~\ref{OpenStrata}. We label the markings getting glued on $[\M_{0,4}/\mu_2]$ by $a,b$, and the markings not getting glued by their label after the gluing. Let $k,\ell\in [3]$ be so that $[3]=\{j,k,\ell\}$. By Theorem~\ref{WDVV mod 2}, we know that the image of 
\[\partial_1^{\Theta_j}:\overline\CH(\M^{\Theta_j},1)\to \CH([\partial\bar\M_{\Theta_j}/\Aut(\Theta_j)])\]
is generated by $\widehat{D}(ka|\ell b)-2\widehat{D}(k \ell|ab)$. We can compute the pushforwards along $\xi_{\Theta_j}$ to be 
    \[\widehat{D}(ka|\ell b)\mapsto \omega\]
    \[\widehat{D}(k\ell|ab) \mapsto 12\lambda (\delta_{\emptyset|*}+\lambda\delta_j),\]
    using Lemma~\ref{automorphism} and the formula for $\theta$ from Theorem~\ref{barM12}. Thus, we have $\xi_{\Theta*}(\im(\partial_1^{\Theta_j}))$ is generated by $\omega-24\lambda (\delta_{\emptyset|*}+\lambda\delta_j)=\omega-24\lambda\delta_{\emptyset|*}$. 
    
   Moreover, $\CH(\M^{\Theta_j})=\frac{\Z[u]}{(2u)}$ by Theorem~\ref{M0n/2}.  By Lemma~\ref{NonSepAStruct}, we have that $\lambda$ acts by multiplication by $u$, so we can write $\CH(\M^{\Theta_j})=\frac{A\langle \theta_j\rangle}{\langle 2\lambda\theta_j\rangle}.$ By Lemma~\ref{barM04/2}, we have a splitting of $\CH([\bar\M_{\Theta_j}/\Aut(\Theta_j)])\twoheadrightarrow \CH(\M^{\Theta_j})$, given by $u^i\mapsto \nu^*(u^i)$, where $\nu: [\bar\M_{0,4}/\mu_2]\to B\mu_2$. This splitting is $A$-linear, as it is given by a pullback. Thus, Proposition~\ref{filtration} gives
    \begin{align*}
        \CH(\bar\M^{\sep,\geq 2})&=\frac{\CH(\bar\M^{\sep,\geq 3})}{\langle\omega-24\lambda\delta_{\emptyset|*}\rangle}\oplus\bigoplus_j \frac{A\langle \theta_j\rangle}{\langle 2\lambda \theta_j\rangle}\\
        &=\frac{A\langle \{\delta_j,\theta_j\}_j,\delta_\emptyset,\delta_{\emptyset|*},\omega \rangle}{\langle \lambda \omega,\omega-24\lambda\delta_{\emptyset|*},\{2\lambda\theta_j\}_j \rangle}\\
        &=\frac{A\langle \{\delta_j,\theta_j\}_j,\delta_\emptyset,\delta_{\emptyset|*}\rangle}{\langle \{2\lambda\theta_j\}_j \rangle}
    \end{align*}

    remembering that $\omega=24\lambda\delta_{\emptyset|*}$ in $\CH(\bar\M^{\sep,\geq 2})$.

    Finally, we consider $\partial\bar\M_{1,3}=\bar\M^{\sep,\geq 1}=\bar\M^{\sep,\geq 2}\cup \M^{\non=1}$. We have $\M^{\non=1}=\M^\Phi$ and by Proposition~\ref{OpenStrata}, we have
    \[\M^\Phi=[\M_\Phi/\Aut(\Phi)]=[\M_{0,5}/\mu_2].\]
     We label the markings getting glued on $[\M_{0,5}/\mu_2]$ by $a,b$, and the markings not getting glued by their label after the gluing. By Theorem~\ref{WDVV mod 2}, the image of 
    \[\partial_1^\Phi:\overline\CH(\M^\Phi,1)\to \CH([\partial\bar\M_\Phi/\Aut(\Phi)])\]
    is generated by
    \[\widehat{D}(1a|2b)-2\widehat{D}(12|ab)\]
    \[\widehat{D}(1a|3b)-2\widehat{D}(13|ab)\]
    \[2(\widehat{D}(23b|1a)+\widehat{D}(1ab|23)-\widehat{D}(123|ab)-\widehat{D}(3ab|12)-\widehat{D}(2ab|13)).\]
   We can compute the pushforwards along $\xi_{\Phi}$ to be
   \[\widehat{D}(ja|k\ell b)\mapsto 2\theta_j\]
   \[\widehat{D}(ab|123)\mapsto 12\lambda\delta_\emptyset\]
   \[\widehat{D}(jk|\ell ab)\mapsto 12\lambda\delta_\ell\]
   using Lemma~\ref{automorphism}. Thus, $\xi_{\Phi}(\im(\partial_1^\Phi))$ is generated by 
    \[2\theta_1+2\theta_2-24\lambda\delta_3-24\lambda\delta_\emptyset\]
    \[2\theta_1+2\theta_3-24\lambda\delta_2-24\lambda\delta_\emptyset\]
    \[2(2\theta_1+12\lambda(\delta_1-\delta_\emptyset-\delta_2-\delta_3)).\]
    
    Additionally, $\CH([\M_{0,5}/\mu_2])=\Z=A/(\lambda)$ by Theorem~\ref{M0n/2}. By Proposition~\ref{filtration}, we have an exact sequence
    \[0\to \frac{\CH(\bar\M^{\sep,\geq 2})}{\im(\partial_1^\Phi)}\to \CH(\partial\bar\M_{1,3})\to \frac{A\langle \phi\rangle}{\langle \lambda\phi\rangle}\to 0.\]
    We use Lemma~\ref{exactPresentation} to get an $A$-module presentation of the middle term. Note that by Lemma~\ref{NonSepAStruct} and Corollary~\ref{uniquetorsion}, we have that 
    \[\lambda\phi=\xi_{\Phi*}(\alpha)=2\theta_1+12\lambda(\delta_1-\delta_\emptyset-\delta_2-\delta_3)\]
    inside $\CH([\bar\M_{\Phi}/\Aut(\Phi)]$, and hence inside $\CH(\partial\bar\M_{1,3})$. This gives the presentation for $\CH(\partial\bar\M_{1,3})$ in the statement. 

    Finally, we want to say that $\overline\CH^2(\bar\M_{1,3},1)\to \overline\CH^2(\M_{1,3}^\irr,1)$ is the zero map. Note $\M_{1,3}^\irr$ is the complement in $\bar\M_{1,3}$ of $\bar\M_{1,3}^{\sep,\geq 2}$, giving a localization exact sequence
    \[\overline\CH^2(\bar\M_{1,3},1)\to \overline\CH^2(\M_{1,3}^\irr,1)\xrightarrow{\partial_1} \CH_1(\bar\M^{\sep,\geq 2})\xrightarrow{\iota_*} \CH^2(\bar\M_{1,3}).\]
    So, it suffices to show that $\partial_1$ is injective. By Proposition~\ref{irr}, we have that $\overline\CH^2(\M_{1,3}^\irr,1)$ is free of rank at most $3$, so it would suffice to show that $\iota_*$ has a kernel of rank at least $3$. Note that $\iota_*$ factors as 
    \[\CH_1(\bar\M^{\sep,\geq 2})\to \CH_1(\partial\bar\M_{1,3})\to \CH^2(\bar\M_{1,3}),\]
    and we have just seen that the first map has a kernel of rank $3$, so the statement is proved.
\end{proof}

\begin{prop}\label{pullbackhigherChowImM13}
We have
\[\partial_1(\mathfrak p_{jk})=\theta_j+12\lambda\delta_j-\theta_k-12\lambda\delta_k\in \CH(\partial\bar\M_{1,3}).\]
\end{prop}
\begin{proof}

Consider the commutative diagram
\begin{center}
    \begin{tikzcd}
        \overline\CH(\M_{1,3},1)\arrow[r,"\partial^3_1"] & \CH(\partial\bar\M_{1,3})\\
        \overline\CH(\M_{1,2},1)\arrow[r,"\partial^2_1"]\arrow[u,"\pi^*"] & \CH(\partial\bar\M_{1,2})\arrow[u,"\pi^*"] 
    \end{tikzcd}.
\end{center}

Using Proposition~\ref{higherChowImM12}, a relation from Proposition~\ref{boundarybarM13}, and Proposition~\ref{ReducedFibers} we have 
\begin{align*}
    \partial^3_1(\mathfrak p_{12})&=\partial^3_1(\pi^*(\mathfrak p))\\
    &=\pi^*(\partial^2_1(\mathfrak p))\\
    &=\pi^*(\theta-12\lambda\delta_\emptyset-\lambda\phi)\\
    &=\theta_1+\theta_2-12\lambda\delta_3-12\lambda\delta_\emptyset-\lambda\phi \\
&= \theta_1+\theta_2-12\lambda\delta_3-12\lambda\delta_\emptyset-(2\theta_1+12\lambda(\delta_1-\delta_\emptyset-\delta_2-\delta_3))\\
&=\theta_2+12\lambda\delta_2-\theta_1-12\lambda\delta_1.
\end{align*}
The result follows for all $j,k$ using the $S_3$ action.
\end{proof}

\begin{rmk}
    From Proposition~\ref{higherChowImM12} and its proof, one can see what the image of 
    \[\partial_1:\overline\CH(\M_{1,2},1)\to \CH(\partial\bar\M_{1,2})\] 
    is doing: there are two extra $B\mu_2$-points in the closure of $Y_1(2)$ inside $\bar\M_{1,2}$, and the image of $\partial_1$ identifies the pushforwards from the Chow groups of these points in $\CH(\bar\M_{1,2})$. Something analogous is happening with $\bar\M_{1,3}$, which is not demonstrated by the proof of Proposition~\ref{pullbackhigherChowImM13}, so we describe it here.
    
    Inside of $\M_{1,3}$, we have the locus of curves with extra automorphisms. This consists of curves $(C,p_1,p_2,p_3)$ with $p_2,p_3$ $2$-torsion with respect to $p_1$. This locus is $1$-dimensional, with generic automorphism group $\mu_2$. Inside of $\bar\M_{1,3}$, there are three points that are in the closure of this locus. They are the point in $\M^{\Theta_j}\cong [\M_{0,4}/\mu_2]$ with  $\mu_2$-stabilizer group. These points are isolated in the substack of points with nontrivial automorphism group. In degrees $i\geq 3$, the image of $\partial_1$ identifies the pushforwards of the Chow groups from these three points inside of $\CH(\bar\M_{1,3})$.  
    
    There is an analogous description for $\bar\M_{1,4}$ as well, which is also not in the proof of Proposition~\ref{higherChowImM14}. For $n\geq 5$, $\M_{1,n}$ has only trivial stabilizer groups, so nothing analogous happens. 
\end{rmk}

\begin{prop}\label{higherChowImM13}
    \mbox{}
    \begin{enumerate}
        \item $\partial_1(\mathfrak f)=0$ in $\CH(\partial\bar\M_{1,3}),$
        \item $6\lambda^2=\theta_1+6\lambda(\delta_1-\delta_\emptyset-\delta_2-\delta_3)$ in $\CH^2(\bar\M_{1,3}),$ and 
        \item $\mathfrak f=0$ in $\overline\CH(\M_{1,3},1).$
    \end{enumerate}
\end{prop}
\begin{rmk}
     Our computation of the analogous class $\partial_1(\mathfrak g)\in \CH(\partial\bar\M_{1,4})$ takes up the entirety of section 9. A similar computation of $\partial_1(\mathfrak f)$ would go through, but we are able to use some tricks to avoid such a lengthy computation. 
\end{rmk}
\begin{proof}
    First, we claim that $\CH(\bar\M_{1,3})\otimes \mathbb Q$ satisfies Poincare duality. To see this, fix $\ell\neq \charp(k)$. By Corollary~\ref{barM1nMKP}, $\bar\M_{1,3}$ has the MKP, and hence the $\mathbb Q$-coefficient CKgP by Proposition~\ref{MKPprops}(8). For smooth proper Deligne-Mumford stacks $X$ with the $\mathbb Q$-coefficient CKgP, the cycle class map
    \[\mathrm{cl}: \CH(X)\otimes \mathbb{Q}_\ell \to H_{\text{\'et}}(X,\mathbb Q_{\ell})\]
    is an isomorphism. This follows from the proof of \cite[Lemma 3.11]{CL2} by replacing singular cohomology with \'etale cohomology. Alternatively, this follows from \cite[Theorem 4.1]{Totaro} because $\mathrm{cl}$ is an isomorphism for $X=\Spec(k)$. Thus, $\CH(\bar\M_{1,3})\otimes \mathbb Q$ satisfies Poincare duality. We first show that 
     $\partial_1(\mathfrak f)$ being a non-torsion element violates this Poincare duality. 

    In degree $1$, we have localization exact sequence
    \[0\to \CH^0(\partial\bar\M_{1,3})\to \CH^1(\bar\M_{1,3})\to \CH^1(\M_{1,3})\to 0,\]
    because $\overline\CH^1(\M_{1,3},1)=0$ by Theorem~\ref{M13}. By Proposition~\ref{boundarybarM13} and Theorem~\ref{M13}, we can conclude that $\dim_{\mathbb{Q}}(\CH^1(\bar\M_{1,3})\otimes_{\Z} \mathbb{Q})=5$. 
     
    From Proposition~\ref{pullbackhigherChowImM13}, we know 
    \[\theta_1+12\lambda\delta_1-\theta_2-12\lambda\delta_2,\theta_1+12\lambda\delta_1-\theta_3-12\lambda\delta_3\]
     are in $\im(\partial_1)$. A computation with Smith normal form on the relations from Proposition~\ref{boundarybarM13} gives
    \[\frac{\CH^1(\partial\bar\M_{1,3})}{\langle\{ \theta_1+12\lambda\delta_1-\theta_\ell-12\lambda\delta_\ell\}_{\ell=2,3}\rangle} \cong \Z^2\oplus \frac{\Z}{4\Z}.\]
    In degree $i=2$, the localization exact sequence is  
     \[0\to \frac{\CH^{1}(\partial\bar\M_{1,3})}{\im(\partial)} \to \CH^2(\bar\M_{1,3})\to \CH^2(\M_{1,3})\to 0.\]
    By Theorem~\ref{M13}, $\CH^2(\M_{1,3})$ is generated by $\lambda^2$, which is torsion of order $6$. If $\partial_1(\mathfrak f)$ were non-torsion, we would then have $\dim_{\mathbb{Q}}(\CH^2(\M_{1,3})\otimes_{\Z}\mathbb{Q})=4$, violating Poincare Duality. So we can conclude that $\partial_1(\mathfrak f)$ is torsion.
    
    The element $6\lambda^2\in \CH^2(\bar\M_{1,3})$ must push forward from the boundary. Moreover, by Proposition~\ref{Order 12}, $\lambda^2$ has order $24$ in $\CH^2(\bar\M_{1,3})$, so whatever pushes forward to $\lambda^2$ must have order $4$ in $\CH^1(\partial\bar\M_{1,3})/\im(\partial_1)$. If $\partial_1(\mathfrak f)$ were not zero, there would be no order $4$ element in $\CH^1(\partial\bar\M_{1,3})/\im(\partial_1)$. This proves (1). And now, by Proposition~\ref{pullbackhigherChowImM13}, we have
    \[\im(\partial_1)=\langle\{ \theta_1+12\lambda\delta_1-\theta_\ell-12\lambda\delta_\ell\}_{\ell=2,3}\rangle\]
    in degree $1$.

    The $4$-torsion elements in $\CH^1(\partial\bar\M_{1,3})/\im(\partial_1)$ are given by 
    \[\pm(\theta_1+6\lambda(\delta_1-\delta_\emptyset-\delta_2-\delta_3)),\]
    so one of them must pushforward to $6\lambda^2$. We pull back along \[\bar\M_{1,2}=\bar\M_{\Delta_2}\xrightarrow{\xi_{\Delta_2}} \bar\M_{1,3}.\] We use Theorem~\ref{tautologicalcalculus} and Proposition~\ref{psi} to compute
    \begin{align*}
        & \xi_{\Delta_2}^*(\theta_1)=0\\
        & \xi_{\Delta_2}^*(\delta_1)=0\\
        & \xi_{\Delta_2}^*(\delta_\emptyset)=\delta_\emptyset\\
        & \xi_{\Delta_2}^*(\delta_2)=-\delta_\emptyset-\lambda\\
        & \xi_{\Delta_2}^*(\delta_3)=0.
    \end{align*}
    Thus, the pullback of 
    \[\pm(\theta_1+6\lambda(\delta_1-\delta_\emptyset-\delta_2-\delta_3))\]
    is $\pm 6\lambda^2$. This has to be equal to $\xi_{\Delta_2}^*(6\lambda^2)=6\lambda^2$, so we must have 
    \[6\lambda^2=\theta_1+6\lambda(\delta_1-\delta_\emptyset-\delta_2-\delta_3).\]

    Finally, because $\partial_1(\mathfrak f)=0$, we have that $\mathfrak f\in \CH^2(\M_{1,3},1)$ must restrict from $\CH^2(\bar\M_{1,3},1)$, by the localization exact sequence. But this restriction factors through the restriction $\CH^2(\bar\M_{1,3},1)\to \CH^2(\M_{1,3}^\irr,1)$, which is $0$ by Proposition~\ref{boundarybarM13}. Thus, $\mathfrak f=0$. 
\end{proof}

This gives the following corollary, using Theorem~\ref{M13}.

\begin{cor}\label{higherM13}
    The indecomposable first higher Chow group of $\M_{1,3}$ is given by
    \[\overline\CH(\M_{1,3},1)=\frac{\Lambda \langle  \mathfrak p_{12},\mathfrak p_{13}\rangle}{\langle 2\mathfrak p_{12},2\mathfrak p_{13}\rangle}. \]
\end{cor}

\begin{thm}\label{barM13}
    The Chow ring of $\bar\M_{1,3}$ is given by \[\CH(\bar\M_{1,3})=\Z[\delta_1,\delta_2,\delta_3,\delta_\emptyset,\lambda]/I,\] where $I$ is the ideal generated by the elements
    \begin{align*}
        & 24\lambda^2 \\
        & \delta_{j}\delta_k\\
    & \delta_{j}\delta_{\emptyset}-\delta_k\delta_{\emptyset}\\
    & \delta_{j}^2+\lambda\delta_j+\delta_j\delta_\emptyset\\
    & \delta_\emptyset^2+\lambda\delta_\emptyset+\delta_j\delta_\emptyset\\
    & 12\lambda^3-12\lambda^2\delta_1+12\lambda^2\delta_2+12\lambda^2\delta_3+12\lambda^2\delta_\emptyset
    \end{align*}
    for $\{j,k,\ell\}=[3]$. As an $A$-module, this is
    \[\CH(\bar\M_{1,3})=\frac{A\langle 1,\delta_\emptyset,\delta_1,\delta_2,\delta_3,\delta_{\emptyset|*}\rangle}{\langle 12\lambda^2(\lambda\cdot 1+\delta_\emptyset+\delta_1+\delta_2+\delta_3)\rangle}.\]
    Finally, we have
     \[\theta_j=6\lambda^{2}-6\lambda\delta_{j}+6\lambda\delta_{k}+6\lambda\delta_{\ell}+6\lambda\delta_\emptyset\]
     for $\{j,k,\ell\}=[3]$ and
     \[\omega=24\lambda\delta_\emptyset\delta_1.\]
\end{thm}
\begin{proof}
We have an exact sequence
\[0\to \frac{\CH(\partial\bar\M_{1,3})}{\im(\partial_1)}\to \CH(\bar\M_{1,3})\to \CH(\M_{1,3})\to 0.\]
By Theorem~\ref{M13}, Proposition~\ref{boundarybarM13}, Proposition~\ref{pullbackhigherChowImM13}, and Proposition~\ref{higherChowImM13}, we can write this exact sequence as
\[0\to \frac{A\langle \{\delta_j,\theta_j\}_j,\delta_\emptyset,\delta_{\emptyset|*}\rangle}{R'}\to \CH(\bar\M_{1,3})\to \frac{A\langle 1\rangle}{\langle 12\lambda\cdot 1, 6\lambda^2\cdot 1\rangle}\to 0,\]
where $R'$ is generated by 
\begin{align*}
   & 2\lambda\theta_j \\
   & 2\lambda\phi\\
   & \lambda\phi-2\theta_1-12\lambda(\delta_1-\delta_\emptyset-\delta_2-\delta_3)\\
   & 2\theta_j+2\theta_k-24\lambda\delta_\ell-2\lambda\delta_\emptyset \\
   & \theta_j+12\lambda\delta_j-\theta_k-12\lambda\delta_k
\end{align*}
for $\{j,k,\ell\}=[3]$.
We lift the relations of $\CH(\M_{1,3})$ to $\CH(\bar\M_{1,3})$ to use Lemma~\ref{exactPresentation}. By Proposition~\ref{Order 12}, we have $\phi=12\lambda$ in $\CH(\bar\M_{1,3})$ and by Proposition~\ref{higherChowImM13}, we can write
\[6\lambda^2=\theta_1+6\lambda(\delta_1-\delta_\emptyset-\delta_2-\delta_3)\]
in $\CH(\bar\M_{1,3})$. From this latter relation, we can solve for $\theta_1$ in terms of other generators. By $S_3$ symmetry, we can also solve for $\theta_2,\theta_3$. Then Lemma~\ref{exactPresentation} gives us
\[\CH(\bar\M_{1,3})=\frac{A\langle 1,\delta_\emptyset,\delta_1,\delta_2,\delta_3,\delta_{\emptyset|*}\rangle}{\langle 12\lambda^2(\lambda\cdot 1+\delta_\emptyset+\delta_1+\delta_2+\delta_3)\rangle}.\]
This describes $\CH(\bar\M_{1,3})$ as an $A$-module, but we still desire to know its ring structure. 

Using Theorem~\ref{tautologicalcalculus} and Proposition~\ref{psi} to compute products between these generators, we obtain
\begin{align*}
    & \delta_\emptyset\delta_j=\delta_{\emptyset|*}\\
    & \delta_\emptyset^2=-\lambda\delta_{\emptyset}-\delta_{\emptyset|*}\\
    & \delta_j^2=-\delta_{\emptyset|*}-\lambda\delta_j\\
    & \delta_j\delta_k=0
\end{align*}
for $\{j,k,\ell\}=[3]$.
From these, we can deduce
\begin{align*}
    & \delta_{\emptyset|*}\delta_j=\delta_\emptyset\delta_k\delta_j=0 \\
    & \delta_{\emptyset|*}\delta_\emptyset=\delta_\emptyset^2\delta_j=-\lambda\delta_\emptyset\delta_j-\delta_{\emptyset|*}\delta_j=-\lambda\delta_\emptyset\delta_j \\
    & \delta_{\emptyset|*}^2=\delta_{\emptyset}^2\delta_j\delta_k=0
\end{align*}
for $\{j,k,\ell\}=[3]$.
These expressions, together with $12\lambda^2(\lambda\cdot 1+\delta_\emptyset+\delta_1+\delta_2+\delta_3)$,  give the kernel of the ring map
\[A[\delta_\emptyset,\delta_1,\delta_2,\delta_3,\delta_{\emptyset|*}]\to \CH(\bar\M_{1,3}).\]
Moreover, we see we do not need $\delta_{\emptyset|*}$ as a generator, so we can write
\[\CH(\bar\M_{1,3})=\frac{A[\delta_\emptyset,\delta_1,\delta_2,\delta_3]}{I'}=\frac{\Z[\lambda,\delta_\emptyset,\delta_1,\delta_2,\delta_3]}{(24\lambda^2,I')}\]
where $I'$ is generated by 
\begin{align*}
    & \delta_{j}\delta_k\\
    & \delta_{j}\delta_{\emptyset}-\delta_k\delta_{\emptyset}\\
    & \delta_{j}^2+\lambda\delta_j+\delta_j\delta_\emptyset\\
    & \delta_\emptyset^2+\lambda\delta_\emptyset+\delta_j\delta_\emptyset\\
    & 12\lambda^3-12\lambda^2\delta_1+12\lambda^2\delta_2+12\lambda^2\delta_3+12\lambda^2\delta_\emptyset
\end{align*}
for $\{j,k,\ell\}=[3]$.

Finally, using Proposition~\ref{boundarybarM13}, we can write
\[\omega=24\lambda\delta_{\emptyset|*}=24\lambda\delta_\emptyset\delta_1\]
inside $\CH(\bar\M_{1,3})$.
\end{proof}

\subsection{$\bar\M_{1,4}$}
For $\{j,k,\ell,m\}=[4]$, define the stable graphs
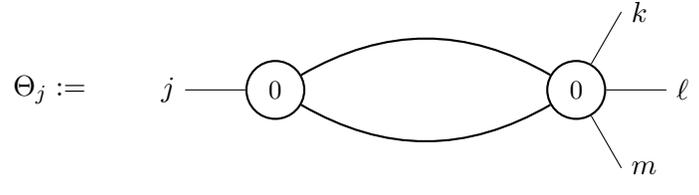
\[\begin{tikzpicture}[
  vertex/.style={
    circle, draw, thick, minimum size=2em, inner sep=0pt, font=\small
  }
]

  \node[vertex] (v0) at (0,0) {$0$};
  \node[vertex] (v1) at (4,0) {$0$};

  \draw[thick, bend left=30] (v0) to (v1);
  \draw[thick, bend right=30] (v0) to (v1);

  \draw (v0.180) -- ++(180:0.8) node[left] {$j$};

  \draw (v1.60) -- ++(60:0.8) node[right] {$k$};
  \draw (v1.0) -- ++(0:0.8) node[right] {$\ell$};
  \draw (v1.-60) -- ++(-60:0.8) node[right] {$m$};

  \node at (-3,0) {$\Theta_j:=$};

\end{tikzpicture},\]

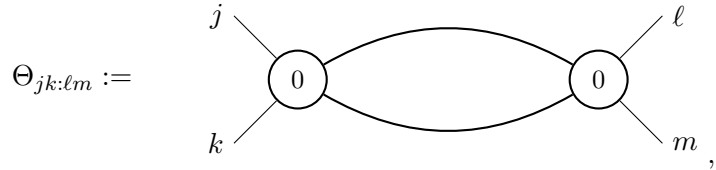
\[\begin{tikzpicture}[
  vertex/.style={
    circle, draw, thick, minimum size=2em, inner sep=0pt, font=\small
  }
]

  \node[vertex] (v0) at (0,0) {$0$};
  \node[vertex] (v1) at (4,0) {$0$};

  \draw[thick, bend left=30] (v0) to (v1);
  \draw[thick, bend right=30] (v0) to (v1);

 \draw (v0.135) -- ++(135:0.8) node[left] {$j$};
  \draw (v0.-135) -- ++(-135:0.8) node[left] {$k$};

  \draw (v1.45) -- ++(45:0.8) node[right] {$\ell$};
  \draw (v1.-45) -- ++(-45:0.8) node[right] {$m$};

  \node at (-3,0) {$\Theta_{jk:\ell m}:=$};

\end{tikzpicture},\] 
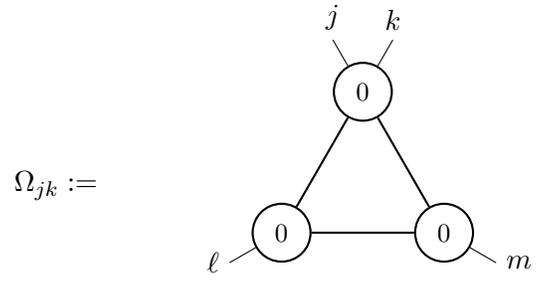
\[\begin{tikzpicture}[
  vertex/.style={
    circle, draw, thick, minimum size=2em, inner sep=0pt, font=\small
  }
]

  \node[vertex] (v1) at (-1.92,1.25) {$0$};         
  \node[vertex] (v2) at (-3,-0.625) {$0$};   
  \node[vertex] (v3) at (-.84,-0.625) {$0$};    

  \draw[thick] (v1) -- (v2);
  \draw[thick] (v2) -- (v3);
  \draw[thick] (v3) -- (v1);

  \draw (v1.120) --++(120:0.4) node[above] {$j$};
  \draw (v1.60) --++(60:0.4) node[above] {$k$};
  \draw (v2.210) --++(210:0.4) node[left] {$\ell$};
  \draw (v3.330) --++(330:0.4) node[right] {$m$};

    \node at (-6,0) {$\Omega_{jk}:=$};
\end{tikzpicture},\]
and
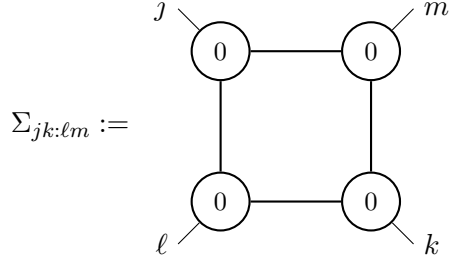
\[\begin{tikzpicture}[
  vertex/.style={
    circle, draw, thick, minimum size=2em, inner sep=0pt, font=\small
  }
]

  \node[vertex] (v1) at (-1,1) {$0$};         
  \node[vertex] (v2) at (-1,-1) {$0$};   
  \node[vertex] (v3) at (1,-1) {$0$};
  \node[vertex] (v4) at (1,1) {$0$};

  \draw[thick] (v1) -- (v2);
  \draw[thick] (v2) -- (v3);
  \draw[thick] (v3) -- (v4);
\draw[thick] (v4) -- (v1);

  \draw (v1.135) --++(135:0.4) node[left] {$j$};
  \draw (v2.-135) --++(-135:0.4) node[left] {$\ell$};
  \draw (v3.-45) --++(-45:0.4) node[right] {$k$};
  \draw (v4.45) --++(45:0.4) node[right] {$m$};

    \node at (-3,0) {$\Sigma_{jk:\ell m}:=$};
\end{tikzpicture}.\]
Note that some of these graphs can be defined by differing orders of the indices, e.g. $\Theta_{12:34}=\Theta_{43:12}$. We treat the corresponding Chow classes as identical, meaning we do not need to quotient out by relations like $\theta_{12:34}-\theta_{43:12}$.

\begin{prop}\label{boundarybarM14}
    The Chow group $\CH(\partial\bar\M_{1,4})$ is generated as an $A$-module by
    \[\phi,\delta_\emptyset,\{\delta_j,\delta_{\emptyset|j},\delta_{j|*},\theta_j\}_j, \{\delta_{jk},\delta_{\emptyset|jk},\theta_{jk}\}_{jk},\delta_{12|34},\delta_{13|24},\delta_{14|23},\delta_{\emptyset|*|*}\]
    subject to only the relations
\begin{align*}
    &2\lambda\phi\\
    & 2\lambda\theta_j\\
    &2\lambda\theta_{jk}\\
    &\lambda \phi-\bar\alpha_{jkl}  \\
    &12\lambda^2(\lambda\delta_{jk}+\delta_{\emptyset|jk}+\delta_{j|*}+\delta_{k|*}+\delta_{jk|\ell m})\\
    &\delta_{\emptyset|jk}+\delta_{\emptyset|\ell m}-\delta_{\emptyset|j\ell}-\delta_{\emptyset|km}\\
    &\delta_{\emptyset|j}+\delta_{\emptyset|jk}-\delta_{\emptyset|\ell}-\delta_{\emptyset|\ell k}\\
&\lambda \theta_j-(6\lambda\hat{\delta}+6\lambda^2(\delta_{jm}-2\delta_j-\delta_{jk}-\delta_{j\ell}+2\delta_{k\ell})) \\
&2(\theta_k+\theta_{j\ell}+\theta_{jm}+\theta_j)-24\lambda(\delta_\emptyset+\delta_m+\delta_\ell+\delta_{\ell m}) 
\end{align*}    
for $\{j,k,\ell,m\}=[4],$ where 
\begin{align*}
\bar\alpha_{jk\ell}& :=2(\theta_\ell+\theta_{jk})+12\lambda(\delta_{\ell m}+\delta_\ell-\delta_m-\delta_\emptyset-\delta_k-\delta_{km}-\delta_j-\delta_{jm}) \\
\hat{\delta}& :=\delta_{\emptyset|14}-\delta_{\emptyset|2}-\delta_{\emptyset|3}+\delta_{1|*}+\delta_{2|*}+\delta_{3|*}+\delta_{4|*}-\delta_{12|34}-\delta_{13|24}-\delta_{14|23}
\end{align*}
for $\{j,k,\ell,m\}=[4].$
Moreover, in $\CH(\partial\bar\M_{1,4})$, we have
\begin{align*}
    \omega_{jk} &=12\lambda(\lambda\delta_{j k}+\delta_{j|*}+\delta_{k|*}-\delta_{jk|\ell m}+\delta_{\emptyset|j k})\\
    \sigma_{jk:\ell m} &=24\lambda\delta_{\emptyset|*|*}
\end{align*}
for $\{j,k,\ell,m\}=[4].$
\end{prop}
\begin{proof}
The components of $\bar\M^{\sep,\geq 5}=\bar\M^\sep$ are given by 
$\bar\M^{\Delta_j}$ for $j\in [4],$
$\bar\M^{\Delta_{jk}}$ for $j\neq k\in [4],$ and
$\bar\M^{\Delta_\emptyset}$. We use the exact sequence
\[\bigoplus_{\Gamma\neq\Gamma'} \CH(\bar\M^{\Gamma}\cap \bar\M^{\Gamma'}) \to \bigoplus_{\Gamma} \CH(\bar\M^{\Gamma})\to \CH(\bar\M^\sep)\to 0\]
to calculate $\CH(\bar\M^\sep)$.

By Proposition~\ref{SeparatingProduct}, we have that
\[\CH(\bar\M^{\Delta_\emptyset})\cong \CH(\bar\M_{\Delta_\emptyset})=\CH(\bar\M_{1,1}\times \bar\M_{0,5}).\]
By Proposition~\ref{MKPM0n}, $\bar\M_{0,5}$ has the MKP, so by Proposition~\ref{MKPprops}(6), we have
\[\CH(\bar\M_{1,1}\times\bar\M_{0,5})=\CH(\bar\M_{1,1})\otimes\CH(\bar\M_{0,5})=A\otimes \CH(\bar\M_{0,5}).\]
By Theorem~\ref{barM0n}, we can write
\[\CH(\bar\M_{0,5})=\frac{\Z\langle 1,\{D_{jk}\}_{jk}, D_{12}D_{13} \rangle}{\langle\{D_{jk}+D_{\ell m}-D_{j\ell}-D_{km}\}_{jk\ell m}\rangle},\]
so 
\[\CH(\bar\M_{1,1})\otimes\CH(\bar\M_{0,5})=\frac{A\langle 1,\{D_{jk}\}_{jk}, D_{12}D_{13} \rangle}{\langle\{D_{jk}+D_{\ell m}-D_{j\ell}-D_{km}\}_{jk\ell m}\rangle}.\]
Unwinding the isomorphisms, we get
{\small \[\CH(\bar\M^{\Delta_\emptyset})=\frac{A\langle\delta_\emptyset, \{\delta_{\emptyset|j}\}_j, \{\delta_{\emptyset|jk}\}_{jk},\delta_{\emptyset|*|*} \rangle}{\langle\{\delta_{\emptyset|jk}+\delta_{\emptyset|\ell m}-\delta_{\emptyset|j\ell}-\delta_{\emptyset|km}\}_{jk\ell m},\{\delta_{\emptyset|j}+\delta_{\emptyset|jk}-\delta_{\emptyset|\ell}-\delta_{\emptyset|\ell k}\}_{jk\ell}\rangle}.\]}
Here we write $\delta_{\emptyset|*|*}$ to indicate any one of $\delta_{\emptyset|j|jk}$, as they are all equal in $\CH(\bar\M^{\Delta_\emptyset})$. Similarly, Proposition~\ref{SeparatingProduct}, Theorem~\ref{barM12}, and Theorem~\ref{barM0n} with $n=4$ give 
\[\CH(\bar\M^{\Delta_j})=A\langle\delta_j, \delta_{\emptyset|j},\delta_{j|*},\delta_{\emptyset|j|*} \rangle,\]
and Proposition~\ref{SeparatingProduct} and Theorem~\ref{barM13} give
\[\CH(\bar\M^{\Delta_{jk}})=\frac{A\langle \delta_{jk},\delta_{\emptyset|jk},\delta_{j|jk},\delta_{k|jk},\delta_{jk|\ell m},\delta_{\emptyset|j|jk}\rangle}{\langle12 \lambda^2(\lambda \delta_{jk}+\delta_{\emptyset|jk}+\delta_{j|jk}+\delta_{k|jk}+\delta_{jk|\ell m})\rangle}.\]
  
Using Theorem~\ref{tautologicalcalculus}, we can compute the intersections between these components to be 
\begin{align*}
    & \bar\M^{\Delta_j}\cap \bar\M^{\Delta_k}=\emptyset \text{ for }j\neq k\\
    & \bar\M^{\Delta_j}\cap\bar\M^{\Delta_{k\ell}}=\emptyset\text{ for } j\neq k,\ell\\
    & \bar\M^{\Delta_{jk}}\cap \bar\M^{\Delta_{j\ell}}=\emptyset\text{ for } k\neq \ell\\
    &\bar\M^{\Delta_j}\cap \bar\M^{\Delta_{jk}}=\bar\M^{\Delta_{j|k}}\\
& \bar\M^{\Delta_{jk}}\cap \bar\M^{\Delta_{\ell m}}=\bar\M^{{}\Delta_{jk|\ell m}}\text{ for }\{j,k,\ell,m\}=[4] \\
& \bar\M^{\Delta_j}\cap\bar\M^{\Delta_\emptyset}=\bar\M^{\Delta_{\emptyset|j}}\\
& \bar\M^{\Delta_{jk}}\cap\bar\M^{\Delta_\emptyset}=\bar\M^{\Delta_{\emptyset|jk}}.
\end{align*}

Similar arguments to the above using Proposition~\ref{SeparatingProduct}, the fact that $\bar\M_{0,n}$ has the MKP, and previous computations of $\CH(\bar\M_{0,n}),\CH(\bar\M_{1,n})$ give
\begin{align*}
    &\CH(\bar\M^{\Delta_{\emptyset|j}})=A\langle \delta_{\emptyset|j},\delta_{\emptyset|j|*} \rangle\\
    &\CH(\bar\M^{\Delta_{j|jk}})=A\langle \delta_{j|jk},\delta_{\emptyset|j|jk}\rangle\\
    &\CH(\bar\M^{\Delta_{jk|\ell m}})=A\langle\delta_{jk|\ell m}, \delta_{\emptyset|jk|\ell m} \rangle\\
    &\CH(\bar\M^{\Delta_{\emptyset|jk}})=A\langle \delta_{\emptyset|jk},\delta_{\emptyset|j|jk}\rangle.
\end{align*}
By the exact sequence
\[\bigoplus \CH(\bar\M^\Delta\cap \bar\M^{\Delta'})\to \bigoplus \CH(\bar\M^{\Delta})\to \CH(\bar\M^{\sep})\to 0,\]
we have 
{\small \[\CH(\bar\M^{\sep})={A\langle \delta_\emptyset,\{\delta_j,\delta_{\emptyset|j},\delta_{j|*}\}_j, \{\delta_{jk},\delta_{\emptyset|jk}\}_{jk}\},\delta_{12|34},\delta_{13|24},\delta_{14|23},\delta_{\emptyset|*|*} \rangle}/{R_1},\]}
where $R_1$ is the module of relations generated by
\[12\lambda^2(\lambda\delta_{jk}+\delta_{\emptyset|jk}+\delta_{j|jk}+\delta_{k|jk}+\delta_{jk|\ell m}) \]
\[\delta_{\emptyset|jk}+\delta_{\emptyset|\ell m}-\delta_{\emptyset|j\ell}-\delta_{\emptyset|km}\]
\[\delta_{\emptyset|j}+\delta_{\emptyset|jk}-\delta_{\emptyset|\ell}-\delta_{\emptyset|\ell k} \]
for $\{j,k,\ell,m\}=[4]$.
Because 
\[\bar\M^{\sep,\geq 4}=\bar\M^{\sep}\amalg \bar\M^{\Sigma_{12:34}}\amalg \bar\M^{\Sigma_{13:24}}\amalg\bar\M^{\Sigma_{14:23}}\] 
and $\bar\M^\Sigma=\Spec(k)$ by Proposition~\ref{OpenStrata}, we have 
\[\CH(\bar\M^{\sep,\geq 4})=\CH(\bar\M^{\sep})\oplus \Z=\CH(\bar\M^{\sep})\oplus \frac{A\langle \sigma_{12:34},\sigma_{13:24},\sigma_{14:23} \rangle}{\langle \lambda \sigma_{12:34},\lambda\sigma_{13:24},\lambda\sigma_{14:23}\rangle}.\]

Next, we consider $\bar\M^{\sep,\geq 3}=\bar\M^{\sep,\geq 4}\cup \M^{\non=3}$; we have 
\[\M^{\non=3}=\coprod_{ j<k}\M^{\Omega_{jk}}.\]
Note 
\[\M^{\Omega_{jk}}=[\M_{\Omega_{jk}}/\Aut(\Omega_{jk})]=\M_{0,3}\times\M_{0,4}\times\M_{0,3}=\M_{0,4}\]
by Proposition~\ref{OpenStrata}. By Theorem~\ref{WDVV mod 2}, we know that the image of 
\[\partial_1^{\Omega_{jk}}:\overline\CH(\M^{\Omega_{jk}},1)\to \CH([\partial\bar\M_{\Omega_{jk}}/\Aut(\Omega_{jk})])\] 
is generated by $D(jk|ab)-D(ja|kb),D(ka|jb|)-D(ja|kb)$. We can compute the pushforward along $\xi_{\Omega_{jk}}$ to be 
    \[D(jk|ab)\mapsto\xi_{\Delta_{\ell m}*}(\omega)\]
    \[D(ja|kb)\mapsto \sigma_{jm:k\ell}\]
    \[D(ka|jb)\mapsto \sigma_{j\ell:km}\]
where $\ell,m$ make $\{j,k,\ell,m\}=[4]$, using Lemma~\ref{automorphism}. By Theorem~\ref{barM13}, we have 
\[\xi_{\Delta_{\ell m}*}(\omega)=24\lambda \delta_{\emptyset|*|*}.\] 
Thus, $\xi_{\Omega_{jk}*}(\im(\partial_1^{\Omega_{jk}}))$ is generated by 
\[\sigma_{jm:k\ell}-24\lambda\delta_{\emptyset|*|*}\]
\[\sigma_{jm:k\ell}-\sigma_{j\ell:k m}\]
for $\{j,k,\ell,m\}=[4]$.
Moreover, by the commutative diagram
\begin{center}
    \begin{tikzcd}
        \bar\M_{\Omega_{jk}}\arrow[r,"\xi_{\Omega_{jk}}"]\arrow[d] & \bar\M_{1,4}\arrow[d]\\
        \partial\bar\M_{1,1}\arrow[r] & \bar\M_{1,1}
    \end{tikzcd}
\end{center}
we see that $\lambda$ acts by $0$ on $\CH(\bar\M_{\Omega_{jk}})$. Thus, the section 
\[\Z=\CH(\M^{\Omega_{jk}})\to \CH(\bar\M^{\Omega_{jk}})\]
is $A$-linear. So, by Proposition~\ref{filtration}, we have
\begin{align*}
    \CH(\bar\M^{\sep,\geq 3})&=\frac{\CH(\bar\M^{\sep,\geq 4})\oplus A\langle \{\omega_{jk}\}_{jk}\rangle}{\langle\{\sigma_{jk:\ell m}-24\lambda\delta_{\emptyset|*|*}\}_{jk:\ell m}, \{\lambda \omega_{jk}\}_{jk}\rangle}\\
    &=\frac{\CH(\bar\M^{\sep})\oplus A\langle \{\omega_{jk}\}_{jk}\rangle}{\langle\{\lambda \omega_{jk}\}_{jk}\rangle}.
\end{align*} 
As we have eliminated the $\sigma_{jk:\ell m}$ as generators, we remember we have the expression $\sigma_{jk:\ell m}=24\lambda\delta_{\emptyset|*|*}$ inside $\CH(\bar\M^{\sep,\geq 3})$.

Now, consider $\bar\M^{\sep,\geq 2}=\bar\M^{\sep,\geq 3}\cup \M^{\non=2}$. We have
\[\M^{\non=2}=\coprod_{1=j<k} \M^{\Theta_{jk:\ell m}} \amalg \coprod_{j\in [4]} \M^{\Theta_j}.\]
Note 
\[\M^{\Theta_{jk:\ell m}}=[\M_{\Theta_{jk:\ell m}}/\mu_2]=[\M_{0,4}\times\M_{0,4}/\mu_2]\]
and
\[\M^{\Theta_j}=[\M_{\Theta_j}/\Aut(\Theta_j)]=[\M_{0,5}/\mu_2]\]
by Proposition~\ref{OpenStrata}.

By Lemma~\ref{M04M04/2}(1), we know that the image of
\[\partial_1^{\Theta_{jk}}:\overline\CH(\M^{\Theta_{jk}},1)\to \CH([\partial\bar\M_{\Theta_{jk:\ell m}}/\Aut(\Theta_{jk:\ell m})])\]
is generated by
\[\pi_1^{-1}\widehat{D}(ja|kb)-2\pi_1^{-1}\widehat{D}(jk|ab)\]
\[\pi_2^{-1}\widehat{D}(\ell a|mb)-2\pi_2^{-1}\widehat{D}(\ell m|ab).\]
We can compute the pushforwards along $\xi_{\Theta_j}$ to be
\[\pi_1^{-1}\widehat{D}(ja|kb)\mapsto\omega_{\ell m}\]
\[\pi_1^{-1}\widehat{D}(jk|ab)\mapsto \xi_{\Delta_{\ell m}*}(\theta_{p})\]
using Lemma~\ref{automorphism}, where $p$ is the marking that the component containing $jk$ gets attached to. By Theorem~\ref{barM13}, we know 
\[\xi_{\Delta_{\ell m}*}(\theta_{p})=6\lambda(\lambda\delta_{\ell m}+\delta_{\ell|m}+\delta_{m|\ell}-\psi_{jk:\ell m}+\delta_{\emptyset| \ell m}).\]
We have symmetric statements for $\pi_2^{-1}\widehat{D}(\ell a|mb),\pi_2^{-1}\widehat{D}(\ell m|ab).$ Thus, we have that $\xi_{\Theta_{jk:\ell m}*}(\im(\partial_1^{\Theta_{jk}}))$ is generated by 
\[\omega_{\ell m}-12\lambda(\lambda\delta_{\ell m}+\delta_{\ell|m}+\delta_{m|\ell}-\delta_{jk|\ell m}+\delta_{\emptyset| \ell m})\]
and
\[\omega_{jk}-12\lambda(\lambda\delta_{jk}+\delta_{j|k}+\delta_{k|j}-\delta_{jk|\ell m}+\delta_{\emptyset| jk}).\]

Next, by Theorem~\ref{WDVV mod 2}, we know that the image of
\[\partial_1^{\Theta_j}:\overline\CH(\M^{\Theta_j},1)\to \CH([\partial\bar\M_{\Theta_j}/\Aut(\Theta_j)])\]
is generated by
\[\widehat{D}(ka|\ell b)-2\widehat{D}(k\ell|ab)\]
    \[\widehat{D}(ka|mb)-2\widehat{D}(km|ab)\]
    \[2(\widehat{D}(\ell mb|ka)+\widehat{D}(kab|\ell m)-\widehat{D}(k\ell m|ab)-\widehat{D}(mab|k \ell)-\widehat{D}(\ell ab|km)).\]
Here, we label the markings on $\bar\M_{\Theta_j}\cong \bar\M_{0,5}$ getting glued by $a$ and $b$, and otherwise by the marking they are getting sent to under the gluing map. We can compute the pushforwards along $\xi_{\Theta_j}$ to be 
\[\widehat{D}(ka|\ell mb)\mapsto \omega_{\ell m}\]
\[\widehat{D}(ab|k\ell m)\mapsto \xi_{\Delta_j*}(\theta)\]
\[\widehat{D}(kab|\ell m)\mapsto \xi_{\Delta_{jk}*}(\theta_p)\]
using Lemma~\ref{automorphism}, where $p$ is the marking mapping to $j$. By Theorem~\ref{barM12} and Theorem~\ref{barM13}, we know
\[\xi_{\Delta_j*}(\theta)=12\lambda(\delta_{\emptyset|j}+\lambda\delta_j)\]
\[\xi_{\Delta_{jk}*}(\theta_p)=6\lambda(\lambda\delta_{jk}-\delta_{j|k}+\delta_{k|j}+\delta_{jk|\ell m}+\delta_{\emptyset|jk}).\]
With these, we can compute the image of $\partial_1^{\Theta_j}$, but the relations are rather complicated. Combining with the image of $\partial_1^{\Theta_{jk}}$, there is much simplification; the image of 
\[\partial_1: \overline\CH(\M^{\non =2},1)\to \CH(\M^{\sep,\geq 3})\]
is generated by
\[\omega_{jk}-12\lambda(\lambda\delta_{jk}+\delta_{j|k}+\delta_{k|j}-\delta_{jk|\ell m}+\delta_{\emptyset| jk})\]
\[2(6\lambda\hat{\delta}+6\lambda^2(\delta_{jm}-2\delta_j-\delta_{jk}-\delta_{j\ell}+2\delta_{k\ell}))\]
for $\{j,k,\ell,m\}=[4]$.

Now, Lemma~\ref{M04M04/2}, we have
\[\CH(\M^{\Theta_{jk:\ell m}})=\CH([\M_{0,4}\times\M_{0,4}/\mu_2])=\frac{\Z[u]}{(2u)}\]
and we have a splitting of \[\CH([\bar\M_{\Theta_{jk:\ell m}}/\Aut(\Theta_{jk:\ell m})])]\to \CH(\M^{\Theta_{jk:\ell m}})\] given by $u^i\mapsto \nu^*(u^i)$, where $\nu: [\bar\M_{0,4}\times\bar\M_{0,4}/\mu_2]\to B\mu_2$. This is $A$-linear, because it is given by a pullback. Additionally, by Theorem~\ref{M0n/2},
\[\CH(\M^{\Theta_j})=\CH([\M_{0,5}/\mu_2])=\Z.\]
By Lemma~\ref{NonSepAStruct} and Corollary~\ref{uniquetorsion}, we have
\[\lambda\theta_j= \xi_{\Theta_j*}(\alpha)=6\lambda\hat{\delta}+6\lambda^2(\delta_{jm}-2\delta_j-\delta_{jk}-\delta_{j\ell}+2\delta_{k\ell}).\]
We can use the exact sequence from Proposition~\ref{filtration} and Lemma~\ref{exactPresentation} to describe $\CH(\bar\M^{\sep,\geq 2})$ as
\[\frac{\CH(\bar\M^{\sep,\geq 3})\oplus A\langle \{\theta_j\}_j,\{\theta_{jk}\}_{jk}\rangle}{R_2}=\frac{\CH(\bar\M^\sep) \oplus A\langle \{\theta_j\}_j,\{\theta_{jk}\}_{jk}\rangle}{R_2},\]
where $R_2$ is the module of relations generated by
\[2\lambda\theta_j\]
\[2\lambda\theta_{jk}\]
\[\lambda \theta_j-(6\lambda\hat{\delta}+6\lambda^2(\delta_{jm}-2\delta_j-\delta_{jk}-\delta_{j\ell}+2\delta_{k\ell}))\]
for $\{j,k,\ell,m\}=[4]$. Because we got rid of $\omega_{jk}$ as a generator, we remember 
\[\omega_{jk}=12\lambda(\lambda\delta_{j k}+\delta_{j|k}+\delta_{k|j}-\delta_{jk|\ell m}+\delta_{\emptyset|j k})\]
for $\{j,k,\ell,m\}=[4]$.

Finally, we consider $\partial\bar\M_{1,4}=\bar\M^{\sep,\geq 1}=\bar\M^{\sep,\geq 2}\cup \M^{\non=1}$. We have $\M^{\non=1}=\M^\Phi$ and note
\[\M^\Phi=[\M_\Phi/\Aut(\Phi)]=[\M_{0,6}/\mu_2]\]
by Proposition~\ref{OpenStrata}.
The localization exact sequence for $\bar\M^{\sep,\geq 2}\subseteq \partial\bar\M_{1,4}$ is
\[\overline\CH([\M_{0,6}/\mu_2],1)\xrightarrow{\partial_1} \CH(\bar\M^{\sep,\geq 2})\to \CH(\partial\bar\M_{1,4})\to \CH([\M_{0,6}/\mu_2])\to 0.\]
By Theorem~\ref{WDVV mod 2}, the image of 
\[\partial_1^\Phi:\overline\CH(\M^\Phi,1)\to \CH([\partial\bar\M_\Phi/\Aut(\Phi)])\] 
is generated by
\[\widehat D(ja|k b)-2\widehat D(jk|ab)\]
    for $j,k\in \{1,\dots,4\}$ and
    \[\alpha_{jk\ell}+\alpha_{j'k'\ell'}\]
    for $j,k,\ell,j',k',\ell'\in \{1,\dots,4\}$ with $|\{j,k,\ell\}|=|\{j',k',\ell'\}|=3$, where 
    \[\alpha_{jk\ell}:=\widehat D(jkb|\ell a)+\widehat D(\ell ab|jk)-\widehat D(\ell jk|ab)-\widehat D(kab|\ell j)-\widehat D(jab|\ell k).\]
We can compute the pushforwards along $\xi_{\Phi}$ to be
\[\widehat{D}(jkab|\ell m)\mapsto 12\lambda\delta_{jk}\]
\[\widehat{D}(jab|k\ell m)\mapsto 12\lambda\delta_j\]
\[\widehat{D}(ab|1234)\mapsto 12\lambda\delta_\emptyset\]
\[\widehat{D}(ja|k\ell mb)\mapsto 2\theta_{j}\]
\[\widehat{D}(jka|k\ell b)\mapsto 2\theta_{jk:\ell m}.\]
Thus, $\xi_{\Theta*}(\im(\partial_1^{\Phi}))$ is generated by
\[2(\theta_k+\theta_{j\ell}+\theta_{jm}+\theta_j)-24\lambda(\delta_\emptyset+\delta_m+\delta_\ell+\delta_{\ell m})\]
for $\{j,k,\ell,m\}=[4]$ and 
\[\bar\alpha_{jk\ell}+\bar\alpha_{j'k'\ell'}\]
for $j,k,\ell,j',k',\ell'\in \{1,\dots,4\}$ with $|\{j,k,\ell\}|=|\{j',k',\ell'\}|=3$, where the $\bar\alpha_{jk\ell}$ are as in the statement of this Proposition.

Now, we have $\CH(\M^\Phi)=\Z=A\langle \phi\rangle/\langle \lambda\phi\rangle$ by Theorem~\ref{M0n/2}. By Lemma~\ref{NonSepAStruct} and Corollary~\ref{uniquetorsion}, we have 
\[\lambda\phi=\xi_{\Phi*}(\alpha)=\bar\alpha_{jk\ell},\]
for $j,k,\ell\in [4]$ distinct, where $\bar\alpha_{jk\ell}$ is as defined in the statement of this proposition.
With this, we can use the exact sequence from Proposition~\ref{filtration} and Lemma~\ref{exactPresentation} to get 
\[\CH(\partial\bar\M_{1,4})=\frac{\CH(\bar\M^{\sep,\geq 2})\oplus A\langle \phi\rangle}{R_3},\]
where $R_3$ is the module of relations generated by
\[2\lambda\phi\]
\[\lambda \phi-\bar\alpha_{jk\ell}\]
\[2(\theta_k+\theta_{j\ell}+\theta_{jm}+\theta_j)-24\lambda(\delta_\emptyset+\delta_m+\delta_\ell+\delta_{\ell m}),\] 
for $\{j,k,\ell,m\}=[4]$. Combining this description of $\CH(\partial\bar\M_{1,4})$ with our above description of $\CH(\bar\M^{\sep,\geq 2})$ and our above description of $\CH(\bar\M^{\sep})$ gives the proposition.
\end{proof}

\begin{prop}\label{pullbackhigherChowImM14}
We have 
{\small \[\partial_1(\mathfrak p_{jk})=\theta_j+\theta_{j\ell:km}+12\lambda\delta_j+12\lambda\delta_{j\ell}-\theta_k-\theta_{k\ell:jm}-12\lambda\delta_k-12\lambda\delta_{k\ell}\in \CH(\partial\bar\M_{1,4}).\]}
\end{prop}

\begin{proof}
Consider the commutative diagram
\begin{center}
    \begin{tikzcd}
        \overline\CH^2(\M_{1,4},1)\arrow[r,"\partial^4_1"] & \CH^1(\partial\bar\M_{1,4})\\
        \overline\CH^2(\M_{1,3},1)\arrow[r,"\partial^3_1"]\arrow[u,"\pi^*"] & \CH^1(\partial\bar\M_{1,3})\arrow[u,"\pi^*"] 
    \end{tikzcd}.
\end{center}
Using Proposition~\ref{pullbackhigherChowImM13} and Proposition~\ref{ReducedFibers}, we have 
\begin{align*}
    \partial^4_1(\mathfrak p_{12})&=\partial^4_1(\pi^*(\mathfrak p_{12}))\\
    &=\pi^*(\partial^3_1(\mathfrak p_{12}))\\
    &=\pi^*(\theta_1+12\lambda\delta_1-\theta_2-12\lambda\delta_2)\\
    &=\theta_1+\theta_{14:23}+12\lambda\delta_1+12\lambda\delta_4-\theta_3-\theta_{34:12}-12\lambda\delta_3-12\lambda\delta_{34}.
\end{align*}
The result follows for all $j,k$ using the $S_4$ action.
\end{proof}

The following theorem gives a computation of $\partial_1(\mathfrak g)$. This is a form of Getzler's relation \cite[Theorem 1.8]{Getzler}, because, by the localization exact sequence, it is a codimension $2$ non-torsion relation between the boundary classes of $\bar\M_{1,4}$ that does not hold on $\partial\bar\M_{1,4}$. 
\begin{thm}\label{Getzler}
    We have  
    \[\partial_1(\mathfrak g)=4\theta_{14:23}-2\theta_1-24\lambda\delta_1-12\lambda\delta_{12}-12\lambda\delta_{13}+12\lambda\delta_{14}+24\lambda\delta_{23}+12\hat{\delta}.\]
    Hence, the above class vanishes on when pushed forward to $\bar\M_{1,4}$.
\end{thm}
The proof of this theorem is left for the next section.

\begin{prop}\label{higherChowImM14}
\mbox{}
\begin{enumerate}
     \item The map $\partial_1: \overline\CH^2(\M_{1,4},1)\to \CH^{1}(\partial\bar\M_{1,4})$ is injective, and its image is generated as $A$-module by 
     \[\tau_{jk}:=\theta_j+\theta_{j\ell:km}+12\lambda\delta_j+12\lambda\delta_{j\ell}-(\theta_k+\theta_{k\ell:jm}+12\lambda\delta_k+12\lambda\delta_{k\ell})\]
     for $j,k\in [4]$ distinct;
     \[\upsilon_{jk}:=\theta_{j\ell:km}+6\lambda\delta_{j\ell}-6\lambda\delta_{jm}-(\theta_{k\ell:jm}+6\lambda\delta_{k\ell}-6\lambda\delta_{km})\]
     for $j,k\in [4]$ distinct and
     \[\varepsilon:=2\theta_{14:23}-\theta_1-12\lambda\delta_1-6\lambda\delta_{12}-6\lambda\delta_{13}+6\lambda\delta_{14}+12\lambda\delta_{23}+6\hat\delta.\]

     \item In $\CH^2(\bar\M_{1,4})$, we have $2\lambda^2=\iota_*(\kappa)$,
     where 
     \[\kappa:=\theta_{14:23}-2\lambda(\delta_\emptyset+\delta_1+\delta_2+\delta_3+\delta_4+\delta_{12}+\delta_{13}+\delta_{24}+\delta_{34})+4(\delta_{14}+\delta_{23})+2\hat \delta\]
\end{enumerate}
\end{prop}
\begin{proof}
    We have that $\overline\CH^2(\M_{1,4},1)\xrightarrow{\partial_1} \CH^1(\partial\bar\M_{1,4})$ induces a map
    \[\partial_1^0:\overline\CH^2(\M_{1,4}^0,1)=\overline\CH^2(\M_{1,4},1)/\langle\mathfrak g\rangle\to \CH^1(\partial\bar\M_{1,4})/\langle 2\varepsilon\rangle\]
    by Theorem~\ref{Getzler}. Both the domain and target of $\partial_1^0$ have an action by $D_4=\langle (14),(12)(34)\rangle\subseteq S_4$.  
    
    By Theorem~\ref{M140}, we have $\overline\CH^2(\M_{1,4}^0,1)\cong \left(\frac{\Z}{2\Z}\right)^6$, so $\im(\partial_1^0)$ is contained in the $2$-torsion subgroup of $\CH^1(\partial\bar\M_{1,4})/\langle 2\varepsilon \rangle$. Since we have explicit generators and relations for this group, we can perform a Smith normal form computation to obtain that the torsion subgroup is generated by  $\{\tau_{jk}\},\{\upsilon_{jk}\},\varepsilon,$ and $\kappa$, where all elements have order $2$ besides $\kappa$ which has order $12$. Thus, the $2$-torsion subgroup is generated by $\{\tau_{jk}\},\{\upsilon_{jk}\},\varepsilon,$ and $6\kappa$. By Proposition~\ref{pullbackhigherChowImM14}, the $\tau_{jk}$ are contained in the image of $\partial_1^0$. 

    By Theorem~\ref{M14}, we have $\lambda^2$ has order $2$ in $\CH^2(\M_{1,4})$, and by Proposition~\ref{Order 12}, we know $\lambda^2$ has order $24$ in $\CH^2(\bar\M_{1,4})$. Thus, by exactness of 
    \[0\to \frac{\CH^1(\partial\bar\M_{1,4})}{\im(\partial_1)}=\frac{\CH^1(\partial\bar\M_{1,4})}{\langle 2\varepsilon,\im(\partial^0_1)\rangle}\xrightarrow{\iota_*} \CH^2(\bar\M_{1,4})\to \CH^2(\M_{1,4})\to 0,\]
    we can write $2\lambda^2=\iota_*(\eta)$ for some $\eta$ of order $12$ in ${\CH^1(\partial\bar\M_{1,4})}/{\langle 2\varepsilon,\im(\partial^0_1)\rangle}$. Because $\im(\partial_1^0)$ consists of only $2$-torsion, $\eta$ lifts to an order $12$ element in ${\CH^1(\partial\bar\M_{1,4})}/{\langle 2\varepsilon\rangle}$. By the above description of the torsion subgroup, we have 
    \[\eta=c_1\kappa+c_{2}\upsilon_{12}+c_{3}\upsilon_{13}+c_{4}\upsilon_{14}+c_5\varepsilon\]
    inside ${\CH^1(\partial\bar\M_{1,4})}/{\im(\partial_1)}$, where $c_1\in \{1,5,7,11\}$ and $c_i\in \{0,1\}$ for $i>1$. We need not include expressions for the $\tau_{jk}$ since at this point we know they are $0$ in ${\CH^1(\partial\bar\M_{1,4})}/{\im(\partial_1)}$, where as we do not (yet) know this for $\upsilon_{jk},\varepsilon$. Thus, we have
    \begin{equation}\label{lambdasquared14}
        2\lambda^2=\iota_*(c_1\kappa+c_{2}\upsilon_{12}+c_{3}\upsilon_{13}+c_{4}\upsilon_{14}+c_5\varepsilon).
    \end{equation}
 
A computation shows that
    \[(12)\varepsilon=\varepsilon+\tau_{12}+\upsilon_{12}\]
    and
    \[(12)\kappa=\kappa+\upsilon_{13}\]
    inside $\CH^1(\partial\bar\M_{1,4}),$ so applying $(12)$ to \eqref{lambdasquared14},
    we get
    \[2\lambda^2=\iota_*(c_1\kappa+c_1\upsilon_{13}+c_{2}\upsilon_{12}+c_{3}\upsilon_{14}+c_4\varepsilon+c_4\tau_{13}+c_4\nu_{12}).\]
    Subtracting \eqref{lambdasquared14} from this and using that we know $\iota_*$ is injective on ${\CH^1(\partial\bar\M_{1,4})}/{\im(\partial_1)}$ and $\tau_{13}=0$ in ${\CH^1(\partial\bar\M_{1,4})}/{\im(\partial_1)}$, we get
    \[0=(c_3+c_4)\upsilon_{12}+c_1\upsilon_{13}\]
    in ${\CH^1(\partial\bar\M_{1,4})}/{\im(\partial_1)}$. Regardless of the values of $c_3,c_4$, because $c_1$ is odd, we see that some $\upsilon_{jk}$ must be in $\im(\partial_1)$, hence they all must be by using the $S_4$ action.

       At this point, we know that the $\tau_{jk},\nu_{jk}\in \im(\partial_1^0)$ and hence, the kernel of $\partial_1^0$ has size at most $2$. We now argue that the kernel is trivial, implying that $\partial_1$ is injective on $\overline\CH^2(\M_{1,4},1)$. Suppose it is not. Then we have that 
    \begin{equation}\label{firstIso}
        \overline\CH^2(\M_{1,4}^0,1)/\ker(\partial_1^0)\cong \langle \tau_{jk}\rangle \oplus \langle \upsilon_{jk}\rangle.
    \end{equation}
    By computing the action, we can see that both $\langle \tau_{jk}\rangle$ and $\langle \upsilon_{jk}\rangle$ are indecomposable $\F_2$-representations of $D_4$, of dimension
    $3$ and $2$ respectively.  In Lemma~\ref{D4action}, we computed the action of $D_4$ on $\overline\CH^2(\M_{1,4}^0,1)$. It may be written as a direct sum 
    \[\langle \mathfrak q_1,\mathfrak q_2,\mathfrak q_3,\mathfrak q_4\rangle\oplus \langle \mathfrak q_5,\mathfrak q_6\rangle.\] The kernel must be invariant, so it must be generated by
    \[\mathfrak q_1+\mathfrak q_2+\mathfrak q_3+\mathfrak q_4,\]
    \[\mathfrak q_5+\mathfrak q_6,\]
    or
    \[\mathfrak q_1+\mathfrak q_2+\mathfrak q_3+\mathfrak q_4+\mathfrak q_5+\mathfrak q_6.\]
    By considering the posets of subrepresentations involved, one sees that the only way \eqref{firstIso} can hold is if the kernel is generated by $\mathfrak q_1+\mathfrak q_2+\mathfrak q_3+\mathfrak q_4$. However, we construct a $D_4$-linear section to $\partial_1^0$ over $\langle \tau_{jk}\rangle$, which must therefore send $\tau_{14}+\tau_{23}$ to $\mathfrak q_1+\mathfrak q_2+\mathfrak q_3+\mathfrak q_4$. Thus the kernel must be trivial.

    To construct the section over $\langle \tau_{jk}\rangle$, we consider the subrepresentation $\langle \mathfrak p_{jk}|j<k\in [4]\rangle \subseteq \overline\CH^2(\M_{1,4},1)$. By Proposition~\ref{pullbackhigherChowImM14}, the image is equal to $\langle \tau_{jk}\rangle$. Moreover, from the relations for the $\mathfrak p_{jk}$ given by Corollary~\ref{pijpjk}, we see that the dimension of $\langle \mathfrak p_{jk}\rangle$, is at most three, hence equal to three, and $\langle \mathfrak p_{jk}\rangle \to \langle \tau_{jk}\rangle$ is an isomorphism.  

  We next compute the coefficient $c_1$. To do this, we pull back the equality \eqref{lambdasquared14} along $\xi_{\Delta_{14}}: \bar\M_{1,3}\cong \bar\M_{\Delta_{14}}\to \bar\M_{1,4}$. The markings are labeled so that the first marking is preserved, the third marking maps to the fourth marking, and the second marking is where the component containing the markings $2,3$ gets glued to. We use Theorem~\ref{tautologicalcalculus} and Proposition~\ref{psi} to compute
  \begin{align*}
      \xi_{\Delta_{14}}^*(\lambda)&=\lambda\\
      \xi_{\Delta_{14}}^*(\delta_\emptyset)&=\delta_{\emptyset}\\
      \xi_{\Delta_{14}}^*(\delta_j)&=0,\text{ for }j=2,3\\
      \xi_{\Delta_{14}}^*(\delta_1)&=\delta_{1}\\
      \xi_{\Delta_{14}}^*(\delta_4)&=\delta_{3}\\
      \xi_{\Delta_{14}}^*(\delta_{jk})&=0, \text{ for } \{j,k\}\notin \{\{1,4\},\{2,3\}\}\\
      \xi_{\Delta_{14}}^*(\delta_{14})&=-\lambda-\delta_{\emptyset}-\delta_{1}-\delta_{3}\\
      \xi_{\Delta_{14}}^*(\delta_{23})&=\delta_2\\
      \xi_{\Delta_{14}}^*(\iota_*(\theta_j))&=\xi_{\Delta_{14}}^*(\iota_*(\theta_{1j}))=0,  \text{ for }j=2,3\\\
      \xi_{\Delta_{14}}^*(\theta_1)&=6\lambda(\lambda-6\delta_{1}+6\delta_2+6\delta_{3}+6\delta_{\emptyset})\\
      \xi_{\Delta_{14}}^*(\theta_4)&=6\lambda(\lambda+6\delta_{1}+6\delta_2-6\delta_{3}+6\delta_{\emptyset})\\
      \xi_{\Delta_{14}}^*(\theta_{14:23})&=6\lambda(\lambda+6\delta_{1}-6\delta_2+6\delta_{3}+6\delta_{\emptyset})
  \end{align*}
where we use the formula for $\theta_j\in \CH(\bar\M_{1,3})$ from Theorem~\ref{barM13}. From these, we can compute
\[\xi_{\Delta_{14}}^*(\iota_*(\kappa))=2\lambda^2\]
\[\xi_{\Delta_{14}}^*(\iota_*(\varepsilon))=0.\]
 Pulling back  \eqref{lambdasquared14} along $\xi_{\Delta_{14}}$, we see that
\[2\lambda^2=c_1\cdot 2\lambda^2,\]
so $c_1=1$. 

 Now because $\partial_1^0$ is injective, we know that one of $6\kappa,6\kappa+\varepsilon,\varepsilon$ must be in $\im(\partial_1^0)$. But our pullback computations to $\bar\M^{\Delta_{14}}$ show that $6\kappa,6\kappa+\varepsilon$ cannot be in the image of $\partial_1^0$, since they do not pull back to $0$. Thus, $\varepsilon$ is in the image of $\partial_1^0$. 

Finally, for (2), note that \eqref{lambdasquared14} becomes
    \[2\lambda^2=\iota_*(\kappa)\]
    because the rest of the terms are in $\im(\partial_1)$, and hence pushforward to $0$ under $\iota_*$.
\end{proof}

\begin{cor}\label{higherM14}
    The indecomposable first higher Chow group of $\M_{1,4}$ is generated as a $\Lambda$-module by $\mathfrak t_j$, $\mathfrak p_{jk}$ subject to only the relations
    \begin{align*}
    &\lambda \mathfrak t_j\\
    &2(\mathfrak t_j-\mathfrak t_k)\\
    &\mathfrak t_j+\mathfrak t_k-\mathfrak t_\ell-\mathfrak t_m\\
    &2\mathfrak p_{jk}\\
    &\mathfrak p_{jk}+\mathfrak p_{j\ell}+\mathfrak p_{k\ell}\\
    &\lambda \mathfrak p_{jk}+\lambda \mathfrak p_{\ell m}
    \end{align*}
    for $\{j,k,\ell,m\}=[4]$, where $\mathfrak t_j$ are some elements of $\overline\CH^2(\M_{1,4},1)$ with $2\mathfrak t_j=\mathfrak g$. Moreover, the action of $S_4$ acts by permuting the subscripts.
\end{cor}
\begin{proof}
   This gives injectivity of 
   \[\partial_1: \overline\CH(\M_{1,4},1)\to \CH(\partial\bar\M_{1,4})\]
   in degree $2$. In degrees larger than $2$, Theorem~\ref{M14} says that $\overline\CH^{\geq 3}(\M_{1,4},1)$ is generated by the $\lambda\mathfrak p_{jk},$ subject to certain relations. Therefore, the image of $\partial_1$ in these degrees is generated by $\lambda\tau_{jk}$, and a Smith normal form computation shows that $\partial_1$ is injective in these degrees.
   
   Because $\partial_1$ is injective, it gives an isomorphism onto the sub $A$-module of $\CH(\partial\bar\M_{1,4})$ generated by $\tau_{jk},\nu_{jk},$ and $\varepsilon$ by the proposition. We have seen $\partial_1(\mathfrak p_{jk})=\tau_{jk}$. Choose $\mathfrak t_1\in \overline\CH^2(\M_{1,4},1)$ so that $\partial_1(\mathfrak t_1)=\varepsilon$, and then define $\mathfrak t_j=(1j)\mathfrak t_1$. The fact that $S_4$ acts in the stated way and the fact that the relations listed are correct follow by computing in $\CH(\partial\bar\M_{1,4})$. Lastly, we have
   \[\partial_1(2\mathfrak t_j)=\partial_1(2\mathfrak t_1)=2\varepsilon=\partial_1(\mathfrak g),\]
    so $2\mathfrak t_j=\mathfrak g$ by injectivity of $\partial_1$.
\end{proof}

Finally, we give our presentation for $\CH(\overline\M_{1,4})$. 
\begin{thm}\label{barM14}
    The Chow ring of $\bar\M_{1,4}$ is given by
    \[\CH(\bar\M_{1,4})=\frac{\Z[\lambda,\delta_\emptyset,\{\delta_{jk}\}_{jk},\{\delta_j\}]}{I}\]
where $I$ is generated by the relations
\begin{align*}   &4\lambda[\lambda(\delta_\emptyset+\delta_j+\delta_k+\delta_\ell+\delta_m+\delta_{j\ell}+\delta_{jm}+\delta_{k\ell}+\delta_{km}-2\delta_{jk}-2\delta_{\ell m})-\hat{\delta}]\\
    &12\lambda^2(\lambda\delta_{jk}+\delta_{\emptyset}\delta_{jk}+\delta_{j}\delta_{jk}+\delta_{k}\delta_{jk}+\delta_{jk}\delta_{\ell m})\\
    &\delta_{\emptyset}(\delta_{jk}+\delta_{\ell m}-\delta_{j\ell}-\delta_{km})\\
    &\delta_{\emptyset}(\delta_{j}+\delta_{jk}-\delta_{\ell}-\delta_{\ell k})\\
&12\lambda(\lambda\delta_{j\ell}+\lambda\delta_{jk}+\lambda\delta_{jm}+\hat{\delta})\\
& \delta_\emptyset^2+\lambda\delta_\emptyset+\delta_{\emptyset}\delta_{jk}+\delta_{\emptyset}\delta_{j}+\delta_{\emptyset}\delta_{k}\\
& \delta_{jk}^2+\lambda\delta_{jk}+\delta_{\emptyset}\delta_{jk}+\delta_j\delta_{jk}+\delta_k\delta_{jk}\\
& \delta_{jk}\delta_{j\ell}\\
& \delta_j(\delta_{jk}-\delta_{j\ell}) \\
& \delta_{jk}\delta_\ell \\
& \delta_j^2+\lambda\delta_j+\delta_{\emptyset}\delta_j+\delta_j\delta_{jk}\\
& \delta_j\delta_k \\
& 24\lambda^2
\end{align*} 
for $\{j,k,\ell,m\}=[4]$. Moreover, in $\CH(\bar\M_{1,4})$, we have
{\footnotesize\begin{align*}   \theta_{jk:\ell m}&=2\lambda(\lambda+\delta_\emptyset+\delta_j+\delta_k+\delta_\ell+\delta_m+\delta_{j\ell}+\delta_{jm}+\delta_{k\ell}+\delta_{km}-2\delta_{jk}-2\delta_{\ell m})-2\hat{\delta}\\
    \theta_j&=2\theta_{jk}+6\lambda(\delta_{jk}-2\delta_j-\delta_{j\ell}-\delta_{jm}+2\delta_{\ell m})+6\hat\delta\\
    \omega_{jk}&=12\lambda(\delta\delta_{jk}+\delta_j\delta_{jk}+\delta_k\delta_{jk}-\delta_{jk|\ell m}+\delta_\emptyset\delta_{jk})\\
    \sigma_{jk:\ell m}&=24\lambda\delta_\emptyset\delta_j\delta_{jk}
\end{align*}}
for $\{j,k,\ell,m\}=[4]$.
\end{thm}
Unlike the case of fewer marked points, it is not obvious that these relations are equivalent to those from the previous description in \cite{BDL1}. The author has verified via a computer algebra system that the ideals of relations do indeed coincide. 

By summing the expression in this theorem given for $\theta_{jk:\ell m}$ over all $j<k$, one obtains Getlzer's relation \cite[Theorem 1.8]{Getzler} up to addition of the torsion class $12\lambda^2$. This means that, integrally, Getzler's relation needs a torsion modification. This was previously noticed in \cite[Corollary 5.9]{BDL2}.

\begin{proof}
    We have the exact sequence
    \[0\to \frac{\CH(\partial\bar\M_{1,4})}{\im(\partial_1)} \xrightarrow{\iota_*} \CH(\bar\M_{1,4})\to \CH(\M_{1,4})\to 0\]
    of $A$-modules. By Theorem~\ref{M14}, we have
    \[\CH(\M_{1,4})=\frac{\Z[\lambda]}{(12\lambda,2\lambda^2)}=\frac{A\langle 1\rangle}{(12\lambda\cdot 1, 2\lambda^2\cdot 1)},\]
    and Proposition~\ref{boundarybarM14} gives an $A$-module presentation for $\CH(\partial\bar\M_{1,4})$. 
    By Proposition~\ref{Order 12}, $\phi=12\lambda\cdot 1\in \CH(\bar\M_{1,4})$, and by Proposition~\ref{higherChowImM14}, we have $2\lambda^2\cdot 1=\iota_*(\kappa)$. Thus, using Lemma~\ref{exactPresentation}, we have a presentation for $\CH(\bar\M_{1,4})$.

    We also have $\iota_*(\varepsilon)=0$ from  Proposition~\ref{higherChowImM14}, which, together with $\iota_*(\kappa)=2\lambda^2\cdot 1$, gives
    {\footnotesize\begin{align*}
        \theta_{14:23}=&2\lambda(\lambda\cdot 1+\delta_\emptyset+\delta_1+\delta_2+\delta_3+\delta_4+\delta_{12}+\delta_{13}+\delta_{24}+\delta_{34}-2\delta_{14}-2\delta_{23})-2\hat{\delta}
    \end{align*}}
    and 
    \begin{align*}
        \theta_1=&2\theta_{14:23}+6\lambda(\delta_{14}-2\delta_1-\delta_{12}-\delta_{13}+2\delta_{23})+6\hat{\delta}
    \end{align*}
    in $\CH(\bar\M_{1,4})$. By applying the $S_4$ action, we obtain the descriptions of $\theta_j,\theta_{jk}$ given in the statement. Thus, we do not have to use $\theta_j$, $\theta_{jk}$ as generators. With this simplification, we have that $\CH(\bar\M_{1,4})$ is generated by
    \[1,\delta_\emptyset,\{\delta_j,\delta_{\emptyset|j},\delta_{j|*}\}_j, \{\delta_{jk},\delta_{\emptyset|jk}\}_{jk},\delta_{12|34},\delta_{13|24},\delta_{14|23},\delta_{\emptyset|*|*}\]
    subject to only the relations
\begin{align*}   &4\lambda[\lambda(\delta_\emptyset+\delta_j+\delta_k+\delta_\ell+\delta_m+\delta_{j\ell}+\delta_{jm}+\delta_{k\ell}+\delta_{km}-2\delta_{jk}-2\delta_{\ell m})-\hat{\delta}]\\
    &12\lambda^2(\lambda\delta_{jk}+\delta_{\emptyset|jk}+\delta_{j|*}+\delta_{k|*}+\delta_{jk|\ell m})\\
    &\delta_{\emptyset|jk}+\delta_{\emptyset|\ell m}-\delta_{\emptyset|j\ell}-\delta_{\emptyset|km}\\
    &\delta_{\emptyset|j}+\delta_{\emptyset|jk}-\delta_{\emptyset|\ell}-\delta_{\emptyset|\ell k}\\
&12\lambda(\lambda\delta_{j\ell}+\lambda\delta_{jk}+\lambda\delta_{jm}+\hat{\delta})
\end{align*}      
for $\{j,k,\ell,m\}=[4].$

Using Theorem~\ref{tautologicalcalculus} and Proposition~\ref{psi} to compute products between these generators, we obtain
\begin{align*}
\delta_\emptyset^2&=-\lambda\delta_\emptyset-\delta_{\emptyset|jk}-\delta_{\emptyset|j}-\delta_{\emptyset|k}\\
\delta_\emptyset\delta_{jk}&=\delta_{\emptyset|jk}\\
\delta_\emptyset\delta_j&=\delta_{\emptyset|j}\\
\delta_{jk}^2&=-\lambda\delta_{jk}-\delta_{\emptyset|jk}-\delta_{j|*}-\delta_{k|*}\\
\delta_{jk}\delta_{j\ell}&=0\\
\delta_{jk}\delta_{\ell m}&=\delta_{jk|\ell m}\\\
\delta_{jk}\delta_j&=\delta_{j|*}\\
\delta_{jk}\delta_\ell&=0\\
\delta_j^2&=-\lambda\delta_j-\delta_{\emptyset|j}-\delta_{j|*}\\
\delta_j\delta_k&=0 \\
\delta_\emptyset\delta_{j|*}&=\delta_{\emptyset|*|*},
\end{align*}
for $\{j,k,\ell,m\}=[4]$. From these, along with our $A$-module relations, we can compute the rest of the products between generators.

We see that we only need to include the degree $1$ generators for $\CH(\bar\M_{1,4})$ as a ring. We get 
\[\CH(\bar\M_{1,4})=\frac{\Z[\lambda,\delta_\emptyset,\{\delta_{jk}\}_{jk},\{\delta_j\}_j]}{I},\]
where $I$ is the ideal generated as an $A$-module by
\begin{align*}   &4\lambda[\lambda(\delta_\emptyset+\delta_j+\delta_k+\delta_\ell+\delta_m+\delta_{j\ell}+\delta_{jm}+\delta_{k\ell}+\delta_{km}-2\delta_{jk}-2\delta_{\ell m})-\hat{\delta}]\\
    &12\lambda^2(\lambda\delta_{jk}+\delta_{\emptyset}\delta_{jk}+\delta_{j}\delta_{jk}+\delta_{k}\delta_{jk}+\delta_{jk}\delta_{\ell m})\\
    &\delta_{\emptyset}(\delta_{jk}+\delta_{\ell m}-\delta_{j\ell}-\delta_{km})\\
    &\delta_{\emptyset}(\delta_{j}+\delta_{jk}-\delta_{\ell}-\delta_{\ell k})\\
&12\lambda(\lambda\delta_{j\ell}+\lambda\delta_{jk}+\lambda\delta_{jm}+\hat{\delta})\\
& \delta_\emptyset^2+\lambda\delta_\emptyset+\delta_{\emptyset}\delta_{jk}+\delta_{\emptyset}\delta_{j}+\delta_{\emptyset}\delta_{k}\\
& \delta_{jk}^2+\lambda\delta_{jk}+\delta_{\emptyset}\delta_{jk}+\delta_j\delta_{jk}+\delta_k\delta_{jk}\\
& \delta_{jk}\delta_{j\ell}\\
& \delta_j(\delta_{jk}-\delta_{j\ell}) \\
& \delta_{jk}\delta_\ell \\
& \delta_j^2+\lambda\delta_j+\delta_{\emptyset}\delta_j+\delta_j\delta_{jk}\\
& \delta_j\delta_k \\
& 24\lambda^2
\end{align*} 
for $\{j,k,\ell,m\}=[4]$.

Finally, by Proposition~\ref{boundarybarM14}, we have 
\[\omega_{jk}=12\lambda(\delta\delta_{jk}+\delta_j\delta_{jk}+\delta_k\delta_{jk}-\delta_{jk|\ell m}+\delta_\emptyset\delta_{jk})\] 
\[\sigma_{jk:\ell m}=24\lambda\delta_\emptyset\delta_j\delta_{jk}\]
for $\{j,k,\ell,m\}=[4]$.
\end{proof}

\section{Proof of Theorem~\ref{Getzler}}
\begin{definition}
    Let $\widetilde{\M}_{1,2}$ be the closure of the embedding $\M_{1,2}^0\hookrightarrow \M_{1,3}\subseteq \bar\M_{1,3}$ from Lemma~\ref{M130Complement} and let $\widetilde{\M}_{1,3}$ be the closure of the embedding $\M_{1,3}^0\hookleftarrow \M_{1,4}\subseteq \bar\M_{1,4}$ from Lemma~\ref{M140Complement}.  
\end{definition}

To prove Theorem~\ref{Getzler}, we utilize the commutative diagram
    \begin{center}
        \begin{tikzcd}
            \overline\CH^1(\M_{1,3}^0,1) \arrow[r,"\partial_1"]\arrow[d] & \CH_2(\widetilde{\M}_{1,3}\setminus \M_{1,3}^0)\arrow[d]\\
            \overline\CH^2(\M_{1,4},1) \arrow[r,"\partial_1"] & \CH^1(\partial\bar\M_{1,4})
        \end{tikzcd}
    \end{center}
First, we wish to compute $\CH_2(\widetilde \M_{1,3}\setminus \M_{1,3}^0)$, i.e. the $2$-dimensional components of $\widetilde \M_{1,3}\setminus \M_{1,3}^0$. We think about the map $\widetilde{\M}_{1,3}\to \bar\M_{1,3}$. The dimension of the image of $2$-dimensional component of $\widetilde{\M}_{1.3}\setminus \M_{1,3}^0$ in $\bar\M_{1,3}\setminus \M_{1,3}^0$ is either $2$, in which case the map is generically finite, or $1$, in which case the fibers are generically $1$ dimensional. Thus, we only need to consider points in $\widetilde{\M}_{1,3}$ lying over $\M_{1,2}^0\subseteq \bar\M_{1,3}$ or $\M^{\Gamma}_{1,3}$ for $\Gamma$ a stable graph of codimension $1$ or $2$.   

Recall the isomorphism $\M_{1,3}^0\cong \M_{1,4}\setminus \M_{1,4}^0$ from Lemma~\ref{M140Complement}. We try to extend the morphism
    \[\M_{1,3}^0\xrightarrow{\sim} \M_{1,4}\setminus \M_{1,4}^0\subseteq \bar\M_{1,4}\]
    to a larger domain inside $\bar\M_{1,3}$. Because the morphism may be written as 
      \[(C,p_1,p_2,p_3)\mapsto (C,p_1,p_2,p_3,p_2\oplus p_3)\]
    it extends to $\M_{1,3}^\Phi$, as $\oplus$ is a morphism on a family of irreducible nodal genus $1$ curves, away from nodes. Moreover, we can extend this map to $\M_{1,2}^0\subseteq \M_{1,3}$ by sending $(C,p_1,p_2,p_3)$ to the curve $C$ with an attached $\bP^1$ at $p_1$ that contains $p_1$ and $p_4$. Thus, over $\M_{1,2}^0$ and $\M_{1,3}^\Phi$, there are unique $2$-dimensional components of $\widetilde{\M}_{1,3}\setminus \M_{1,3}^0$. Let $\widetilde \M_{1,3}^\Xi$ and $\widetilde \M_{1,3}^\Phi$ denote these components, respectively.

 Next, note the curves in $\M_{1,3}^0\hookrightarrow \bar\M_{1,4}$ are invariant under the permutation $(14)(23)$, as the isomorphism $q\mapsto 2p_4-q$ takes $(C,p_1,p_2,p_3,p_4)$ to $(C,p_4,p_3,p_2,p_1)$ for $(C,p_1,p_2,p_3,p_4)\in \M_{1,3}^0$. Thus, the points in $\widetilde{\M}_{1,3}$ must also be invariant under $(14)(23)$. Moreover, the other permutation $(12)(34)$ will not generally fix the curves in $\M_{1,3}^0$, but it does send the locus $\M_{1,3}^0$ to itself. 

For $j=2,3$, over points in $\M_{1,3}^{\Delta_j}$ there is a unique point in $\bar\M_{1,4}$ that is invariant under $(14)(23)$, given by attaching a $\bP^1$ containing the markings $j$ and $4$ to the marking $j$. Thus, $\bar\M^{\Delta_{13|24}}$ and $\bar\M^{\Delta_{12|34}}$ are the only $2$-dimensional components over $\M_{1,3}^{\Delta_2}$ and $\M_{1,3}^{\Delta_3}$. Additionally, over a general point of $\M_{1,3}^{\Delta_\emptyset}$, there is a unique point invariant under $(14)(23)$. Over the other points of $\M_{1,3}^{\Delta_\emptyset}$, there is exactly one additional curve invariant under $(14)(23)$, meaning there will not be a $2$-dimensional component over these points. We can describe the $2$-dimensional component of $\widetilde{\M}_{1,3}\setminus \M_{1,3}^0$ over $\bar\M^{\Delta_\emptyset}$ as the image of $V\times \bar\M_{1,1}\subseteq \bar\M_{0,5}\times \bar\M_{1,1}$ under the map gluing the $5$-th marked point of the genus $0$ curve to the only marked point of the genus $1$ curve, where $V$ is the closure in $\bar\M_{0,5}$ of the curves in $\M_{0,5}$ invariant under $(14)(23)$. Call this component $\widetilde{\M}_{13}^{\Delta_\emptyset}$. 

There are exactly two $(14)(23)$ invariant curves over every point in $\bar\M^{\Delta_1}_{1,3}$. These are given by placing $p_4$ either on a rational component with $p_1$ or at the point $2q$, where $q$ is the point on the genus $1$ component where the rational curve containing $2,3$ is attached. We have that curves of the latter type are all inside $\widetilde{\M}_{1,3}$ because they are equal to $(12)(34)$ applied to a point of $\widetilde \M_{1,3}^\Xi$. Moreover, the map $\varphi: \bar\M_{1,3}\dashrightarrow \bar\M_{1,4}$ extends to a morphism over an open subset of $\bar\M_{1,3}^{\Delta_1}$, because $\bar\M_{1,3}$ is normal and $\bar\M_{1,4}$ is proper. Thus,  there is only one point in $\widetilde{\M}_{1,3}$ lying over a general point of $\bar\M_{1,3}^{\Delta_1}$. And the remaining points have at most $2$ points lying over them, so we can conclude that there is a unique $2$-dimensional component lying over $\M_{1,3}^{\Delta_1}$. Call this component $\widetilde\M_{1,3}^{\Delta_1}$. This takes care of all of the codimension $1$ graphs of $\bar\M_{1,3}$.

For all codimension $2$ graphs $\Gamma$ of $\bar\M_{1,3}$ besides $\Theta_1$, for any point in $\M^{\Gamma}_{1,3}$, one can see that there are only finitely many $(14)(23)$-invariant preimages over in $\bar\M_{1,4}$, meaning there cannot be a $2$-dimensional component of $\widetilde{\M}_{1,3}\setminus \M_{1,3}^0$ lying over these $\M^{\Gamma}_{1,3}$. But over a point in $\M^{\Theta_1}_{1,3}$, there are infinitely many preimages invariant under $(14)(23)$, obtained by placing $p_4$ anywhere on the same component as $p_1$. These are the points of $\bar\M_{1,4}$ in $\M_{1,4}^{\Theta_{14:23}}$. We do have $\M^{\Theta_{14:23}}_{1,4}\subseteq \widetilde{\M}_{1,3}$, hence $\bar\M_{1,4}^{\Theta_{1,4}}$ is a $2$-dimensional component of $\widetilde{\M}_{1,3}\setminus \M_{1,3}^0$, but we will not show this now. It will follow from Lemma~\ref{last}.

We have thus proven
\begin{prop}
    $\CH_2(\widetilde \M_{1,3})$ is freely generated by 
    \[[\widetilde \M_{1,3}^\Phi], [\widetilde\M_{1,3}^{\Delta_1}],\delta_{13|24},\delta_{12|34}, [\widetilde\M_{1,3}^{\Delta_\emptyset}], [\widetilde\M_{1,3}^\Xi]\] 
    and potentially $\theta_{14:23}$.
\end{prop}

By the proposition, we can write $\partial_1(\mathfrak g)=\divisor(f_0)$ as
\[a_0[\widetilde \M_{1,3}^\Phi]+a_1[\widetilde\M_{1,3}^{\Delta_1}]+a_2\delta_{13|24}+a_3\delta_{12|34}+a_4[\widetilde{\M}_{1,3}^{\Delta_\emptyset}]+a_5[\widetilde{\M}_{1,3}^\Xi]+a_6\theta_{14:23}\]
for some $a_i\in \Z$. By commutativity of 
    \begin{center}
        \begin{tikzcd}
            \overline\CH^1(\M_{1,3}^0,1) \arrow[r,"\partial_1"]\arrow[d,"\sim"] & \CH_2(\widetilde{\M}_{1,3}\setminus \M_{1,3}^0)\arrow[d]\\
            \overline\CH^1(\M_{1,3}^0,1) \arrow[r,"\partial_1'"] & \CH_2(\bar{\M}_{1,3}\setminus \M_{1,3}^0)
        \end{tikzcd}
    \end{center}
the above expression for $\partial_1(\mathfrak g)$ pushes forward to 
\[\partial_1'(\mathfrak g)=a_0\phi+a_1\delta_1+a_2\delta_2+a_3\delta_3+a_4\delta_\emptyset+a_5[\widetilde{\M}_{1,2}]\]
on $\partial\bar\M_{1,3}$, hence 
\[0=a_0\phi+a_1\delta_1+a_2\delta_2+a_3\delta_3+a_4\delta_\emptyset+a_5[\widetilde{\M}_{1,2}]\]
in $\bar\M_{1,3}$. Because $\overline \CH(\M_{1,3}^0,1)$ is generated by $\mathfrak g$ by Theorem~\ref{M130}, this is the only nonzero relation between these classes in $\bar\M_{1,3}$.

Now we compute $[\widetilde \M_{1,2}]$.
\begin{lemma}\label{classoftildeM12}
    The class $[\widetilde\M_{1,2}]\in \CH(\bar\M_{1,3})$ is given by
    \[[\widetilde{\M}_{1,2}]=2\lambda-\delta_1+2\delta_2+2\delta_3+2\delta_\emptyset.\]
\end{lemma}
This Lemma gives
\[0=\phi-6\delta_1+12\delta_2+12\delta_3+12\delta_\emptyset-6[\widetilde{\M}_{1,2}]\]
in $\bar\M_{1,3}$ by multiplying by $6$. Since $\divisor(f_0)$ gives the only relation between these classes up to scaling, this relation must be an integer multiple of $\divisor(f_0)$. Because this above relation is primitive, we can conclude that either $\divisor(f_0)$ or $\divisor(f_0^{-1})$ is given by the right-hand side. But, by definition of $f_0$, it is clear that the order of vanishing along irreducible nodal curves is negative. Thus, we have $a_0=-1, a_1=a_5=6,$ and $a_2=a_3=a_4=-12$. 

\begin{proof}[Proof of Lemma~\ref{classoftildeM12}]
By Theorem~\ref{barM13}, we can write
\begin{equation}\label{kappa}
    [\widetilde{\M}_{1,2}]=b_0\lambda+b_1\delta_1+b_2\delta_2+b_3\delta_3+b_4\delta_\emptyset
\end{equation}
for some $b_i\in \Z$. 

For a curve $(C,p_1,p_2,p_3)\in \M_{1,2}^0\subseteq \M_{1,3}$, elliptic curve inversion swaps the points $p_2$ and $p_3$, giving an isomorphism $(C,p_1,p_2,p_3)\xrightarrow{\sim} (C,p_1,p_3,p_2)$. Thus, every point in $\M_{1,2}^0\subseteq \M_{1,3}$ is invariant under $(23)$, hence the same is true for $\widetilde{\M}_{1,2}$. Applying the permutation $(23)$ to both sides of \eqref{kappa}, we see that $b_2=b_3$. 

Next, the composite $\widetilde{\M}_{1,2}\hookrightarrow \bar\M_{1,3}\xrightarrow{\pi} \bar\M_{1,2}$ has degree $1$, since it has a rational inverse given by  the map $\M_{1,2}^0\hookrightarrow \M_{1,3}$. Thus, after pushing forward along $\pi$, \eqref{kappa} becomes 
\[[\bar\M_{1,2}]=(b_1+b_2)[\bar\M_{1,2}],\]
so $b_1+b_2=1$. 

Let $\pi':\bar\M_{1,3}\to \bar\M_{1,2}$ denote the map forgetting the first point. Given two distinct points $p,p'$ on a smooth genus $1$ curve, there are always $4$ points $q$ such that $p+p'\sim 2q$. Thus, the composite $\widetilde\M_{1,2}\hookrightarrow \bar\M_{1,3}\xrightarrow{\pi'}\bar\M_{1,2}$ has degree $4$. Thus, after pushing forward along $\pi'$, \eqref{kappa} becomes 
\[4[\bar\M_{1,2}]=(b_2+b_3)[\bar\M_{1,2}],\]
so $b_2+b_3=4$. Putting these equations together, we get $b_1=-1$, $b_2=b_3=2$.

No points in $\bar\M_{1,3}^{\Delta_1}$ are invariant under $(23)$, so its intersection with $\widetilde{\M}_{1,2}$ must be empty. Thus, multiplying by $\delta_2$, we have
\[0=b_0\lambda\delta_2+b_2\delta_2^2+b_4\delta_\emptyset\delta_2=(b_0-b_2)\lambda\delta_2+(b_4-b_2)\delta_\emptyset\delta_2\]
from which we get $b_0=b_4=b_2=2$. 
\end{proof}

Now, we push forward $\divisor(f_0)$ to $\bar\M_{1,4}$, giving
{\small\[\partial_1((\M_{1,3}^0,f_0))=a_6\theta_{14:23}-[\widetilde \M_{1,3}^\Phi]+6[\widetilde\M_{1,3}^{\Delta_1}]-12\delta_{13|24}-12\delta_{12|34}-12[\widetilde{\M}_{1,3}^{\Delta_\emptyset}]+6[\widetilde{\M}_{1,3}^\Xi]\]}
We want to express $[\widetilde\M_{1,3}^\Phi],[\widetilde\M_{1,3}^{\Delta_1}], [\widetilde{\M}_{13}^{\Delta_\emptyset}], [\widetilde{\M}_{1,3}^\Xi]$ in terms of our generators for the $\CH(\partial\bar\M_{1,4})$ given in Proposition~\ref{boundarybarM14}. 

Note $\widetilde{\M}_{1,3}^\Xi$ is equal to the image of $\widetilde{\M}_{1,2}$ along the section $\bar\M_{1,3}\xrightarrow{\sim} \bar\M^{\Delta_{2,3}}\hookrightarrow \bar\M_{1,4}$. Thus, we have
\[[\widetilde{\M}_{1,3}^\Xi]=2\lambda\delta_{23}-\delta_{14|23}+2\delta_{2|*}+2\delta_{3|*}+2\delta_{\emptyset|23}\]
on $\partial\bar\M_{1,4}$ by Lemma~\ref{classoftildeM12}. 
As noted before, the image of $\widetilde{\M}_{1,3}^\Xi$ under $(12)(34)$ is $\widetilde{\M}_{1,3}^{\Delta_1}$, so we have
\[[\widetilde{\M}_{1,3}^{\Delta_1}]=2\lambda\delta_{14}-\delta_{14|23}+2\delta_{1|*}+2\delta_{4|*}+2\delta_{\emptyset|14}\]
on $\partial\bar\M_{1,4}$.

\begin{lemma}
    Inside $\CH(\partial\bar\M_{1,4})$, we have $[\widetilde{\M}_{1,3}^{\Delta_\emptyset}]=\delta_{\emptyset|2}+\delta_{\emptyset|3}+\delta_{\emptyset|23}$.
\end{lemma}
\begin{proof}
    As noted above, $\widetilde{\M}_{1,3}^{\Delta_\emptyset}$ is equal to the pushforward of $V\times \bar\M_{1,1}$ under $\bar\M_{0,5}\times \bar\M_{1,1}\to \bar\M^{\Delta_\emptyset}_{1,4}\subseteq \partial\bar\M_{1,4}$, where $V$ is the closure in $\bar\M_{0,5}$ of the subvariety of curves fixed by $(14)(23)$ in $\M_{0,5}$. We compute $[V]\in\CH^1(\bar\M_{0,5})$. It is more convenient to compute the class of $V':=(15)V$.
    
    For ease of notation, set $D_{ij}:=D(ij|k\ell m)$. By using Theorem~\ref{WDVV mod 2} or Theorem~\ref{barM0n}, one can see that the divisors $D_{12}, D_{13}. D_{14},D_{15},$ and $D_{45}$ form a basis for $\CH^1(\bar\M_{0,5})$, so we can write
\[[V']=a_{12}D_{12}+a_{13}D_{13}+a_{14}D_{14}+a_{15}D_{15}+a_{45}D_{45},\]
for some coefficients $a_{ij}\in \Z$. Because this locus is fixed point wise by $(45)(23)$, after applying this automorphism to both sides, we get $a_{12}=a_{13}$ and $a_{14}=a_{45}$. Now we push forward along maps $\bar\M_{0,5}\to \bar\M_{0,4}$ forgetting a marked point, noting that the degree of $D_{ij}\to \bar\M_{0,4}$ is $1$ if the point being forgotten is $i$ or $j$ and $0$ otherwise. To calculate the degree for the left-hand side, we work over the opens $\M_{0,5}\to \M_{0,4}$.

We have 
\[(45)(23)(\bP^1,\infty,x,y,1,-1)=(\bP^1,\infty,y,x,-1,1)=(\bP^1,-y,-x,1,-1),\] 
so $(\bP^1,\infty,x,y,1,-1)\in V$ if and only if $x=-y$. Thus, for the morphism forgetting the third marked point, $V\subseteq \bar\M_{0,5}\to \bar\M_{0,4}$ has degree $1$. Thus, we get
\[1=a_{13}.\]
By symmetry in the definition of $V$, the degree of the morphism forgetting the fourth marked point gives is also $1$, so we get
\[1=a_{14}+a_{45}.\]
Next, note that
\[(\bP^1,x,y,1,-1)=\left(\bP^1,\infty,\frac{-xy+1}{y-x},1,-1\right)\]
using the automorphism $t\mapsto \frac{-xt+1}{t-x}$. Hence, the morphism $V\subseteq \bar\M_{0,5}\to \bar\M_{0,4}$ forgetting the first point looks like
\[y\mapsto \frac{y^2+1}{2y},\]
which has degree $2$. Thus, we get
\[2=a_{12}+a_{13}+a_{14}+a_{15}.\]
Combining the equations on the coefficients, we get
\[[V']=D_{12}+D_{13}+D_{45},\]
and hence
\[[V]=D_{25}+D_{35}+D_{14}.\]
Therefore,
\[[\widetilde{\M}_{1,3}^{\Delta_\emptyset}]=\iota_*([V]\times [\bar\M_{1,1}])=\delta_{\emptyset|2}+\delta_{\emptyset|3}+\delta_{\emptyset|23}.\]
\end{proof}

\begin{lemma}
     Inside $\CH(\partial\bar\M_{1,4})$, we have
     \[[\widetilde\M_{1,3}^\Phi]=2\theta_1+12\lambda\delta_{13}+12\lambda\delta_{12}-12\lambda\delta_{23}+24\lambda\delta_1.\]
\end{lemma}
\begin{proof}
Recall $\widetilde\M_{1,3}^\Phi$ is the closure of the locus in $\M_{1,4}^\Phi$ where the fourth marked point is the sum of the second and third marked points, and it has codimension one in $\bar\M_{1,4}^\Phi$. Let $W'\subseteq \bar\M_{0,6}$ be the pullback of this locus under 
\[\bar\M_{0,6}\to [\bar\M_{0,6}/\mu_2]\xrightarrow{\xi_{\Phi}}\bar\M^\Phi_{1,4}.\] 
Because $\widetilde\M_{1,3}^\Phi\subseteq \widetilde{\M}_{1,3}$, it is invariant under $(14)(23)$. Then, we have that elements of $W$ are invariant under either $(14)(23)$ or $(14)(23)(56)$. An element of $\M_{0,6}$ may be written as $(\bP^1, \infty, x,y,z,1,-1)$. One can check that the locus of $\M_{0,6}$ invariant under $(14)(23)$ is given by the vanishing of $yz+xz-xy-1,yz+xy-xz-1$ (which can be reduced to the vanishing of $x,yz-1$). One can also check that the locus of $\M_{0,6}$ invariant under $(14)(23)(56)$ is given by the vanishing of $xy+xz-yz-1$. The former locus has codimension $2$ and the latter has codimension $1$, hence the latter must be contained in $W'$. Let $W\subseteq W'$ be given as the closure of the vanishing of $xy+xz-yz-1$ inside $\bar\M_{0,6}$, so that $W\to \widetilde\M_{1,3}^\Phi$ is finite and surjective of degree $2$. Let $\widehat{W}$ be the image of $W$ in $[\M_{0,6}/\mu_2]$. Then $[\widehat{W}]$ pushes forward to $[\widetilde\M_{1,3}^\Phi]$ under 
\[\xi_\Phi: [\bar\M_{0,6}/\mu_2]\to \bar\M_{1,4}^\Phi.\]

For ease of notation, define 
\[\widehat{D}_{ij}:=\widehat{D}(ij|k\ell m n),\]
\[\widehat{D}_{ijk}:=\widehat{D}(ijk|\ell m n),\]
\[{D}_{ij}:={D}(ij|k\ell m n),\]
and
\[{D}_{ijk}:={D}(ijk|\ell m n),\]
for $\{i,j,k,\ell,m,n\}=[6]$.  One can use the relations in Theorem~\ref{WDVV mod 2} to see that the elements 
\[\widehat{D}_{ab}, \widehat{D}_{14a},\{\widehat{D}_{ij}|i\neq j \in [4]\},\{\widehat{D}_{iab}|i\in [4]\}\]
generate $\Pic([\bar\M_{0,6}/\mu_2])/\langle \alpha \rangle$. Moreover, using Theorem~\ref{WDVV mod 2}, this group has a presentation with $18$ generators and $6$ relations, so these $12$ generators must generate freely. So we can write
\[[\widehat{W}]=c_\alpha \alpha+c_{ab}\widehat{D}_{ab}+c_{14a}\widehat{D}_{14a}+\sum_{1\leq i<j\leq 4} c_{ij}\widehat{D}_{ij}+\sum_{1\leq i \leq 4} c_{iab}\widehat{D}_{iab},\]
for some coefficients in $\Z$. 
The permutations $(14)(23),(12)(34),$ and $(14)$ map $\widehat{W}$ to itself. Applying these symmetries to the above sum, we get all $c_{iab}$ are equal, $c_{14}=c_{23}$, and $c_{12}=c_{13}=c_{24}=c_{34}$. 

Pulling our expression for $\widehat{W}$ back to $\bar\M_{0,6}$ to get an expression for $[W]$, we have
\[[W]=c_{ab}{D}_{56}+c_{14a}({D}_{145}+D_{146})+\sum_{1\leq i<j\leq 4} c_{ij}{D}_{ij}+c_{1ab}\sum_{1\leq i \leq 4} {D}_{i56},\]
where we know $\alpha$ must pull back to $0$ because it is $2$-torsion. For $D$ an irreducible boundary divisor, the degree of the map $D\subseteq \bar\M_{0,6}\to \bar\M_{0,5}$ forgetting the $i$-th marked point is $0$ unless $D=D_{ij}$ for some $j\in [6]$. We have that the degree of $W\subseteq \bar\M_{0,6}\to \bar\M_{0,5}$ forgetting the fourth marked point is $1$, by the fact that $W$ is the closure of $V(yz+xz-xy-1)$. By symmetry, this is also the degree forgetting the fifth marked point. Pushing forward by forgetting the fourth marked point, we get $1=2c_{12}+c_{14}$; pushing forward by forgetting the fifth marked point, we get $c_{ab}=1$. Finally, we have that, at least set-theoretically, $W\cap D_{16}=D_{16}\cap D_{45}$, so pulling back our expression along $\bar\M_{0,5}\to D_{16}\hookrightarrow \bar\M_{0,6}$ gives
\[mD_{45}=c_{14a}D_{14}+c_{12}(D_{24}+D_{34})+c_{14}D_{23}+c_{1ab}D_{15},\]
where $m$ is some multiplicity. The linear dependencies between these divisors are given by multiples of 
\[D_{45}=D_{23}+D_{14}-D_{24}-D_{34}-D_{15},\]
using Theorem~\ref{WDVV mod 2}, hence 
\[c_{14}=-c_{12}=c_{14a}=-c_{1ab}.\] 
Combined with using that  $1=2c_{12}+c_{14}$, we get that these numbers are equal to $-1$, so we have
\[[\widehat{W}]=c_\alpha \alpha+\widehat{D}_{ab}-\widehat{D}_{14a}+\widehat D_{12}+\widehat D_{13}+\widehat D_{24}+\widehat D_{34}-\widehat D_{14}-\widehat D_{23}+\sum_{1\leq i \leq 4}\widehat{D}_{iab}.\]
By adding the expression for $\alpha$ in Theorem~\ref{WDVV mod 2}, with $p=2,q=3$ and $r=4$, we get
\[[\widehat{W}]=(c_\alpha+1)\alpha+\widehat{D}_{24}+\widehat{D}_{34}-\widehat{D}_{14}+2\widehat{D}_{1ab}+\widehat{D}_{1a}.\]

To determine whether $c_\alpha$ is $0$ or $1$, we consider pulling back along 
\[\xi_\Gamma: [\bar\M_{0,5}/\mu_2]\to\widehat{D}_{12}\hookrightarrow [\bar\M_{0,6}/\mu_2],\]
where $\Gamma$ is the graph corresponding to $D_{12}\subseteq \bar\M_{0,6}$. 
The above gives
\[\xi_\Gamma^*([\widehat{W}])= (c_\alpha+1)\xi_\Gamma^*(\alpha)+\xi^*_\Gamma(\widehat{D}_{34}).\]
The pullback of $\alpha$ gives $\alpha\in \CH^1([\bar\M_{0,5}/\mu_2])$, and we also have 
\[\xi_\Gamma^{-1}(\widehat{W})=\xi_{\Gamma}^{-1}(\widehat{D}_{34})=\widehat{D}_{34}\]
set-theoretically. Thus, we get
\[m\widehat{D}_{34}=(c_\alpha-1)\alpha+m'\widehat{D}_{34}\]
for some multiplicities $m,m'\in \mathbb{N}$. Because $\widehat{D}_{34}$ is non-torsion, we have $m=m'$, and hence $c_\alpha\equiv 1$ mod $2$. Thus, 
\[[\widehat{W}]=\widehat{D}_{24}+\widehat{D}_{34}-\widehat{D}_{14}+2\widehat{D}_{1ab}+\widehat{D}_{1a}.\]
Pushing forward this expression for $\widehat{W}$ along $[\bar\M_{0,6}/\mu_2]\to \bar\M_{1,4}^\Phi\subseteq \partial\bar\M_{1,4}$, we get the expression for $[\widetilde\M_{1,3}^\Phi]$ in the statement.
\end{proof}

Putting the above together we get
\[\partial_1([\M_{1,3}^0,f_0])=a_6\theta_{14:23}-2\theta_1-24\lambda\delta_1-12\lambda\delta_{12}-12\lambda\delta_{13}+12\lambda\delta_{14}+24\lambda\delta_{23}+12\hat{\delta}.\]
To finish the proof of Theorem~\ref{Getzler}, we just need to compute $a_6$.

\begin{lemma}\label{last}
    The coefficient $a_6$ is equal to $4$.
\end{lemma}
\begin{proof}
We have that the element
\[a_6\theta_{14:23}-2\theta_1-24\lambda\delta_1-12\lambda\delta_{12}-12\lambda\delta_{13}+12\lambda\delta_{14}+24\lambda\delta_{23}+12\hat{\delta}\]
is in $\im(\partial_1)$, which is an $S_4$-invariant subgroup of $\CH^2(\partial\bar\M_{1,4})$. Thus, we also have that
\[a_6\theta_{13}-2\theta_1-24\lambda\delta_1-12\lambda\delta_{12}-12\lambda\delta_{14}+12\lambda\delta_{13}+24\lambda\delta_{24}+12\hat{\delta}\]
is in $\im(\partial_1)$, and hence their difference
\[a_6(\theta_{14:23}-\theta_{13})+24\lambda(\delta_{14}-\delta_{13})+24\lambda(\delta_{23}-\delta_{24})\]
is also in $\im(\partial_1)$. Note that this element must be torsion, because the group $\overline\CH^2(\M_{1,4},1)$ has rank $1$ by Theorem~\ref{M14}. A Smith normal form calculation on the relations of $\CH^2(\partial\bar\M_{1,4})$ given in Proposition~\ref{boundarybarM14} gives that the only value of $a_6$ making this a torsion relation is $a_6=4$.
\end{proof}
\begin{rmk}
There is a more straightforward way to show that $\bar\M^{\Theta_{14:23}}\subseteq \widetilde\M_{1,3}$ and that $a_6=4$, but it takes more work. We outline the approach here.

Consider the family of three pointed curves given by
\[\bA^1\dashrightarrow \M_{1,3}^0=[U_3/\Gm]\]
\[t\mapsto (1+ct,t,1+dt,t)\]
for $c,d\in k$. This induces a morphism
\[\Spec(k[t]_{(t)})\to \bar\M_{1,3},\]
where the closed point maps to a curve in $\M_{1,3}^{\Theta_1}$. By taking the composition $\Spec k(t)\to \M_{1,3}^0\hookrightarrow \M_{1,4}$ and extending to a morphism 
\[\Spec(k[t]_{(t)})\to \bar\M_{1,4},\]
we see that the limit has to be in $\M^{\Theta_{14:23}}$. This is because forgetting the fourth marked point gives a curve in $\M^{\Theta_1}$ and the limit has to be fixed by $(14)(23)$. By varying $c,d\in k$, one can show that the limit can be any point in $\M^{\Theta_{1,4}}$, hence $\bar\M^{\Theta_{1,4}}\subseteq \widetilde\M_{1,3}$.

One can show that these curves intersect $\bar\M^{\Theta_{14:23}}$ transversely. By Proposition~\ref{transverse}, we have
\[a_6=\ord_{\overline\M^{\Theta_{14:23}}}(f_0)=\ord_ 0\left(\frac{(x_2-x_3)^6}{\Delta}|_{(1+ct,t,1+dt,t)}\right),\]
which one can compute to be $4$.
\end{rmk}

\section{Appendix: First Higher Chow Classes}
In this appendix, we give a treatment of the first higher Chow classes $[V,f]$, whose properties we outlined in Proposition~\ref{formalsum}.

Let us recall the definition of higher Chow groups. Define
\[\Delta^j:=\Spec \left(\frac{k[t_0,\dots,t_j]}{(t_0+\dots+t_j-1)}\right).\]
For an injection $[k]\hookrightarrow [j]$, we get a face map $\Delta^k\hookrightarrow \Delta^j$ in the usual way.  These are closed embeddings.
Define the subgroups $z(X,j)\subseteq Z(X\times \Delta^j)$ to be generated by the irreducible varieties $W\subseteq X\times \Delta^j$ that have proper intersections with all faces $X\times \Delta^k$. One can then form a complex
\[z(X,\cdot):=\left(\dots\xrightarrow{d_{j+1}} z(X,j)\xrightarrow{d_j} \dots \xrightarrow{d_2} z(X,1)\xrightarrow{d_1} z(X,0)\to 0\right)\]
by taking $d_j$ to be the alternating sum of the maps intersecting with the codimension $1$ faces. The higher Chow group $\CH(X,j)$ is the $j$-th cohomology group of this complex. We have a subgroup $z_i(X,j)\subseteq z(X,j)$ generated by $W\subseteq X\times \Delta^j$ that have dimension $i+j$. These $z_i(X,j)$ form a complex, $z_i(X,\cdot)$, with the same differentials, and we have
\[z(X,\cdot)=\bigoplus_i z_i(X,\cdot).\] The higher Chow group $\CH_i(X,j)$ is the $j$-th cohomology of $z_i(X,\cdot)$, and so
\[\CH(X,j)=\bigoplus_i \CH_i(X,j).\]
If $X$ is equidimensional of dimension $n$, we can instead grade by codimension, setting 
\[z^i(X,j):=z_{n-i}(X,j),\]
and
\[\CH^i(X,j):=\CH_{n-i}(X,j).\]
\begin{definition}
Given an irreducible variety $V$ and a rational function $f\in K(V)^\times$, write $[V,f]$ for the class of $\divisor(t_0+ft_1)\in Z^1(V\times \Delta^1)$. This class intersects the faces properly, so it is actually an element of $z^1(V,1)\subseteq Z^1(V\times \Delta^1)$. We also write $[V,f]$ for its pushforward to $z(X,1)$ for $X$ containing $V$ as a closed subvariety, or for its class in $\CH(X,1)$ if we have $d_1([V,f])=0$. 

If $V\subseteq X$ is a closed subscheme, and $f$ is a rational function on $V,$ set
\[[V,f]:=\sum_i m_i [V_i,f|_{V_i}],\]
where the $V_i$ are the irreducible components of $V$, with multiplicity $m_i$. 
\end{definition}

Now, Proposition~\ref{formalsum}(1) follows from the definition. Proposition~\ref{formalsum}(2) is implied by the following.
 \begin{lemma}\label{pullbackFormal}
For a flat map $\pi:X\to Y$ of constant relative dimension and $[V,f]\in z(Y,1)$, we have 
\[\pi^*([V,f])=[\pi^{-1}(V),\pi^*(f)].\]
\end{lemma}
\begin{proof}
We may assume $V$ is irreducible. Let $\{W_i\}$ be the irreducible components of $\pi^{-1}(V)$, and let $m_i$ be the multiplicity of $W_i$. By \cite[\href{https://stacks.math.columbia.edu/tag/0EPH}{Tag 0EPH}]{stacks-project}, we have
\begin{align*}
    \pi^*([V,f])&=(\pi\times \id)^*(\divisor(t_0+ft_1))\\
    &=\sum_i m_i \divisor(t_0+\pi^*(f)t_1)\\
    &=[\pi^{-1}(V),\pi^*(f)].\qedhere
\end{align*}
\end{proof}

\begin{prop}\label{div}
    Suppose $V$ is a irreducible variety over $k$, and $f\in K(V)^\times$. Then $d_1([V,f])=\divisor(f)$.
\end{prop}
\begin{proof}
Near any point, we can write $f=\frac{g}{h}$ for $g,h$ defined near this point. Then, we can write 
\begin{align*}
d_1(\divisor(t_0+ft_1))&=\sigma^*_{(0,1)}(\divisor(t_0+ft_1))-\sigma^*_{(1,0)}(\divisor(t_0+ft_1))\\
&=\sigma^*_{(0,1)}\divisor(ht_0+gt_1))-\sigma^*_{(1,0)}(\divisor(ht_0+gt_1)),    
\end{align*}
where $\sigma_p: V\to V\times \Delta_1$ is the inclusion corresponding to the point $p\in \Delta_1$, and $\sigma_p^*$ sends a prime divisor in $z^1(V,1)$ to its scheme-theoretic preimage in $V$. This pullback is equal to the intersection product defined in \cite[Chapter 2]{Fulton}. Using \cite[Theorem 2.4]{Fulton}, we have
\[\sigma^*_{(0,1)}(\divisor(ht_0+gt_1))=\sigma^*_{\divisor(ht_0+gt_1)}(V(t_0))=\divisor(h)\]
and similarly  $\sigma^*_{(1,0)}(\divisor(ht_0+gt_1))=\divisor(g)$, so
\[d_1(\divisor(t_0+ft_1))=\divisor(g)-\divisor(h)=\divisor(f).\qedhere\]
\end{proof}
\begin{cor}[Proposition~\ref{formalsum}(4)]
    If $X$ is irreducible, $U\subseteq X$ is open with complement $Z$, and $f\in K(U)$ with $\divisor_U(f)=0$, then the map
\[\partial_1: \CH^1(U,1)\to \CH(Z)\]
sends $[U,f]$ to $\divisor_X(f)$.
\end{cor}
\begin{proof}
    Note $d_1([U,f])=\divisor(f)=0$ because $f\in \OO_U(U)^\times$, so $[U,f]\in \ker(d_1)$ and thus gives a well-defined class of $\CH^1(U,1)$.

    Now, the map $\partial$ is the connecting homomorphism associated to the exact sequence 
    \[0\to z(Z,*)\to z(X,*)\to z(U,*)\]
    of complexes. So, to compute $\partial_1([U,f])$, we first lift $[U,f]$ to something in $z(X,1)$. We can choose $[X,f]$ as such a lift, viewing $f$ now as a rational function on $X$. By Proposition~\ref{div}, we have $d_1([X,f])=\divisor_X(f)$, as claimed.
\end{proof}

Now, all of Proposition~\ref{formalsum} that is left to prove is (3).

\begin{lemma}
    Suppose $V$ is an irreducible variety over $k$. For $W\subseteq V$ a prime divisor, we have $[W\times \Delta_1]$ is in $\im(d_2)\subseteq z^1(V,1)$. 
\end{lemma}
\begin{proof}
The intersection of $W\times \Delta_1$ with $t_i=0$ is $W$, so the intersection is proper. Similarly, $[W\times \Delta_2]\in z^1(V,2)$, and note
\[d_2([W\times \Delta_2])=[W\times \Delta_1]-[W\times \Delta_1]+[W\times \Delta_1]=[W\times \Delta_1].\qedhere\] 
\end{proof}

\begin{prop}\label{final}
    Suppose $V$ is an irreducible variety over $k$. The map
    \[K(V)^\times \to z^1(V,1)/\im(d_2)\]
    \[f\mapsto [V,f]\]
    is a group isomorphism.
\end{prop}
By restricting this isomorphism to the kernel of $d_1$ and using Proposition~\ref{div}, we get
\begin{cor}
    Suppose $V$ is an irreducible variety over $k$. Then
    \[\{f\in K(V)^\times| \divisor(f)=0\}\to \CH^1(V,1)\]
    \[f\mapsto [V,f]\]
    is an isomorphism. 
\end{cor}
This gives Proposition~\ref{formalsum}(3), using the algebraic Hartogs' Lemma.

\begin{proof}[Proof of Proposition~\ref{final}]
    Consider the homomorphism
    \[H: z^1(V,1)/\im(d_2)\to z^1(K(V),1)/\im(d_2).\]
    The kernel of $H$ is generated by $[S\times \Delta^1]$ for codimension $1$ closed subvarieties $S\subseteq V$. By the lemma, these classes are $0$, so this homomorphism is injective. Moreover, given a generator of $z^1(K(V),1)$, $W\subseteq \Spec(K(V))\times \Delta^1$, we can consider the closure $\overline W\subseteq V\times \Delta^1$. Note $\overline W$ is irreducible and has codimension $1$ in $V\times \Delta^1$. If $\overline W$ did not intersect a face of $\Delta^1$ properly, it would need to be equal to said face by considering dimension and irreducibility. Then $W$ would also be equal to the corresponding face, contradicting $W\in z^1(K(V),1)$. As $\overline W \cap \Spec(K(V))\times \Delta^1=W$, we have $H(\overline W)=W$, and so $H$ is surjective as well. By Lemma~\ref{pullbackFormal}
    \[H([V,f])=[K(V),f],\]
    and so it suffices to prove the proposition for fields $k$. 

    Because $\Delta_n\cong \bA^n$, we have an isomorphism
    \[K(\Delta_n)^\times/k^\times \to Z^1(\Delta_n)\]
    \[g(t_0,\dots,t_n)\mapsto \divisor(g(t_0,\dots,t_n)).\]
    Moreover, we have that $\divisor(g(t_0,\dots,t_n))$ is in $z^1(k,n)\subseteq Z^1(\Delta_n)$ if and only if the numerator and denominator are not ideals of the form \[(t_0,\dots,\widehat{t}_j,\dots, t_n)\] of the coordinate ring of $\Delta^n$. 
     
    Now, we define a $z^1(k,1)\to k^\times$ which will induce an inverse to $f\mapsto [k,f]$. We define
    \[\Psi: z^1(k,1)\to k^\times\]
    \[\divisor(g(t_0,t_1)) \mapsto \frac{g(1,0)}{g(0,1)}.\]
    Note that this is well defined, because $f$ is well defined up to a unit and $t_0,t_1$ are not factors of the numerator or denominator of $f$, and is a group homomorphism.
    
    We want to show $\Psi$ descends to a homomorphism on $z^1(k,1)/\im(d_2)$. For $\divisor(h(t_0,t_1,t_2))\in z^1(k,2)$, we have
    \[d_2(\divisor(h(t_0,t_1,t_2)))=\divisor(h(t_0,t_1,0))-\divisor(h(t_0,0,t_1))-\divisor(h(0,t_0,t_1)).\]
    Then
    \[\Psi(d_2(U))=\left(\frac{h(1,0,0)}{h(0,1,0)}\right)\left(\frac{h(1,0,0)}{h(0,0,1)}\right)^{-1}\left(\frac{h(0,1,0)}{h(0,0,1)}\right)=1.\]
    Thus $\Psi$ induces a homomorphism 
    \[\bar \Psi: z^1(k,1)/\im(d_2)\to k^\times.\]
    Note 
    \[\bar\Psi([k,f])=\Psi([k,f])=\Psi(\divisor(t_0+ft_1))=\frac{f}{1}=f,\]
    so $\bar\Psi$ is a right inverse of $f\mapsto [k,f]$ and is therefore surjective.

    Finally, we show $\bar\Psi$ is injective. Supposing $1=\Psi(\divisor g(t_0,t_1))$, we want to show that $\divisor(g(t_0,t_1))\in \im(d_2)$. Note
    \[1=\Psi(\divisor g(t_0,t_1))=\frac{g(1,0)}{g(0,1)},\]
    hence $g(1,0)=g(0,1)$. Define 
    \[h(t_0,t_1,t_2):=g(0,1)+t_1\frac{g(t_0,t_1+t_2)-g(0,1)}{t_1+t_2}.\]
    Writing  $g(t_0,t_1)=\frac{g_n(t_0,t_1)}{g_d(t_0,t_1)}$ for $g_n,g_d\in \frac{k[t_0,t_1]}{(t_0+t_1-1)},$ we have that 
    \[\ell(t_0,t_1,t_2):=\frac{g_n(t_0,t_1+t_2)-g(0,1)g_d(t_0,t_1+t_2)}{t_1+t_2}\]
    is a polynomial, because $g(1,0)=g(0,1)$. Then, we can write
    \[h(t_0,t_1,t_2)=\frac{g(0,1)g_d(t_0,t_1+t_2)+t_1\ell(t_0,t_1,t_2)}{g_d(t_0,t_1+t_2)},\]
    from which we can see that $\divisor( h(t_0,t_1,t_2))\in z^1(k,2)$ by computing the restrictions $t_i=0$. Now, we have  
    \[h(t_0,t_1,0)=g(t_0,t_1)\]
    and
    \[h(t_0,0,t_1)=h(0,t_0,t_1)=g(0,1),\] 
    so 
    \begin{align*}
        d_2(\divisor( h(t_0,t_1,t_2)))&=\divisor (g(t_0,t_1))-\divisor( g(0,1))+\divisor (g(0,1))\\
        &=\divisor (g(t_0,t_1))
    \end{align*}
    as needed.

    Thus, $\bar\Psi$ is a group isomorphism, and a right inverse of $f\mapsto [k,f]$, so this function must also be a group isomorphism. 
\end{proof}
\begin{rmk}
    The proof only indirectly shows that $f\mapsto [V,f]$ is a group homomorphism. One can directly see that $[V,fg]=[V,f]+[V,g]$ modulo the image of $d_2$ by computing $d_2(\divisor(t_0+ft_1+fgt_2))$.
\end{rmk}

\bibliographystyle{amsalpha}
\bibliography{refs.bib}

\end{document}